\documentclass[a4paper]{article}
\usepackage{geometry}
\usepackage{amssymb}
\usepackage{latexsym}
\usepackage{amsmath}
\usepackage{indentfirst}
\usepackage{graphicx}
\usepackage[colorlinks=true]{hyperref}
\usepackage{mathrsfs}
\usepackage{chngcntr}
\usepackage{color}

\newtheorem{Thm}{Theorem}[section]

\newtheorem{Lem}{Lemma}[section]
\newtheorem{Pro}{Proposition}[section]

\newtheorem{Def}{Definition}[section]

\newcommand{\R}{\mathbb{R}}

\numberwithin{equation}{section}
\newcommand{\qed}{\hfill $\Box$ \smallskip}
\newenvironment{proof}{\medskip\par\noindent{\bf Proof\/}:\quad}{\qed\par\medskip}
\begin{document}
	\title{Local uniqueness of multi-peak positive solutions to a class of fractional Kirchhoff equations}
	
	\author{Zhipeng Yang\thanks{Email: yangzhipeng326@163.com} \\[2pt]
		\small Department of Mathematics, Yunnan Normal University, Kunming, China\\
		\small Yunnan Key Laboratory of Modern Analytical Mathematics and Applications, Kunming, China
	}
	\date{}
	\maketitle
	
	\begin{abstract}
		This paper has two main purposes. In the first part, combining the nondegeneracy of the ground state with the Lyapunov--Schmidt reduction method, we prove the existence of multi-peak positive solutions to the singularly perturbed problem
		\[
		\Big(\varepsilon^{2s}a+\varepsilon^{4s-N} b\int_{\mathbb{R}^{N}}|(-\Delta)^{\frac{s}{2}}u|^2\,dx\Big)(-\Delta)^s u+V(x)u=u^p\quad \text{in }\mathbb{R}^{N},
		\]
		for all sufficiently small $\varepsilon> 0$, under the assumptions $2s<N<4s$, $1<p<2^*_s-1$, and some mild conditions on the potential $V$. The main difficulty comes from the interplay between the nonlocal operator $(-\Delta)^s$ and the nonlocal Kirchhoff term, which makes the corresponding limiting problem a coupled system of partial differential equations rather than a single fractional Kirchhoff equation. In the second part, under additional assumptions on $V$, we establish the local uniqueness of positive multi-peak solutions by means of a local Pohoz\v{a}ev identity.
		
				
	\end{abstract}


\section{Introduction and main results}

In this paper, we are concerned with the following singularly perturbed fractional Kirchhoff problem
\begin{equation}\label{eq1.1}
	\Big(\varepsilon^{2s}a+\varepsilon^{4s-N} b\int_{\mathbb{R}^{N}}|(-\Delta)^{\frac{s}{2}}u|^2\,dx\Big)(-\Delta)^s u+V(x)u=u^p,\quad \text{in }\mathbb{R}^{N},
\end{equation}
where $a,b>0$, $\varepsilon>0$ is a parameter, $V: \mathbb{R}^{N} \rightarrow \mathbb{R}$ is a bounded continuous function, $(-\Delta )^s$ is the fractional Laplacian and $p$ satisfies
\begin{equation*}
	1<p<2_s^*-1=
	\begin{cases}
		\displaystyle\frac{N+2s}{N-2s}, &  0<s<\frac{N}{2}, \\
		+\infty, &  s\geq \frac{N}{2},
	\end{cases}
\end{equation*}
where $2_s^*$ is the standard fractional Sobolev critical exponent. Throughout this paper we assume
\[
0<s<1 \quad\text{and}\quad 2s<N<4s.
\]

Problem \eqref{eq1.1} and its variants have been studied extensively in the literature. The equation that goes under the name of Kirchhoff equation was proposed in \cite{Kirchhoff1883} as a model for the transverse oscillation of a stretched string in the form
\begin{equation}\label{eq1.2}
	\rho h \partial_{tt}^{2} u-\left(p_{0}+\frac{\mathcal{E}h}{2 L} \int_{0}^{L}\left|\partial_{x} u\right|^{2} d x\right) \partial_{xx}^{2} u=0,
\end{equation}
for $t \geq 0$ and $0<x<L$, where $u=u(t, x)$ is the lateral displacement at time $t$ and at position $x$, $\mathcal{E}$ is the Young modulus, $\rho$ is the mass density, $h$ is the cross section area, $L$ the length of the string, and $p_{0}$ is the initial stress tension.

Through the years, this model was generalized in several ways that can be collected in the form
\begin{equation*}
	\partial_{tt}^{2} u-M\big(\|u\|^{2}\big) \Delta u=f(t, x, u), \quad x \in \Omega,
\end{equation*}
for a suitable function $M:[0, \infty) \rightarrow \mathbb{R}$, called Kirchhoff function. Here $\Omega$ is a bounded domain of $\mathbb{R}^{N}$, and $\|u\|^{2}=\|(-\Delta)^{\frac{s}{2}} u\|_{2}^{2}$ denotes the nonlocal Dirichlet norm of $u$. The basic case corresponds to the choice
\begin{equation*}
	M(t)=a+b t^{\gamma-1}, \quad a \geq 0,\ b \geq 0,\ \gamma \geq 1.
\end{equation*}
Problem \eqref{eq1.2} and its variants have been studied extensively in the literature. Bernstein obtained a global stability result in \cite{Bernstein1940BASUS}, which was generalized to arbitrary dimension $N\geq 1$ by Pohoz\v{a}ev in \cite{Pohozaev1975MS}. We also point out that such problems may describe processes of some biological systems depending on the average of themselves, such as the density of population (see e.g. \cite{Arosio-Panizzi1996TAMS}). From a mathematical point of view, the interest in studying Kirchhoff equations comes from the nonlocality of Kirchhoff type equations. For instance, if we take $M(t)=a+bt$, the consideration of the stationary analogue of Kirchhoff’s wave equation leads to elliptic problems
\begin{equation}\label{eq1.3}
	\begin{cases}
		\Big(a+b\displaystyle\int_{\Omega}|(-\Delta)^{\frac{s}{2}} u|^2\,dx\Big)(-\Delta)^s u=f(x,u) &\text{in } \Omega, \\
		u=0 &\text{on } \partial \Omega,
	\end{cases}
\end{equation}
for some nonlinear functions $f(x,u)$. Note that the term $\big(\int_{\Omega}|(-\Delta)^{\frac{s}{2}} u|^2\,dx\big)(-\Delta)^s u$
depends not only on the nonlocal operator $(-\Delta)^s u$, but also on the integral of $|(-\Delta)^{\frac{s}{2}} u|^2$ over the whole
domain. In this sense, \eqref{eq1.1} and \eqref{eq1.3} are no longer usual pointwise equalities. This
new feature brings new mathematical difficulties that make the study of Kirchhoff type equations
particularly interesting.

When $s=1$ and $N=3$, \eqref{eq1.1} reduces to the following equation
\begin{equation}\label{eq1.4}
	\begin{cases}
		-\bigg(\varepsilon^2a+ \varepsilon b\displaystyle\int_{\mathbb{R}^3}|\nabla u|^2\,dx\bigg)\Delta u+ V(x)u=f(u), & x\in\mathbb{R}^3, \\
		u\in H^1(\mathbb{R}^3),&
	\end{cases}
\end{equation}
and the existence and multiplicity of solutions to \eqref{eq1.4} with $\varepsilon=1$ have been studied in some recent works. Li and Ye \cite{Li-Ye2014JDE} obtained the existence of a positive ground state of \eqref{eq1.4} with $f(u)=|u|^{p-1}u$ for $2<p<5$. In \cite{Deng-Peng-Shuai2015JFA}, Deng, Peng and Shuai studied the existence and asymptotic behavior of nodal solutions of \eqref{eq1.4} with $V$ and $f$ radially symmetric in $x$ as $b\to 0^+$. There are also some works concerning the concentration behavior of solutions as $\varepsilon\to 0^+$. It seems that He and Zou \cite{He-Zou2012JDE} were the first to study singularly perturbed Kirchhoff equations. In \cite{He-Zou2012JDE}, they considered problem \eqref{eq1.4} when  $V$ is assumed to satisfy the global condition of Rabinowitz \cite{Rabinowitz1992ZAMP}
\begin{equation*}
	\liminf _{|x| \rightarrow \infty} V(x)>\inf _{x \in \mathbb{R}^{3}} V(x)>0
\end{equation*}
and $f: \mathbb{R} \rightarrow \mathbb{R}$ is a nonlinear function with subcritical growth of type $u^{q}$ for some $3<q<5$. They proved the existence of multiple positive solutions for $\varepsilon$ sufficiently small. A similar result for the critical case $f(u)=\lambda g(u)+|u|^4u$ was obtained separately in \cite{He-Zou2014ADM} and \cite{Wang-Tian-Xu-Zhang2012JDE}, where the subcritical term $g(u)\sim |u|^{p-2}u$ with $4<p<6$. In \cite{He-Li-Peng2014ANS}, He, Li and Peng constructed a family of positive solutions which
concentrate around a local minimum of $V$ as $\varepsilon\to 0^+$
for a critical problem $f(u)=g(u)+|u|^4u$ with $g(u)\sim |u|^{p-2}u$ ($4<p<6$). For the more delicate case $f(u)=\lambda |u|^{p-2}u+|u|^4u$ with $2<p\leq 4$ we refer to He and Li \cite{He-Li2015CVPDE}, where a family of positive solutions which concentrate around a local minimum of $V$ as $\varepsilon\to 0^+$ was obtained. We refer e.g. to \cite{Deng-Peng-Shuai2015JFA,MR3218834,MR3360660,MR4293909,MR4305432,MR4021897} for more mathematical researches on Kirchhoff type equations in the whole space. We also refer to \cite{MR3987384} for a recent survey of the results connected to this model.

Very recently, Li et al. \cite{MR4021897} proved that the positive ground state solution of \eqref{eq1.4} with $V\equiv1$ and $f(u)=|u|^{p-1}u$ ($1<p<5$) is unique and nondegenerate. Then, using the Lyapunov--Schmidt reduction method, they proved the existence and uniqueness of single peak solutions to equation \eqref{eq1.4} for all $1<p<5$. Under some mild conditions on $V$, Luo et al. \cite{MR3988638} proved the existence of multi-peak solutions to \eqref{eq1.4}. As a continuation of the work \cite{MR3988638}, Li et al. \cite{MR4056531} established a local uniqueness result for multi-peak solutions by the technique of the local Pohoz\v{a}ev identity from \cite{MR3420409}. By local uniqueness we mean that if $u_{\varepsilon}^{1}, u_{\varepsilon}^{2}$ are two $k$-peak solutions concentrating at the same $k$ points, then $u_{\varepsilon}^{1} \equiv u_{\varepsilon}^{2}$ for $\varepsilon$ sufficiently small.

On the other hand, the interest in generalizing to the fractional case the model introduced by Kirchhoff does
not arise only for mathematical purposes. In fact, following the ideas of \cite{MR2675483} and the concept
of fractional perimeter, Fiscella and Valdinoci proposed in \cite{MR3120682} an equation describing the
behaviour of a string constrained at the extrema in which the fractional length of the rope appears. Recently, problems similar to \eqref{eq1.1} have been extensively investigated by many authors using different techniques and producing several relevant results  (see, e.g. \cite{MR4071927,MR4056169,MR4305431,Gu-Yang,MR3877201,MR3961733,MR3917341,MR3993416}).
\par
Notice that although it is already known that problem \eqref{eq1.1} admits multiple single-peak solutions, it is still an open problem whether there exist multi-peak solutions to \eqref{eq1.1}, in striking contrast to the extensive results on multi-peak solutions to singularly perturbed fractional Schr\"odinger equations. This motivates us to study multi-peak solutions to problem \eqref{eq1.1}. To be precise, we recall the standard definition of a multi-peak solution of \eqref{eq1.1}.

\begin{Def}\label{Def1.1}
	Let $k \in\{1,2,\ldots\}$. We say that $u_{\varepsilon}$ is a $k$-peak solution of \eqref{eq1.1} if $u_{\varepsilon}$ satisfies
	\begin{itemize}
		\item [$(i)$] $u_{\varepsilon}$ has $k$ local maximum points $y_{\varepsilon}^{i} \in \mathbb{R}^{N}$, $i=1,2, \ldots, k$, satisfying
		\begin{equation*}
			y_{\varepsilon}^{i} \rightarrow a_{i}
		\end{equation*}
		for some $a_{i} \in \mathbb{R}^{N}$ as $\varepsilon \rightarrow 0$ for each $i$.
		\item [$(ii)$] For any given $\tau>0$, there exists $R \gg 1$ such that
		\begin{equation*}
			\left|u_{\varepsilon}(x)\right| \leqslant \tau \quad \text{for all } x \in \mathbb{R}^{N} \setminus \bigcup_{i=1}^{k} B_{R \varepsilon}\left(y_{\varepsilon}^{i}\right).
		\end{equation*}
		\item [$(iii)$] There exists a constant $C>0$, independent of $\varepsilon$, such that
		\begin{equation*}
			\int_{\mathbb{R}^{N}}\left(\varepsilon^{2s} a\left|(-\Delta)^{\frac{s}{2}} u_{\varepsilon}\right|^{2}+u_{\varepsilon}^{2}\right)\,dx \leqslant C \varepsilon^{N}.
		\end{equation*}
	\end{itemize}
\end{Def}
\par
In order to state our main results, we recall some preliminary
facts about the fractional Laplacian. For $0<s<1$, the fractional
Sobolev space $H^s(\mathbb{R}^N)$ is defined by
\begin{equation*}
	H^{s}(\mathbb{R}^N)=\bigg\{u\in L^2(\mathbb{R}^N):\frac{u(x)-u(y)}{|x-y|^{\frac{N}{2}+s}}\in L^2(\mathbb{R}^{N}\times\mathbb{R}^{N})\bigg\},
\end{equation*}
endowed with the natural norm
\begin{equation*}
	\|u\|^2=\int_{\mathbb{R}^N}|u|^2\,dx+\int_{\mathbb{R}^N}\int_{\mathbb{R}^N}\frac{|u(x)-u(y)|^2}{|x-y|^{N+2s}}\,dx\,dy.
\end{equation*}
\par
The fractional Laplacian $(-\Delta)^s$ is the pseudo-differential
operator defined by
\begin{equation*}
	\mathcal{F}\big((-\Delta)^s u\big)(\xi)=|\xi|^{2s}\mathcal{F}(u)(\xi),\quad \xi\in \mathbb{R}^N,
\end{equation*}
where $\mathcal{F}$ denotes the Fourier transform. It is also given by
\begin{equation*}
	(-\Delta)^{s}u(x)=-\frac{1}{2}C(N,s)\int_{\mathbb{R}^N}\frac{u(x+y)+u(x-y)-2u(x)}{|y|^{N+2s}}\,dy,
\end{equation*}
where
\begin{equation*}
	C(N,s)=\bigg(\int_{\mathbb{R}^N}\frac{1-\cos \xi_1}{|\xi|^{N+2s}}\,d\xi\bigg)^{-1},\quad \xi=(\xi_1,\xi_2,\ldots,\xi_N).
\end{equation*}
From \cite{Nezza-Palatucci-Valdinoci2012BSM}, we have
\begin{equation*}
	\|(-\Delta)^{\frac{s}{2}}u\|^2_2=\int_{\mathbb{R}^N}|\xi|^{2s}|\mathcal{F}(u)|^2\,d\xi=\frac{1}{2}C(N,s)
	\int_{\mathbb{R}^N}\int_{\mathbb{R}^N}\frac{|u(x)-u(y)|^2}{|x-y|^{N+2s}}\,dx\,dy.
\end{equation*}
\par
From the viewpoint of the calculus of variations, the fractional Kirchhoff problem \eqref{eq1.1} is much more complex and difficult than the
classical fractional Laplacian equation because of the presence of the term
\[
b\Big(\int_{\mathbb{R}^{N}}\big|(-\Delta)^{\frac{s}{2}}u\big|^2\,dx\Big)(-\Delta )^s u,
\]
which is (formally) of order $4s$ (and of order four when $s=1$).
This fact leads to difficulties in obtaining the boundedness of Palais--Smale sequences for the corresponding energy functional when $p\leq3$.
Recently, R\u{a}dulescu and Yang \cite{R-Yang1} established uniqueness and nondegeneracy for positive solutions to Kirchhoff equations with subcritical growth. More precisely,
they proved that the following fractional Kirchhoff equation
\begin{equation}\label{eq1.5}
	\Big(a+b\int_{\mathbb{R}^{N}}\big|(-\Delta)^{\frac{s}{2}}u\big|^2\,dx\Big)(-\Delta)^s u+u=u^p,\quad \text{in }\mathbb{R}^{N},
\end{equation}
where $a,b>0$, $\frac{N}{4}<s<1$, $1<p<2_s^*-1$, has a unique nondegenerate positive radial solution. One of the main ideas is based on a scaling technique which allows one to find a relation between solutions of \eqref{eq1.5} and the following equation
\begin{equation}\label{eq1.6}
	(-\Delta)^{s} Q+Q=Q^{p} , \quad \text{in } \mathbb{R}^{N},
\end{equation}
where $0<s<1$ and $1<p<2_{s}^{*}-1$. For high dimensions and the critical case we refer to \cite{Gu-Yang1,Yang,Yang-Yu}. We first summarize some results in \cite{R-Yang1} for convenience.

\begin{Pro}\label{Pro1.1}
	Let $a,b>0$. Assume that $\frac{N}{4}<s<1$ and $1<p<2_s^*-1$. Then
	equation \eqref{eq1.5} has a ground state solution $U\in H^s(\mathbb{R}^N)$ which is unique up to translations, and:
	\begin{itemize}
		\item[$(i)$] $U>0$ belongs to $C^\infty(\mathbb{R}^N)\cap H^{2s+1}(\mathbb{R}^N)$;
		\item[$(ii)$] there exists some $x_0\in \mathbb{R}^N$ such that
		$U(\cdot-x_0)$ is radial and strictly decreasing in $r=|x-x_0|$;
		\item[$(iii)$] there exist constants $C_1,C_2>0$ such that
		\begin{equation*}
			\frac{C_1}{1+|x|^{N+2s}}\leq U(x)\leq \frac{C_2}{1+|x|^{N+2s}},\quad \forall\, x\in \mathbb{R}^N.
		\end{equation*}
	\end{itemize}
	Moreover, $U$ is nondegenerate in $H^s(\mathbb{R}^N)$ in the sense that
	\[
	\ker L_+=\operatorname{span}\{\partial_{x_1}U, \partial_{x_2}U,\ldots,  \partial_{x_N}U\},
	\]
	where $L_+$ is defined by
	\begin{equation*}
		L_+\varphi=\Big(a+b\int_{\mathbb{R}^{N}}\big|(-\Delta
		)^{\frac{s}{2}}U\big|^2\,dx\Big)(-\Delta
		)^s\varphi+\varphi-pU^{p-1}\varphi+2b\Big(\int_{\mathbb{R}^{N}}(-\Delta
		)^{\frac{s}{2}}U\,(-\Delta )^{\frac{s}{2}}\varphi\,dx\Big)(-\Delta )^sU
	\end{equation*}
	acting on $L^2(\mathbb{R}^N)$ with domain $H^s(\mathbb{R}^N)$.
\end{Pro}

By Proposition \ref{Pro1.1}, it is now possible to apply the Lyapunov--Schmidt
reduction to study the perturbed fractional Kirchhoff equation \eqref{eq1.1}.
We look for multi-peak positive solutions of \eqref{eq1.1} in the Sobolev space
$H^s(\mathbb{R}^N)$ for sufficiently small $\varepsilon>0$, which are usually
referred to as semiclassical solutions. We also refer to such solutions as
concentrating solutions, since they concentrate at certain points of the
potential $V$ as $\varepsilon\to0$.

Assume that $V: \mathbb{R}^{N} \rightarrow \mathbb{R}$ satisfies the following
conditions:
\begin{itemize}
	\item[$(V_1)$] $V \in L^{\infty}\left(\mathbb{R}^{N}\right)$ and
	\[
	0<\inf_{\mathbb{R}^{N}} V \leqslant \sup_{\mathbb{R}^{N}} V<\infty;
	\]
	
	\item[$(V_2)$] There exist $k\geqslant 2$ distinct points
	$\left\{a_{1}, \ldots, a_{k}\right\} \subset \mathbb{R}^{N}$ such that for every
	$1 \leqslant i \leqslant k$ one has $V \in C^{\alpha}\!\left(\overline{B_{r_{0}}(a_i)}\right)$
	for some $\alpha \in\bigl(0,\tfrac{N+4s}{2}\bigr)$, and there exists $r>0$ with
	$0<r<r_{0}:=\frac{1}{2}\min_{1 \leqslant i \neq j \leqslant k}\left|a_{i}-a_{j}\right|$ such that
	\[
	V\left(a_{i}\right)<V(x) \quad \text{for all } x \text{ satisfying }0<\left|x-a_{i}\right|<r;
	\]
	
	\item[$(V_3)$] There exist $m\in\bigl(1,\tfrac{N+4s}{2}\bigr)$, $\delta>0$,
	points $a_{i}=\left(a_{i,1}, a_{i,2},\ldots,a_{i,N}\right) \in \mathbb{R}^{N}$ and
	coefficients $c_{i,j} \in \mathbb{R}$ with $c_{i,j} \neq 0$ for each
	$i=1,\ldots,k$ and $j=1,\ldots,N$ such that $V \in C^{1}\bigl(B_{\delta}(a_{i})\bigr)$ and
	\begin{equation*}
		\begin{cases}
			\displaystyle
			V(x)=V\left(a_i\right)+\sum\limits_{j=1}^{N} c_{i,j}\left|x_{j}-a_{i,j}\right|^{m}
			+O\bigl(\left|x-a_i\right|^{m+1}\bigr),
			& x \in B_{\delta}\left(a_i\right), \\[6pt]
			\displaystyle
			\frac{\partial V}{\partial x_{j}}(x)
			=m c_{i,j}\left|x_{j}-a_{i,j}\right|^{m-2}\left(x_{j}-a_{i,j}\right)
			+O\bigl(\left|x-a_{i}\right|^{m}\bigr),
			& x \in B_{\delta}\left(a_{i}\right),
		\end{cases}
	\end{equation*}
	for $i=1,\ldots,k$ and $j=1,\ldots,N$, where
	$x=\left(x_{1}, x_{2},\ldots, x_{N}\right) \in \mathbb{R}^{N}$.
\end{itemize}

The assumption $(V_1)$ allows us to introduce the inner product
\begin{equation*}
	\langle u, v\rangle_{\varepsilon}
	=\int_{\mathbb{R}^{N}}\left(\varepsilon^{2s} a\, (-\Delta)^{\frac{s}{2}} u
	\cdot(-\Delta)^{\frac{s}{2}} v+V(x)\, u v\right)\,dx
\end{equation*}
for $u, v \in H^{s}\left(\mathbb{R}^{N}\right)$.
We also write
\begin{equation*}
	H_{\varepsilon}
	=\left\{u \in H^{s}\left(\mathbb{R}^{N}\right):\|u\|_{\varepsilon}
	=\langle u, u\rangle_{\varepsilon}^{\frac{1}{2}}<\infty\right\},
\end{equation*}
so that $H_{\varepsilon}=H^{s}(\mathbb{R}^N)$ with equivalent norms.

Now we state our main results as follows.

\begin{Thm}\label{Thm1.1}
	Let $a,b>0$. Assume that $0<s<1$, $2s<N<4s$, $1<p<2_s^*-1$ and that $V$
	satisfies $(V_1)$–$(V_2)$. Then there exists $\varepsilon_{0}>0$ such that for all
	$\varepsilon \in\left(0, \varepsilon_{0}\right)$, problem \eqref{eq1.1} admits a
	$k$-peak positive solution in the sense of Definition \ref{Def1.1}, whose
	peaks concentrate around the points $a_i$, $1\leq i\leq k$, as $\varepsilon\to0$.
\end{Thm}

\begin{Thm}\label{Thm1.2}
	Assume that $0<s<1$, $2s<N<4s$, $1<p<2_s^*-1$, and that $V$ satisfies
	$(V_1)$–$(V_3)$. If $u_{\varepsilon}^{(1)},u_{\varepsilon}^{(2)}$ are two
	$k$-peak solutions concentrating at the set of $k$ distinct points
	$\{a_1,a_2,\ldots,a_k\}$, then
	\begin{equation*}
		u_{\varepsilon}^{(1)} \equiv u_{\varepsilon}^{(2)}
	\end{equation*}
	for all $\varepsilon$ sufficiently small.
	
	Moreover, let
	\begin{equation*}
		u_{\varepsilon}(x)=\sum_{i=1}^{k}U^i\!\left(\frac{x-y^i_{\varepsilon}}{\varepsilon}\right)
		+\varphi_{\varepsilon}(x)
	\end{equation*}
	be this unique solution, where $(U^1,U^2,\ldots,U^k)$ is the positive solution of
	the system \eqref{eq2.5}. Then there exist $0<\tau<1$ sufficiently small such that
	for all $i=1,\ldots,k$ one has
	\begin{equation*}
		\left|y^i_{\varepsilon}-a_i\right|=o(\varepsilon)
		\quad\text{and}\quad
		\left\|\varphi_{\varepsilon}\right\|_{\varepsilon}
		=O\left(\varepsilon^{\frac{N}{2}+m(1-\tau)}\right).
	\end{equation*}
\end{Thm}

\par
To prove Theorem \ref{Thm1.1}, let us first recall that, to construct multi-peak solutions to the Schr\"odinger equation
\begin{equation}\label{eq1.7}
	-\varepsilon^{2} \Delta u+V(x) u=u^{q}, \quad u>0 \quad \text{in } \mathbb{R}^{N},
\end{equation}
it is very important to understand the limiting equation as $\varepsilon \rightarrow 0$, which is known as the unperturbed Schr\"odinger equation
\begin{equation*}
	-\Delta u+V(x) u=u^{q}, \quad u>0 \quad \text{in } \mathbb{R}^{N}.
\end{equation*}
Denote by $Q_{i}$ the unique (see \cite{MR969899}) positive radial solution to the equation
\[
-\Delta Q_{i}+V\left(a_{i}\right) Q_{i}=Q_{i}^{q} \quad \text{in } \mathbb{R}^{N}.
\]
Then, to construct a $k$-peak solution to equation \eqref{eq1.7} concentrating at $\left\{a_{1}, \ldots, a_{k}\right\}$, natural candidates are functions of the form
\[
u_{\varepsilon}=\sum_{i=1}^{k} Q_{i}\left(\frac{x-y^i_{\varepsilon}}{\varepsilon}\right)+\varphi_{\varepsilon},
\]
where $y^i_{\varepsilon} \rightarrow a_{i}$ and $\varphi_{\varepsilon}$ is chosen appropriately so that $u_{\varepsilon}$ is indeed a solution to equation \eqref{eq1.7}.
\par
It seems that the above idea should also work for problem \eqref{eq1.1}, with the unperturbed Kirchhoff equation \eqref{eq1.5} as the limiting equation. Indeed, to construct single-peak solutions to problem \eqref{eq1.1}, this idea works, as can be seen in R\u{a}dulescu and Yang \cite{R-Yang1}. However, for multi-peak solutions it turns out to be false: there are no multi-peak solutions of the form
\[
u_{\varepsilon}=\sum_{i=1}^{k} w^{i}\left(\frac{x-y_{\varepsilon}^{i}}{\varepsilon}\right)+\varphi_{\varepsilon},
\]
where $w^i$ is the unique positive solution to
\begin{equation*}
	\Big(a+b\int_{\mathbb{R}^{N}}\big|(-\Delta)^{\frac{s}{2}}u\big|^2\,dx\Big)(-\Delta)^s u+V(a_i)u=u^p,\quad \text{in }\mathbb{R}^{N}.
\end{equation*}
\par
To overcome this difficulty, we borrow some ideas from \cite{MR3988638} for the Kirchhoff equation. By the definition of multi-peak solutions of problem \eqref{eq1.1}, we first prove that if $u_{\varepsilon}$ is a $k$-peak solution to \eqref{eq1.1}, then $u_{\varepsilon}$ must be of a particular form, and $a_{i}$ must be critical points of $V$ if $V$ is continuously differentiable in a neighbourhood of $a_{i}$. In fact, via this step, we prove that the correct limiting equation of problem \eqref{eq1.1} is a system of partial differential equations; see Section 2 for details. This reveals a new phenomenon for multi-peak solutions of singular perturbation problems, which is quite different from the known results on singularly perturbed elliptic equations.
\par
To prove the local uniqueness, we will follow the idea of \cite{MR3426103}. More precisely, if $u_{\varepsilon}^{(1)},u_{\varepsilon}^{(2)}$ are two distinct solutions, then the function
\begin{equation*}
	\xi_{\varepsilon}=\frac{u_{\varepsilon}^{(1)}-u_{\varepsilon}^{(2)} }{\left\|u_{\varepsilon}^{(1)}-u_{\varepsilon}^{(2)}\right\|_{L^{\infty}\left(\mathbb{R}^{N}\right)}}
\end{equation*}
satisfies $\left\|\xi_{\varepsilon}\right\|_{L^{\infty}\left(\mathbb{R}^{N}\right)}=1$. We will show, by using the equation satisfied by $\xi_{\varepsilon}$, that $\left\|\xi_{\varepsilon}\right\|_{L^{\infty}\left(\mathbb{R}^{N}\right)} \rightarrow 0$ as $\varepsilon \rightarrow 0$. This gives a contradiction, and thus yields uniqueness. To derive the contradiction, we need quite delicate estimates on the asymptotic behavior of solutions and of the concentrating points $y^i_{\varepsilon}$, due to the presence of the nonlocal term $\big(\int_{\mathbb{R}^{N}}|(-\Delta)^{\frac{s}{2}} u|^{2}\,dx\big)(-\Delta)^s u$ and of the fractional operator.
\par
This paper is organized as follows. We derive the form and location of multi-peak solutions to equation \eqref{eq1.1} in Section 2. In Section 3, we present some basic results which will be used later and explain the strategy of the proof of Theorem \ref{Thm1.1}. Finally, we complete the proof of Theorem \ref{Thm1.1} in Section 4 and of Theorem \ref{Thm1.2} in Section 5.
\par
\vspace{3mm}
{\bf Notation.~}Throughout this paper, we make use of the following notations.
\begin{itemize}
	\item[$\bullet$]  For any $R>0$ and any $x\in \mathbb{R}^N$, $B_R(x)$ denotes the ball of radius $R$ centered at $x$;
	\item[$\bullet$]  $\|\cdot\|_q$ denotes the usual norm of the space $L^q(\mathbb{R}^N)$, $1\leq q\leq\infty$;
	\item[$\bullet$]  $o_n(1)$ denotes a quantity such that $o_n(1)\rightarrow 0$ as $n\rightarrow\infty$;
	\item[$\bullet$]  $C$ or $C_i$ ($i=1,2,\ldots$) denote positive constants which may change from line to line.
\end{itemize}

\section{The form and locations of multi-peak solutions}

In this section, we first fix the form of multi-peak solutions of equation \eqref{eq1.1} and locate the related concentrating points. In particular, we prove that if equation \eqref{eq1.1} has a concentrating solution, then $V$ must have at least one critical point. We will use the following inequality repeatedly.

\begin{Lem}\label{Lem2.1}
	For any $2 \leq q \leq 2^*_s$, there exists a constant $C>0$ depending only on $N,s,a,q$ and on the bounds of $V$, but independent of $\varepsilon$, such that
	\begin{equation}\label{eq2.1}
		\|\varphi\|_{L^{q}\left(\mathbb{R}^{N}\right)} \leq C \,\varepsilon^{\frac{N}{q}-\frac{N}{2}}\|\varphi\|_{\varepsilon}
	\end{equation}
	holds for all $\varphi \in H_{\varepsilon}$.
\end{Lem}

\begin{proof}
	Set $\tilde{\varphi}(x)=\varphi(\varepsilon x)$. Then by a change of variables,
	\[
	\int_{\mathbb{R}^{N}}|\varphi|^{q}\,dx =\varepsilon^{N} \int_{\mathbb{R}^{N}}|\tilde{\varphi}|^{q}\,dx.
	\]
	By the fractional Sobolev embedding $H^{s}\left(\mathbb{R}^{N}\right) \subset L^{q}\left(\mathbb{R}^{N}\right)$, we have
	\[
	\int_{\mathbb{R}^{N}}|\tilde{\varphi}|^{q}\,dx \leq C_{1} \left(\int_{\mathbb{R}^{N}}\big(|(-\Delta)^{\frac{s}{2}} \tilde{\varphi}|^{2}+|\tilde{\varphi}|^{2}\big)\,dx\right)^{q / 2},
	\]
	where $C_{1}>0$ is the best constant for this embedding. Using the scaling property of the fractional Laplacian,
	\[
	(-\Delta)^{\frac{s}{2}}\tilde{\varphi}(x)=\varepsilon^{s}\big((-\Delta)^{\frac{s}{2}}\varphi\big)(\varepsilon x),
	\]
	we obtain
	\[
	\int_{\mathbb{R}^{N}}\big(|(-\Delta)^{\frac{s}{2}} \tilde{\varphi}|^{2}+|\tilde{\varphi}|^{2}\big)\,dx
	= \varepsilon^{2s-N}\int_{\mathbb{R}^{N}}|(-\Delta)^{\frac{s}{2}}\varphi|^{2}\,dx
	+ \varepsilon^{-N}\int_{\mathbb{R}^{N}}|\varphi|^{2}\,dx
	= \varepsilon^{-N}\int_{\mathbb{R}^{N}}\big(\varepsilon^{2s}|(-\Delta)^{\frac{s}{2}}\varphi|^{2}+|\varphi|^{2}\big)\,dx.
	\]
	Hence
	\[
	\int_{\mathbb{R}^{N}}|\varphi|^{q}\,dx
	\leq C_1\,\varepsilon^{N-\frac{Nq}{2}}
	\left(\int_{\mathbb{R}^{N}}\big(\varepsilon^{2s}|(-\Delta)^{\frac{s}{2}}\varphi|^{2}+|\varphi|^{2}\big)\,dx\right)^{q/2}.
	\]
	By $(V_1)$, the norms $\|\cdot\|_{\varepsilon}$ and
	\[
	\Big(\int_{\mathbb{R}^{N}}\big(\varepsilon^{2s}|(-\Delta)^{\frac{s}{2}}\varphi|^{2}+|\varphi|^{2}\big)\,dx\Big)^{1/2}
	\]
	are equivalent on $H^{s}(\mathbb{R}^{N})$, with equivalence constants independent of $\varepsilon$. Therefore there exists $C_2>0$, depending only on $N,s,a,q$ and on the bounds of $V$, such that
	\[
	\int_{\mathbb{R}^{N}}|\varphi|^{q}\,dx \le C_2\,\varepsilon^{N-\frac{Nq}{2}}\|\varphi\|_{\varepsilon}^{q},
	\]
	which yields \eqref{eq2.1} after taking the $q$-th root.
\end{proof}
For convenience, we introduce the notation
\begin{equation*}
	u_{\varepsilon, y}(x)=u\left(\frac{x-y}{\varepsilon}\right)
\end{equation*}
for $\varepsilon>0$ and $y \in \mathbb{R}^{N}$.
\par
Denote by $u^{(i)} \in H^{s}(\mathbb{R}^{N})$ the unique positive radial solution (see \cite[Theorem 1.1]{R-Yang1}) to the equation
\begin{equation*}
	\Big(a+b{\int_{\mathbb{R}^{N}}}\big|(-\Delta)^{\frac{s}{2}}u\big|^2\,dx\Big)(-\Delta)^s u+V(a_i)u=u^p,\quad \text{in }\mathbb{R}^{N}.
\end{equation*}
Then, for each $i=1, \ldots, k$,  $u_{\varepsilon, y_{\varepsilon}^{i}}^{(i)}(x)=u^{(i)}\!\left(\frac{x-y_{\varepsilon}^{i}}{ \varepsilon}\right)>0$ satisfies
\begin{equation}\label{eq2.2}
	\Big(\varepsilon^{2s} a+\varepsilon^{4s-N} b \int_{\mathbb{R}^{N}}\big|(-\Delta)^{\frac{s}{2}}u_{\varepsilon, y_{\varepsilon}^{i}}^{(i)}\big|^{2}\,dx\Big) (-\Delta)^s u_{\varepsilon, y_{\varepsilon}^{i}}^{(i)}+V\left(a_{i}\right) u_{\varepsilon, y_{\varepsilon}^{i}}^{(i)}=\big(u_{\varepsilon, y_{\varepsilon}^{i}}^{(i)}\big)^{p} \quad \text{in }\mathbb{R}^{N}.
\end{equation}
\par
As mentioned in the introduction, we have our first result.

\begin{Thm}\label{Thm2.1}
	Let $u_{\varepsilon}$ be a $k$-peak solution of equation \eqref{eq1.1} in the sense of Definition \ref{Def1.1}, with local maximum points at $y_{\varepsilon}^{i}$ and $y_{\varepsilon}^{i} \rightarrow a_{i}$ as $\varepsilon\rightarrow0$. Then, for $\varepsilon>0$ sufficiently small, $u_{\varepsilon}$ cannot be represented in the form
	\begin{equation}\label{eq2.3}
		u_{\varepsilon}(x)=\sum_{i=1}^{k} u_{\varepsilon, y_{\varepsilon}^{i}}^{(i)}(x)+\varphi_{\varepsilon}(x)
	\end{equation}
	with a remainder $\varphi_{\varepsilon}$ satisfying
	\begin{equation}\label{eq2.6}
		\|\varphi_{\varepsilon}\|_{\varepsilon}=o\big(\varepsilon^{\frac{N}{2}}\big).
	\end{equation}
	In fact, $u_{\varepsilon}$ must be of the form
	\begin{equation}\label{eq2.4}
		u_{\varepsilon}(x)=\sum_{i=1}^{k} U^{i}\left(\frac{x-y_{\varepsilon}^{i}}{\varepsilon}\right)+\varphi_{\varepsilon}(x)
	\end{equation}
	satisfying
	\begin{itemize}
		\item [$(i)$] $\left(U^{1}, \ldots, U^{k}\right)$ is the unique vector of positive radial solutions to the system
		\begin{equation}\label{eq2.5}
			\left\{\begin{array}{l}
				\Big(a+b \displaystyle\sum\limits_{i=1}^{k} \int_{\mathbb{R}^N}\big|(-\Delta)^{\frac{s}{2}} U^{i}\big|^{2}\,dx\Big)(-\Delta)^s U^{i}+V\left(a_{i}\right) U^{i}=\big(U^{i}\big)^{p} \quad \text{in }\mathbb{R}^{N}, \\[4pt]
				U^{i}>0 \quad \text{in }\mathbb{R}^{N}, \\
				U^{i} \in H^{s}(\mathbb{R}^{N}),
			\end{array}\right.
		\end{equation}
		\item [$(ii)$] the smallness condition \eqref{eq2.6} holds, that is,
		\[
		\|\varphi_{\varepsilon}\|_{\varepsilon}=o\big(\varepsilon^{\frac{N}{2}}\big).
		\]
	\end{itemize}
\end{Thm}

\begin{proof}
	We first prove the negative part. It follows from Proposition \ref{Pro1.1} that
	\begin{equation*}
		u^{(i)}(x)+\big|(-\Delta)^{\frac{s}{2}} u^{(i)}(x)\big| \leq \frac{C}{1+|x|^{N+2s}}, \quad x \in \mathbb{R}^{N}
	\end{equation*}
	for some $C>0$ and for all $i=1,\ldots,k$. Note that
	\[
	\frac{y_{\varepsilon}^{i}-y_{\varepsilon}^{j}}{\varepsilon} \rightarrow \infty \quad\text{as }\varepsilon\to0
	\]
	whenever $i\neq j$, since we assume $a_{i} \neq a_{j}$. A standard change of variables argument then shows that, for $i\neq j$,
	\begin{equation}\label{eq2.7}
		\int_{\mathbb{R}^N}\Big(\varepsilon^{2s}\big|(-\Delta)^{\frac{s}{2}} u_{\varepsilon, y_{\varepsilon}^{i}}^{(i)} \cdot (-\Delta)^{\frac{s}{2}} u_{\varepsilon, y_{\varepsilon}^{j}}^{(j)}\big|+u_{\varepsilon, y_{\varepsilon}^{i}}^{(i)} u_{\varepsilon, y_{\varepsilon}^{j}}^{(j)}\Big)\,dx=o\big(\varepsilon^{N}\big).
	\end{equation}
	Indeed, by setting $x=y_{\varepsilon}^{j}+\varepsilon z$ and using the decay estimates for $u^{(i)}$ and $(-\Delta)^{\frac{s}{2}}u^{(i)}$, one has
	\[
	\int_{\mathbb{R}^N}u_{\varepsilon, y_{\varepsilon}^{i}}^{(i)} u_{\varepsilon, y_{\varepsilon}^{j}}^{(j)}\,dx
	=\varepsilon^{N}\!\int_{\mathbb{R}^N}u^{(i)}\big(z+\tfrac{y_{\varepsilon}^{j}-y_{\varepsilon}^{i}}{\varepsilon}\big)u^{(j)}(z)\,dz
	=o(\varepsilon^N),
	\]
	and similarly for the term involving $(-\Delta)^{\frac{s}{2}} u_{\varepsilon, y_{\varepsilon}^{i}}^{(i)}$.
	
	If we write
	\begin{equation*}
		\mathcal{E}^{2s}=\varepsilon^{2s} a+\varepsilon^{4s-N} b \int_{\mathbb{R}^N}\big|(-\Delta)^{\frac{s}{2}} u_{\varepsilon}\big|^{2}\,dx,
	\end{equation*}
	then using the decomposition
	\[
	u_\varepsilon=\sum_{i=1}^k u_{\varepsilon,y_\varepsilon^i}^{(i)}+\varphi_\varepsilon
	\]
	and the orthogonality estimate \eqref{eq2.7}, together with the smallness assumption \eqref{eq2.6}, we obtain
	\begin{equation}\label{eq2.8}
		a \varepsilon^{2s} \leqslant \mathcal{E}^{2s}
		=\varepsilon^{2s}\left(a+b \sum_{i=1}^{k} \int_{\mathbb{R}^N}\big|(-\Delta)^{\frac{s}{2}} u^{(i)}\big|^{2} \,dx+o_{\varepsilon}(1)\right)
		\leqslant A \varepsilon^{2s}
	\end{equation}
	for some constant $A>a>0$, where $o_{\varepsilon}(1) \rightarrow 0$ as $\varepsilon \rightarrow 0$. The upper bound also follows directly from Definition \ref{Def1.1}(iii), since
	\[
	\varepsilon^{2s}a\int_{\mathbb{R}^N}\big|(-\Delta)^{\frac{s}{2}} u_{\varepsilon}\big|^{2}\,dx
	\le C\varepsilon^N.
	\]
	
	Assume by contradiction that there exists a solution $u_{\varepsilon}$ to equation \eqref{eq1.1} of the form \eqref{eq2.3}, with $\varphi_{\varepsilon}$ satisfying \eqref{eq2.6}. Then we can write
	\begin{equation}\label{eq2.9}
		\begin{aligned}
			\left(\sum_{i=1}^{k} u_{\varepsilon, y_{\varepsilon}^{i}}^{(i)}+\varphi_{\varepsilon}\right)^{p}&=
			\Big(\varepsilon^{2s}a+\varepsilon^{4s-N}b{\int_{\mathbb{R}^{N}}}\big|(-\Delta)^{\frac{s}{2}}u_{\varepsilon}\big|^2\,dx\Big)(-\Delta)^s u_{\varepsilon}+V(x)u_{\varepsilon}\\
			&=
			\sum_{i=1}^{k}\Big(\mathcal{E}^{2s} (-\Delta)^s u_{\varepsilon, y_{\varepsilon}^{i}}^{(i)}+V(x) u_{\varepsilon, y_{\varepsilon}^{i}}^{(i)}\Big)
			+\Big(\mathcal{E}^{2s}(- \Delta)^s \varphi_{\varepsilon}+V(x) \varphi_{\varepsilon}\Big).
		\end{aligned}
	\end{equation}
	Combining \eqref{eq2.2} and \eqref{eq2.9} we have
	\begin{equation}\label{eq2.10}
		\begin{aligned}
			\left(\sum_{i=1}^{k} u_{\varepsilon, y_{\varepsilon}^{i}}^{(i)}+\varphi_{\varepsilon}\right)^{p}-\sum_{i=1}^{k}\big(u_{\varepsilon, y_{\varepsilon}^{i}}^{(i)}\big)^{p}&
			=\sum_{i=1}^{k}\big(V(x)-V(a_{i})\big) u_{\varepsilon, y_{\varepsilon}^{i}}^{(i)}+\Big(\mathcal{E}^{2s}(- \Delta)^s \varphi_{\varepsilon}+V(x) \varphi_{\varepsilon}\Big)\\
			&\quad+\sum_{i=1}^{k}\Big(\mathcal{E}^{2s}-\big(\varepsilon^{2s} a+\varepsilon^{4s-N} b \int_{\mathbb{R}^N}\big|(-\Delta)^{\frac{s}{2}} u_{\varepsilon, y_{\varepsilon}^{i}}^{(i)}\big|^{2}\,dx\big)\Big)(- \Delta)^s u_{\varepsilon, y_{\varepsilon}^{i}}^{(i)}\\
			&=\sum_{i=1}^{k}\big(V(x)-V(a_{i})\big) u_{\varepsilon, y_{\varepsilon}^{i}}^{(i)}+\Big(\mathcal{E}^{2s}(- \Delta)^s \varphi_{\varepsilon}+V(x) \varphi_{\varepsilon}\Big)\\
			&\quad+\sum_{i=1}^{k}\Big(\mathcal{E}^{2s}-\big(\varepsilon^{2s} a+\varepsilon^{2s} b \int_{\mathbb{R}^N}\big|(-\Delta)^{\frac{s}{2}} u^{(i)}\big|^{2}\,dx\big)\Big)(- \Delta)^s u_{\varepsilon, y_{\varepsilon}^{i}}^{(i)}.
		\end{aligned}
	\end{equation}
	Let
	\begin{equation*}
		K_{i}=\sum_{j \neq i}\int_{\mathbb{R}^N}\big|(-\Delta)^{\frac{s}{2}} u^{(j)}\big|^{2} \,dx>0.
	\end{equation*}
	By \eqref{eq2.8}, the last term of \eqref{eq2.10} can be rewritten as
	\begin{equation*}
		\varepsilon^{2s} \sum_{i=1}^{k}\big(b K_{i}+o_{\varepsilon}(1)\big) (-\Delta)^s u_{\varepsilon, y_{\varepsilon}^{i}}^{(i)}.
	\end{equation*}
	Now fix $j\in\{1,\ldots,k\}$. Multiply both sides of \eqref{eq2.10} by $u_{\varepsilon, y_{\varepsilon}^{j}}^{(j)}$ and integrate over $\mathbb{R}^{N}$. Using the fractional integration by parts formula, we obtain
	\begin{equation}\label{eq2.11}
		\begin{aligned}
			&\sum_{i=1}^{k} \varepsilon^{2s}\big(b K_{i}+o_{\varepsilon}(1)\big)  \int_{\mathbb{R}^N} (-\Delta)^{\frac{s}{2}} u_{\varepsilon, y_{\varepsilon}^{i}}^{(i)} \cdot (-\Delta)^{\frac{s}{2}} u_{\varepsilon, y_{\varepsilon}^{j}}^{(j)}\,dx\\
			&=- \int_{\mathbb{R}^N} \sum_{i=1}^{k}\big(V(x)-V(a_{i})\big) u_{\varepsilon, y_{\varepsilon}^{i}}^{(i)} u_{\varepsilon, y_{\varepsilon}^{j}}^{(j)}\,dx
			- \int_{\mathbb{R}^N}\Big(\mathcal{E}^{2s} (-\Delta)^s \varphi_{\varepsilon}+V(x) \varphi_{\varepsilon}\Big) u_{\varepsilon, y_{\varepsilon}^{j}}^{(j)}\,dx \\
			&\quad+ \int_{\mathbb{R}^N}\Bigg(\Big(\sum_{i=1}^{k} u_{\varepsilon, y_{\varepsilon}^{i}}^{(i)}+\varphi_{\varepsilon}\Big)^{p}-\sum_{i=1}^{k}\big(u_{\varepsilon, y_{\varepsilon}^{i}}^{(i)}\big)^{p}\Bigg) u_{\varepsilon, y_{\varepsilon}^{j}}^{(j)}\,dx \\
			&=:-T_{1}-T_{2}+T_{3}.
		\end{aligned}
	\end{equation}
	For $i\neq j$, \eqref{eq2.7} implies that the corresponding terms in the sum on the left-hand side are $o(\varepsilon^N)$. For $i=j$, using the scaling properties we have
	\[
	\int_{\mathbb{R}^N}\big|(-\Delta)^{\frac{s}{2}} u_{\varepsilon, y_{\varepsilon}^{j}}^{(j)}\big|^{2}\,dx
	=\varepsilon^{N-2s}\int_{\mathbb{R}^N}\big|(-\Delta)^{\frac{s}{2}} u^{(j)}\big|^{2}\,dx.
	\]
	Hence
	\begin{equation}\label{eq2.12}
		\sum_{i=1}^{k} \varepsilon^{2s}\big(b K_{i}+o_{\varepsilon}(1)\big)  \int_{\mathbb{R}^N} (-\Delta)^{\frac{s}{2}} u_{\varepsilon, y_{\varepsilon}^{i}}^{(i)} \cdot (-\Delta)^{\frac{s}{2}}u_{\varepsilon, y_{\varepsilon}^{j}}^{(j)}\,dx
		=\varepsilon^{N}\Big(b K_{j}  \int_{\mathbb{R}^N}\big|(-\Delta)^{\frac{s}{2}} u^{(j)}\big|^{2} \,dx+o_{\varepsilon}(1)\Big).
	\end{equation}
	
	Next we estimate $T_1,T_2,T_3$. We rewrite $T_{1}$ as
	\begin{equation*}
		\begin{aligned}
			T_{1} &=  \int_{\mathbb{R}^N}\big(V(x)-V(a_{i})\big)\big(u_{\varepsilon, y_{\varepsilon}^{i}}^{(i)}\big)^{2}\,dx
			+\sum_{i \neq j} \int_{\mathbb{R}^N} \big(V(x)-V(a_{i})\big) u_{\varepsilon, y_{\varepsilon}^{i}}^{(i)} u_{\varepsilon, y_{\varepsilon}^{j}}^{(j)}\,dx \\
			&=: T_{11}+T_{12}.
		\end{aligned}
	\end{equation*}
	Since $V$ is bounded, \eqref{eq2.7} implies $T_{12}=o\big(\varepsilon^{N}\big)$. We further decompose $T_{11}$ as
	\begin{equation*}
		\begin{aligned}
			T_{11} &= \int_{\mathbb{R}^N}\big(V(x)-V(y_{\varepsilon}^{i})\big)\big(u_{\varepsilon, y_{\varepsilon}^{i}}^{(i)}\big)^{2}\,dx
			+ \int_{\mathbb{R}^N}\big(V(y_{\varepsilon}^{i})-V(a_{i})\big)\big(u_{\varepsilon, y_{\varepsilon}^{i}}^{(i)}\big)^{2}\,dx \\
			&=: T_{111}+T_{112}.
		\end{aligned}
	\end{equation*}
	By $(V_2)$, there exists $\alpha\in(0,\frac{N+4s}{2})$ such that $|V(x)-V(y_{\varepsilon}^{i})|\le C|x-y_{\varepsilon}^{i}|^\alpha$ near $y_{\varepsilon}^{i}$. Thus
	\begin{equation*}
		\begin{aligned}
			\left|T_{111}\right| & \leqslant \int_{B_{1}\!\left(y_{\varepsilon}^{i}\right)}\left|V(x)-V(y_{\varepsilon}^{i})\right|\big(u_{\varepsilon, y_{\varepsilon}^{i}}^{(i)}\big)^{2}\,dx
			+\int_{B_{1}^{c}\!\left(y_{\varepsilon}^{i}\right)}\left|V(x)-V(y_{\varepsilon}^{i})\right|\big(u_{\varepsilon, y_{\varepsilon}^{i}}^{(i)}\big)^{2}\,dx \\
			& \leqslant C \int_{B_{1}\!\left(y_{\varepsilon}^{i}\right)}\left|x-y_{\varepsilon}^{i}\right|^{\alpha}\big(u_{\varepsilon, y_{\varepsilon}^{i}}^{(i)}\big)^{2}\,dx
			+2\|V\|_{L^{\infty}\left(\mathbb{R}^{N}\right)} \int_{B_{1}^{c}\!\left(y_{\varepsilon}^{i}\right)}\big(u_{\varepsilon, y_{\varepsilon}^{i}}^{(i)}\big)^{2}\,dx \\
			&=o\big(\varepsilon^{N}\big),
		\end{aligned}
	\end{equation*}
	where we used the change of variables $x=y_{\varepsilon}^{i}+\varepsilon z$ and the decay of $u^{(i)}$. Since $y_{\varepsilon}^{i} \rightarrow a_{i}$ and
	\[
	\int_{\mathbb{R}^N}\big(u_{\varepsilon,y_\varepsilon^i}^{(i)}\big)^2\,dx = \varepsilon^N \int_{\mathbb{R}^N}\big(u^{(i)}\big)^2\,dz,
	\]
	we also have
	\begin{equation*}
		T_{112}=\int_{\mathbb{R}^N}\big(V(y_{\varepsilon}^{i})-V(a_{i})\big)\big(u_{\varepsilon, y_{\varepsilon}^{i}}^{(i)}\big)^{2}\,dx=o\big(\varepsilon^{N}\big).
	\end{equation*}
	Hence $T_{11}=o\big(\varepsilon^{N}\big)$, which together with the estimate of $T_{12}$ gives
	\begin{equation}\label{eq2.13}
		T_{1}=o\big(\varepsilon^{N}\big).
	\end{equation}
	
	For $T_{2}$, using \eqref{eq2.8}, the boundedness of $V$ and Cauchy--Schwarz inequality we obtain
	\begin{equation*}
		\begin{aligned}
			T_{2}
			&=\int_{\mathbb{R}^N} \Big(\mathcal{E}^{2s}(- \Delta)^s \varphi_{\varepsilon}+V(x) \varphi_{\varepsilon}\Big) u_{\varepsilon, y_{\varepsilon}^{j}}^{(j)}\,dx\\
			&\le C\,\big\|\varphi_{\varepsilon}\big\|_{\varepsilon}\,\big\|u_{\varepsilon, y_{\varepsilon}^{j}}^{(j)}\big\|_{\varepsilon}
			=O\!\Big(\big\|\varphi_{\varepsilon}\big\|_{\varepsilon}\,\varepsilon^{\frac N2}\Big),
		\end{aligned}
	\end{equation*}
	so that, by the assumption \eqref{eq2.6},
	\begin{equation}\label{eq2.14}
		T_{2}=o\big(\varepsilon^{N}\big).
	\end{equation}
	
	To estimate the last term $T_{3}$, we apply an elementary inequality to obtain
	\begin{equation*}
		\begin{aligned}
			\left|T_{3}\right| & \leqslant C \int_{\mathbb{R}^N} \Bigg(\Big(\sum_{i=1}^{k} u_{\varepsilon, y_{\varepsilon}^{i}}^{(i)}\Big)^{p-1}\big|\varphi_{\varepsilon}\big|+\sum_{i=1}^{k} u_{\varepsilon, y_{\varepsilon}^{i}}^{(i)}\big|\varphi_{\varepsilon}\big|^{p-1}+\big|\varphi_{\varepsilon}\big|^{p}\Bigg) u_{\varepsilon, y_{\varepsilon}^{j}}^{(j)}\,dx \\
			& \leqslant C  \int_{\mathbb{R}^N}\Bigg(\sum_{i=1}^{k}\big(u_{\varepsilon, y_{\varepsilon}^{i}}^{(i)}\big)^{p-1}\big|\varphi_{\varepsilon}\big|
			+\sum_{i=1}^{k} u_{\varepsilon, y_{\varepsilon}^{i}}^{(i)}\big|\varphi_{\varepsilon}\big|^{p-1}
			+\big|\varphi_{\varepsilon}\big|^{p}\Bigg) u_{\varepsilon, y_{\varepsilon}^{j}}^{(j)}\,dx.
		\end{aligned}
	\end{equation*}
	Using H\"older's inequality, Lemma \ref{Lem2.1} (with $q=p+1$) and the assumption $\|\varphi_{\varepsilon}\|_{\varepsilon}=o\big(\varepsilon^{\frac{N}{2}}\big)$, we obtain
	\begin{equation*}
		\begin{gathered}
			\int_{\mathbb{R}^N}\big(u_{\varepsilon, y_{\varepsilon}^{i}}^{(i)}\big)^{p-1} u_{\varepsilon, y_{\varepsilon}^{j}}^{(j)}\big|\varphi_{\varepsilon}\big| \,dx
			\leqslant\big\|u_{\varepsilon, y_{\varepsilon}^{i}}^{(i)}\big\|_{L^{p+1}}^{p-1}\big\|u_{\varepsilon, y_{\varepsilon}^{j}}^{(j)}\big\|_{L^{p+1}}\big\|\varphi_{\varepsilon}\big\|_{L^{p+1}}
			=o\big(\varepsilon^{N}\big), \\[2pt]
			\int_{\mathbb{R}^N} u_{\varepsilon, y_{\varepsilon}^{i}}^{(i)} u_{\varepsilon, y_{\varepsilon}^{j}}^{(j)}\big|\varphi_{\varepsilon}\big|^{p-1}dx
			\leqslant\big\|u_{\varepsilon, y_{\varepsilon}^{i}}^{(i)}\big\|_{L^{p+1}}\big\|u_{\varepsilon, y_{\varepsilon}^{j}}^{(j)}\big\|_{L^{p+1}}\big\|\varphi_{\varepsilon}\big\|_{L^{p+1}}^{p-1}
			=o\big(\varepsilon^{N}\big), \\[2pt]
			\int_{\mathbb{R}^N}\big|\varphi_{\varepsilon}\big|^{p} u_{\varepsilon, y_{\varepsilon}^{j}}^{(j)}\,dx
			\leqslant\big\|u_{\varepsilon, y_{\varepsilon}^{j}}^{(j)}\big\|_{L^{p+1}}\big\|\varphi_{\varepsilon}\big\|_{L^{p+1}}^{p}
			=o\big(\varepsilon^{N}\big),
		\end{gathered}
	\end{equation*}
	where we used that $\|u_{\varepsilon,y_\varepsilon^i}^{(i)}\|_{L^{p+1}} \sim \varepsilon^{\frac N{p+1}}$ and $\|\varphi_\varepsilon\|_{L^{p+1}}\le C\varepsilon^{\frac N{p+1}-\frac N2}\|\varphi_\varepsilon\|_\varepsilon$. Therefore,
	\begin{equation}\label{eq2.15}
		T_{3}=o\big(\varepsilon^{N}\big).
	\end{equation}
	Finally, combining \eqref{eq2.12}–\eqref{eq2.15} we get
	\begin{equation*}
		\varepsilon^{N}\Big(b K_{j}  \int_{\mathbb{R}^N}\big|(-\Delta)^{\frac{s}{2}} u^{(j)}\big|^{2}\,dx+o_{\varepsilon}(1)\Big)
		=o\big(\varepsilon^{N}\big)
	\end{equation*}
	as $\varepsilon \rightarrow 0$. This is impossible since $K_{j}>0$ and $u^{(j)}\not\equiv0$. The first part of Theorem \ref{Thm2.1} is complete.
	\par
	
	Now we prove the remaining assertions. For comparison, recall that in the case of fractional Schr\"odinger equations (i.e., $b=0$), if $u_{\varepsilon}$ is a multi-peak solution, then $u_{\varepsilon}$ must be of the form
	\begin{equation*}
		u_{\varepsilon}(x)=\sum_{i=1}^{k} u_{\varepsilon, y_{\varepsilon}^{i}}^{i}(x)+\varphi_{\varepsilon}(x),
	\end{equation*}
	where $u^{i} \in H^{s}\left(\mathbb{R}^{N}\right)$ is the unique positive radial solution to
	\begin{equation*}
		a(- \Delta)^s v+V\left(a_{i}\right) v=v^{p}, \quad v>0 \text{ in } \mathbb{R}^{N},
	\end{equation*}
	and $y_{\varepsilon}^{i}, \varphi_{\varepsilon}$ satisfy suitable properties (see \cite{MR3343559}).
	\par
	In our case, suppose $u_{\varepsilon}$ is a multi-peak solution to equation \eqref{eq1.1} with local maximum points $y_{\varepsilon}^{i}$ ($1 \leqslant i \leqslant k$). For each $1 \leqslant i \leqslant k$ set
	\[
	\bar{u}_{\varepsilon}(x) \equiv u_{\varepsilon}\big(\varepsilon x+y_{\varepsilon}^{i}\big).
	\]
	Using Definition \ref{Def1.1}(iii) and the scaling properties, one checks that $(\bar u_\varepsilon)$ is uniformly bounded in $H^{s}(\mathbb{R}^N)$ and satisfies
	\begin{equation*}
		\Big(a+b \int_{\mathbb{R}^N}\big|(-\Delta)^{\frac{s}{2}} \bar{u}_{\varepsilon}\big|^{2}\,dx\Big) (-\Delta)^s \bar{u}_{\varepsilon}
		+V\big(\varepsilon x+y_{\varepsilon}^{i}\big) \bar{u}_{\varepsilon}
		=\bar{u}_{\varepsilon}^{p} \quad \text{in } \mathbb{R}^{N}.
	\end{equation*}
	Thus there exists a subsequence $\varepsilon_{l} \rightarrow 0$ such that $\bar{u}_{l}(x) \equiv \bar u_{\varepsilon_l}(x)=u_{\varepsilon_{l}}\big(\varepsilon_{l} x+y_{\varepsilon_{l}}^{i}\big)$ converges weakly to a function $U^{i}$ in $H^{s}(\mathbb{R}^N)$ and
	\begin{equation*}
		a+b \int_{\mathbb{R}^N}\big|(-\Delta)^{\frac{s}{2}} \bar{u}_{l}\big|^{2}\,dx \rightarrow A
	\end{equation*}
	as $l \rightarrow \infty$ for some constant $A>0$. Passing to the limit in the weak formulation, we find that
	\begin{equation*}
		A(-\Delta)^s U^{i}+V\left(a_{i}\right) U^{i}=\big(U^{i}\big)^{p} \quad \text{in } \mathbb{R}^{N} .
	\end{equation*}
	Note that $x=0$ is a maximum point of $U^{i}$. By the uniqueness and symmetry results for positive solutions (see Proposition \ref{Pro1.1} and the scaling arguments in \cite{R-Yang1}), $U^{i}(x)=U^{i}(|x|)$ must be the unique positive radial solution to the above equation. Moreover, $U^{i}(r)$ is strictly decreasing as $|x|\rightarrow\infty$.
	
	Using standard concentration--compactness arguments for multi-peak solutions of fractional Schr\"odinger-type equations, we can then show that, up to a subsequence,
	\[
	u_{\varepsilon_l}(x)=\sum_{i=1}^{k} U^{i}\!\left(\frac{x-y_{\varepsilon_{l}}^{i}}{ \varepsilon_{l}}\right)+\varphi_{\varepsilon_{l}}(x),
	\]
	with $y_{\varepsilon_l}^i\to a_i$ and $\|\varphi_{\varepsilon_l}\|_{\varepsilon_l}=o\big(\varepsilon_l^{\frac N2}\big)$ as $l\to\infty$.
	
	Finally, noting that $\frac{|y_{\varepsilon_{l}}^{i}-y_{\varepsilon_{l}}^{j}|}{\varepsilon_{l}} \rightarrow \infty$ implies that $U_{\varepsilon_{l}, y_{\varepsilon_{l}}^{i}}^{i}$, $1 \leqslant i \leqslant k$, are mutually asymptotically orthogonal, that is,
	\begin{equation*}
		\int_{\mathbb{R}^N}\Big((-\Delta)^{\frac{s}{2}} U_{\varepsilon_{l}, y_{\varepsilon_{l}}^{i}}^{i} \cdot (-\Delta)^{\frac{s}{2}} U_{\varepsilon_{l}, y_{\varepsilon_{l}}^{j}}^{j}
		+U_{\varepsilon_{l}, y_{\varepsilon_{l}}^{i}}^{i} \, U_{\varepsilon_{l}, y_{\varepsilon_{l}}^{j}}^{j}\Big)\,dx \rightarrow 0 \quad \text{as } l \rightarrow \infty \text{ for } i \neq j,
	\end{equation*}
	we deduce
	\begin{equation*}
		A=\lim _{l \rightarrow \infty}\Big(a+b \int_{\mathbb{R}^N}\big|(-\Delta)^{\frac{s}{2}} \bar{u}_{l}\big|^{2}\,dx\Big)
		=a+b \sum_{i=1}^{k} \int_{\mathbb{R}^N}\big|(-\Delta)^{\frac{s}{2}} U^{i}\big|^{2}\,dx .
	\end{equation*}
	Thus, $\big(U^{1},\ldots,U^{k}\big)$ satisfies the system \eqref{eq2.5}, and the constant $A$ is independent of the choice of the subsequence $\{\varepsilon_l\}$. This, in turn, means that the above analysis applies to the whole sequence $\{u_{\varepsilon}\}$, yielding the representation \eqref{eq2.4} with \eqref{eq2.6}. The proof of Theorem \ref{Thm2.1} is complete.
\end{proof}

\par
In order to locate multi-peak solutions of equation \eqref{eq1.1}, we recall the following local Pohoz\v{a}ev type identity from \cite{R-Yang2}.
\begin{Lem}\label{Lem2.2}
	Let $u$ be a positive solution of \eqref{eq1.1}. Let $\Omega$ be a bounded smooth domain in $\mathbb{R}^N$. Then, for each $j=1,2,\cdots,N$, there holds
	\begin{equation}\label{eq2.16}
		\begin{aligned}
			\int_{\Omega} \frac{\partial V}{\partial x_{j}} u^{2}\,dx
			&=\left(\varepsilon^{2s} a+\varepsilon^{4s-N} b \int_{\mathbb{R}^{N}}|(-\Delta)^{\frac{s}{2}} u|^{2}\,dx\right)
			\int_{\partial \Omega}\left(|(-\Delta)^{\frac{s}{2}} u|^{2} \nu_{j}-2 \frac{\partial u}{\partial \nu} \frac{\partial u}{\partial x_{j}}\right)\,d\sigma \\
			&\quad+\int_{\partial \Omega} V u^{2} \nu_{j}\,d\sigma-\frac{2}{p+1} \int_{\partial \Omega} u^{p+1} \nu_{j}\,d\sigma .
		\end{aligned}
	\end{equation}
	Here $\nu=\left(\nu_{1}, \nu_{2},\cdots, \nu_{N}\right)$ is the unit outward normal of $\partial \Omega$.
\end{Lem}

\begin{Lem}\label{Lem2.3}
	Suppose $V $ satisfies $(V_1)-(V_3)$. Let
	\[
	u_{\varepsilon}=\sum_{i=1}^{k} U_{\varepsilon, y_{\varepsilon}^{i}}^{i}+\varphi_{\varepsilon}
	\]
	be a multi-peak solution to equation \eqref{eq1.1} as in Theorem \ref{Thm2.1}. Then $\nabla V(a_{i})=0$ for each $i=1, \ldots, k$.
\end{Lem}

\begin{proof}
	Without loss of generality, we only prove the result for $i=1$. Assume that
	\[
	\left|V_{x_{1}}\left(a_{1}\right)\right|=C_{0}>0 .
	\]
	We will apply Lemma \ref{Lem2.2} to $u_{\varepsilon}$ with $\Omega=B_{r}\left(a_{1}\right)$ to deduce a contradiction.
	\par
	We choose the radius $r$ as follows. Let
	\[
	r_{0} \equiv \min _{i \neq 1}\left\{1,\frac{|y_{\varepsilon}^{i}-y_{\varepsilon}^{1}|}{10}\right\}.
	\]
	By \eqref{eq2.1} and the assumption $\left\|\varphi_{\varepsilon}\right\|_{\varepsilon}=o\left(\varepsilon^{\frac{N}{2}}\right)$, we have
	\[
	\left\|\varphi_{\varepsilon}\right\|_{L^{p+1}\left(\mathbb{R}^{N}\right)}
	\leqslant C \varepsilon^{\frac{N}{p+1}-\frac{N}{2}}\left\|\varphi_{\varepsilon}\right\|_{\varepsilon}
	=o\left(\varepsilon^{\frac{N}{p+1}}\right).
	\]
	Set
	\[
	f=\varepsilon^{2s}\left|(-\Delta)^{\frac{s}{2}} \varphi_{\varepsilon}\right|^{2}+\left|\varphi_{\varepsilon}\right|^{2}+\left|\varphi_{\varepsilon}\right|^{p+1}.
	\]
	Using polar coordinates,
	\[
	\int_{0}^{r_{0}} \int_{\partial B_{r}\left(a_{1}\right)} f\,d\sigma\,dr
	=\int_{B_{r_{0}}\left(a_{1}\right)} f\,dx,
	\]
	we can choose $r \in\left(0, r_{0}\right)$ such that
	\begin{equation}\label{eq2.17}
		\int_{\partial B_{r}\left(a_{1}\right)}\left(\varepsilon^{2s}\left|(-\Delta)^{\frac{s}{2}} \varphi_{\varepsilon}\right|^{2}
		+\left|\varphi_{\varepsilon}\right|^{2}+\left|\varphi_{\varepsilon}\right|^{p+1}\right)\,d\sigma
		=o\left(\varepsilon^{N}\right) .
	\end{equation}
	\par
	Now we apply the Pohoz\v{a}ev identity \eqref{eq2.16} to $u_{\varepsilon}$ with $\Omega=B_{r}\left(a_{1}\right)$ and $r$ chosen as above. We obtain
	\begin{equation}\label{eq2.18}
		\begin{aligned}
			\int_{B_{r}\left(a_{1}\right)} \frac{\partial V}{\partial x_{1}} u_{\varepsilon}^{2}\,dx
			&= \mathcal{E}^{2s} \int_{\partial B_{r}\left(a_{1}\right)}\left(\left|(-\Delta)^{\frac{s}{2}} u_{\varepsilon}\right|^{2} \nu_{1}
			-2 \frac{\partial u_{\varepsilon}}{\partial \nu} \frac{\partial u_{\varepsilon}}{\partial x_{1}}\right)\,d\sigma \\
			&\quad+\int_{\partial B_{r}\left(a_{1}\right)} V u_{\varepsilon}^{2} \nu_{1}\,d\sigma
			-\frac{2}{p+1} \int_{\partial B_{r}\left(a_{1}\right)} u_{\varepsilon}^{p+1} \nu_{1}\,d\sigma,
		\end{aligned}
	\end{equation}
	where
	\[
	\mathcal{E}^{2s}=\varepsilon^{2s} a+\varepsilon^{4s-N} b  \int_{\mathbb{R}^N}\left|(-\Delta)^{\frac{s}{2}} u_{\varepsilon}\right|^{2}\,dx
	=O\left(\varepsilon^{2s}\right) .
	\]
	We estimate \eqref{eq2.18} term by term. To estimate
	\[
	\int_{B_{r}\left(a_{1}\right)} \frac{\partial V}{\partial x_{1}} u_{\varepsilon}^{2}\,dx,
	\]
	split it into
	\begin{equation}\label{eq2.19}
		\int_{B_{r}\left(a_{1}\right)} \frac{\partial V}{\partial x_{1}} u_{\varepsilon}^{2}\,dx
		=\int_{B_{r}\left(a_{1}\right)}\left(V_{x_{1}}(x)-V_{x_{1}}\left(a_{1}\right)\right) u_{\varepsilon}^{2}\,dx
		+V_{x_{1}}\left(a_{1}\right) \int_{B_{r}\left(a_{1}\right)} u_{\varepsilon}^{2}\,dx .
	\end{equation}
	By continuity, we have
	\[
	\left|\int_{B_{r}\left(a_{1}\right)}\left(V_{x_{1}}(x)-V_{x_{1}}\left(a_{1}\right)\right) u_{\varepsilon}^{2}\,dx\right|
	\leqslant \max _{x \in B_{r}\left(a_{1}\right)}\left|V_{x_{1}}(x)-V_{x_{1}}\left(a_{1}\right)\right|
	\int_{B_{r}\left(a_{1}\right)} u_{\varepsilon}^{2}\,dx .
	\]
	Noting that $\left|a_{i}-a_{1}\right|>2 r$ for each $i \neq 1$, using Proposition \ref{Pro1.1} $(iii)$ and the assumption $\left\|\varphi_{\varepsilon}\right\|_{\varepsilon}=o\left(\varepsilon^{\frac{N}{2}}\right)$, we deduce
	\[
	C_{1} \varepsilon^{N} \leqslant \int_{B_{r}\left(a_{1}\right)} u_{\varepsilon}^{2}\,dx
	=\int_{B_{r}\left(a_{1}\right)}\left(U_{\varepsilon, y_{\varepsilon}^{1}}^{1}\right)^{2}\,dx+o\left(\varepsilon^{N}\right)
	\leqslant C_{2} \varepsilon^{N}
	\]
	for $\varepsilon$ sufficiently small, where $C_{1}, C_{2}>0$ are independent of $\varepsilon$. Hence, for $\varepsilon$ sufficiently small, there holds
	\[
	\left|\int_{B_{r}\left(a_{1}\right)}\left(V_{x_{1}}(x)-V_{x_{1}}\left(a_{1}\right)\right) u_{\varepsilon}^{2}\,dx\right|
	\leqslant C_{2} \max _{x \in B_{r}\left(a_{1}\right)}\left|V_{x_{1}}(x)-V_{x_{1}}\left(a_{1}\right)\right| \varepsilon^{N}
	\]
	and
	\[
	\left|V_{x_{1}}\left(a_{1}\right)\right| \int_{B_{r}\left(a_{1}\right)} u_{\varepsilon}^{2}\,dx
	\geqslant C_{0} C_{1} \varepsilon^{N} .
	\]
	Combining the above two estimates and choosing $r$ sufficiently small, we obtain
	\begin{equation}\label{eq2.20}
		\left|\int_{B_{r}\left(a_{1}\right)} \frac{\partial V}{\partial x_{1}} u_{\varepsilon}^{2}\,dx\right|
		\geqslant\left(C_{0} C_{1}-C_{2} \max _{x \in B_{r}\left(a_{1}\right)}\left|V_{x_{1}}(x)-V_{x_{1}}\left(a_{1}\right)\right|\right) \varepsilon^{N}
		\geqslant \frac{C_{0} C_{1}}{2} \varepsilon^{N}.
	\end{equation}
	On the contrary, we have
	\begin{equation}\label{eq2.21}
		\begin{aligned}
			&\mathcal{E}^{2s}\left|\int_{\partial B_{r}\left(a_{1}\right)}\left(\left|(-\Delta)^{\frac{s}{2}} u_{\varepsilon}\right|^{2} \nu_{1}
			-2 \frac{\partial u_{\varepsilon}}{\partial \nu} \frac{\partial u_{\varepsilon}}{\partial x_{1}}\right)\,d\sigma\right| \\
			& \leqslant C \varepsilon^{2s} \int_{\partial B_{r}\left(a_{1}\right)}\left(\sum_{i=1}^{k}\left|(-\Delta)^{\frac{s}{2}} U_{\varepsilon, y_{\varepsilon}^{i}}^{i}\right|^{2}
			+\left|(-\Delta)^{\frac{s}{2}} \varphi_{\varepsilon}\right|^{2}\right)\,d\sigma \\
			&=o\left(\varepsilon^{N}\right)
		\end{aligned}
	\end{equation}
	and
	\begin{equation}\label{eq2.22}
		\begin{aligned}
			&\left|\int_{\partial B_{r}\left(a_{1}\right)} V u_{\varepsilon}^{2} \nu_{1}\,d\sigma
			-\frac{2}{p+1} \int_{\partial B_{r}\left(a_{1}\right)} u_{\varepsilon}^{p+1} \nu_{1}\,d\sigma\right| \\
			& \leqslant C \int_{\partial B_{r}\left(a_{1}\right)}\left(\sum_{i=1}^{k}\left(U_{\varepsilon, y_{\varepsilon}^{i}}^{i}\right)^{2}
			+\left|\varphi_{\varepsilon}\right|^{2}
			+\sum_{i=1}^{k}\left(U_{\varepsilon, y_{\varepsilon}^{i}}^{i}\right)^{p+1}
			+\left|\varphi_{\varepsilon}\right|^{p+1}\right)\,d\sigma \\
			&=o\left(\varepsilon^{N}\right),
		\end{aligned}
	\end{equation}
	where we have used \eqref{eq2.17} and the polynomial decay of $U^{i}$ at infinity. Finally, combining \eqref{eq2.19}–\eqref{eq2.22}, we obtain
	\[
	\frac{C_{0} C_{1}}{2} \varepsilon^{N} \leqslant o\left(\varepsilon^{N}\right), \quad \text { as } \varepsilon \rightarrow 0.
	\]
	We reach a contradiction. The proof is complete.
\end{proof}

\section{Some preliminaries}
In this section, we introduce the Lyapunov--Schmidt reduction method used in the proof of Theorem \ref{Thm1.2} and present some elementary estimates for later use.
\par
It is known that every solution to \eqref{eq1.1} is a critical point of the energy functional $I_{\varepsilon}: H_{\varepsilon} \rightarrow \mathbb{R}$, given by
\begin{equation*}
	I_{\varepsilon}(u)
	=\frac{1}{2}\|u\|_{\varepsilon}^{2}
	+\frac{b\varepsilon^{4s-N}}{4}\left(\int_{\mathbb{R}^{N}}\left|(-\Delta)^{\frac{s}{2}} u\right|^{2}dx\right)^{2}
	-\frac{1}{p+1} \int_{\mathbb{R}^{N}} u^{p+1}dx
\end{equation*}
for $u \in H_{\varepsilon}$. It is standard to verify that $I_{\varepsilon} \in C^{2}\left(H_{\varepsilon}\right)$. So we are left to find a critical point of $I_{\varepsilon}$. However, due to the presence of the double nonlocal terms $(-\Delta)^s$ and $\left(\int_{\mathbb{R}^{N}}\left|(-\Delta)^{\frac{s}{2}} u\right|^{2}dx\right)$, it requires more careful estimates on the orders of $\varepsilon$ in the procedure. In particular, the nonlocal terms bring new difficulties in the higher order remainder term, which is more complicated than in the case of the fractional Schr\"odinger equation \eqref{eq1.6} or the usual Kirchhoff equation \eqref{eq1.4}.
\par
To obtain multi-peak solutions to equation \eqref{eq1.1}, Theorem \ref{Thm2.1} suggests to construct solutions of the form \eqref{eq2.4}. Following the idea of Cao and Peng \cite{MR2581983} (see also \cite{R-Yang1}), we will use the unique ground state $(U^1,\cdots, U^k)$ to the system \eqref{eq2.5} to build solutions of \eqref{eq1.1}. Since the $\varepsilon$-scaling makes $U^i\left(\frac{x-y^{i}}{\varepsilon}\right)$ concentrate around $y^{i}$, the superposition $U_{\varepsilon,y}$ constitutes a good positive approximate solution to \eqref{eq1.1}.
\par
Let $k$ be any positive integer and let $\delta>0$ be such that the balls $B_{\delta}(a_i)$ are pairwise disjoint for all $i = 1, \ldots , k$. We define
\begin{equation*}
	\begin{aligned}
		D_{\varepsilon, \delta}
		=\Big\{y=\left(y^{1}, \ldots, y^{k}\right) \in\left(\mathbb{R}^{N}\right)^{k}:\ &
		|y^{i}-a_{i}|<\delta \ \text{for } i=1,\ldots,k,\\
		&\frac{|y^{i}-y^{j}|}{\varepsilon}\geq\varepsilon^{\theta-1} \ \text{for all } i\neq j\Big\},
	\end{aligned}
\end{equation*}
where $\frac{N+2s}{N+2s+\alpha}<\theta<1$. For each $y\in D_{\varepsilon,\delta}$ we set
\[
U_{\varepsilon, y}(x)=\sum_{i=1}^{k} U^i_{\varepsilon, y^{i}}(x),
\qquad
U^i_{\varepsilon,y^{i}}(x):=U^i\Big(\frac{x-y^{i}}{\varepsilon}\Big).
\]
We also define
\begin{equation*}
	M_{\varepsilon,\eta}
	=\left\{(y, \varphi): y \in D_{\varepsilon, \delta},\ \varphi\in E_{\varepsilon,y},\ \|\varphi\|_{\varepsilon}^{2}\leq\eta\varepsilon^{N}\right\},
\end{equation*}
where
\begin{equation*}
	E_{\varepsilon, y}:=\left\{\varphi \in H^{s}\left(\mathbb{R}^{N}\right):
	\left\langle\frac{\partial U_{\varepsilon, y}}{\partial y^{i}_{j}}, \varphi\right\rangle_{\varepsilon}=0,\ i=1, \ldots, k,\ j=1,\ldots,N\right\}.
\end{equation*}
\par
We will restrict our argument to the existence of a critical point of $I_{\varepsilon}$ that concentrates, as $\varepsilon\to0$, near the points $a_{1}, \ldots, a_{k}$. Thus we
are looking for a critical point of the form
\begin{equation*}
	u_\varepsilon=U_{\varepsilon, y}+\varphi_{\varepsilon},
\end{equation*}
where $\varphi_{\varepsilon} \in E_{\varepsilon, y}$, $y^{i}_{\varepsilon } \rightarrow a_{i}$ and $\left\|\varphi_{\varepsilon}\right\|_{\varepsilon}^{2}=o\left(\varepsilon^{N}\right)$ as $\varepsilon \rightarrow 0$. For this we introduce a new functional $J_{\varepsilon}:M_{\varepsilon,\eta} \rightarrow \mathbb{R}$ defined by
\begin{equation*}
	J_{\varepsilon}(y, \varphi)=I_{\varepsilon}\left(U_{\varepsilon, y}+\varphi\right), \quad \varphi \in E_{\varepsilon, y}.
\end{equation*}
In fact, we divide the proof of Theorem \ref{Thm1.2} into two steps:
\begin{itemize}
	\item [{\bf Step 1:}] for each $\varepsilon, \delta$ sufficiently small and for each $y \in D_{\varepsilon, \delta}$, we will find a critical point $\varphi_{\varepsilon, y}$ for $J_{\varepsilon}(y, \cdot)$ (moreover, the map $y \mapsto \varphi_{\varepsilon, y}$ is of class $C^{1}$ with values in $H_{\varepsilon}$);
	\item [{\bf Step 2:}] for each $\varepsilon, \delta$ sufficiently small, we will find a critical point $y_{\varepsilon}$ for the function $j_{\varepsilon}: D_{\varepsilon, \delta} \rightarrow \mathbb{R}$ induced by
	\begin{equation}\label{eq3.1}
		y \mapsto j_{\varepsilon}(y) \equiv J_{\varepsilon}\left(y, \varphi_{\varepsilon, y}\right).
	\end{equation}
	That is, we will find a critical point $y_{\varepsilon}$ in the interior of $D_{\varepsilon, \delta}$.
\end{itemize}
\par
It is standard to verify that $\left(y_{\varepsilon}, \varphi_{\varepsilon, y_{\varepsilon}}\right)$ is a critical point of $J_{\varepsilon}$ for $\varepsilon$ sufficiently small, by the chain rule. This gives a solution $u_{\varepsilon}= U_{\varepsilon, y_{\varepsilon}}+\varphi_{\varepsilon, y_{\varepsilon}}$ to \eqref{eq1.1} for $\varepsilon$ sufficiently small in virtue of the following lemma.
\begin{Lem}\label{Lem3.1}\cite{R-Yang1}
	There exist $\varepsilon_{0}, \eta_{0}>0$ such that for $\varepsilon \in\left(0, \varepsilon_{0}\right]$, $\eta \in\left(0, \eta_{0}\right]$, and $(y, \varphi) \in M_{\varepsilon,\eta}$ the following are equivalent:
	\begin{itemize}
		\item [$(i)$] $u_{\varepsilon}= U_{\varepsilon, y}+\varphi$ is a critical point of $I_{\varepsilon}$ in $H_{\varepsilon}$;
		\item [$(ii)$] $(y, \varphi)$ is a critical point of $J_{\varepsilon}$.
	\end{itemize}
\end{Lem}
\par
Now, in order to realize {\bf Step 1}, we expand $J_{\varepsilon}(y, \cdot)$ near $\varphi=0$ for each fixed $y$ as follows:
\begin{equation*}
	J_{\varepsilon}(y, \varphi)
	=J_{\varepsilon}(y, 0)+l_{\varepsilon}(\varphi)
	+\frac{1}{2}\left\langle\mathcal{L}_{\varepsilon} \varphi, \varphi\right\rangle
	+R_{\varepsilon}(\varphi),
\end{equation*}
where $J_{\varepsilon}(y, 0)=I_{\varepsilon}\left(U_{\varepsilon, y}\right)$, and $l_{\varepsilon}, \mathcal{L}_{\varepsilon}$ and $R_{\varepsilon}$ are defined for $\varphi, \psi \in H_{\varepsilon}$ as follows:
\begin{equation}\label{eq3.2}
	\begin{aligned}
		l_{\varepsilon}(\varphi)
		&=\left\langle I_{\varepsilon}^{\prime}\left(U_{\varepsilon, y}\right), \varphi\right\rangle \\
		&=\left\langle U_{\varepsilon, y}, \varphi\right\rangle_{\varepsilon}
		+b\varepsilon^{4s-N}\left(\int_{\mathbb{R}^{N}}\left| (-\Delta)^{\frac{s}{2}}U_{\varepsilon, y}\right|^{2}dx\right)
		\int_{\mathbb{R}^{N}} (-\Delta)^{\frac{s}{2}} U_{\varepsilon, y} \cdot (-\Delta)^{\frac{s}{2}} \varphi\,dx \\
		&\quad -\int_{\mathbb{R}^{N}} U_{\varepsilon, y}^{p} \varphi\,dx,
	\end{aligned}
\end{equation}
and $\mathcal{L}_{\varepsilon}$ is the bounded self-adjoint operator on $H_{\varepsilon}$ defined by
\begin{equation*}
	\begin{aligned}
		\left\langle\mathcal{L}_{\varepsilon} \varphi, \psi\right\rangle
		&=\left\langle I_{\varepsilon}^{\prime \prime}\left(U_{\varepsilon, y}\right)[\varphi], \psi\right\rangle \\
		&=\langle\varphi, \psi\rangle_{\varepsilon}
		+ b\varepsilon^{4s-N}\left(\int_{\mathbb{R}^{N}}\left|(-\Delta)^{\frac{s}{2}} U_{\varepsilon, y}\right|^{2}dx\right)
		\int_{\mathbb{R}^{N}} (-\Delta)^{\frac{s}{2}} \varphi \cdot (-\Delta)^{\frac{s}{2}} \psi\,dx \\
		&\quad +2b\varepsilon^{4s-N}\left(\int_{\mathbb{R}^{N}} (-\Delta)^{\frac{s}{2}} U_{\varepsilon, y} \cdot (-\Delta)^{\frac{s}{2}} \varphi\,dx\right)
		\left(\int_{\mathbb{R}^{N}} (-\Delta)^{\frac{s}{2}} U_{\varepsilon, y} \cdot (-\Delta)^{\frac{s}{2}} \psi\,dx\right) \\
		&\quad -p \int_{\mathbb{R}^{N}} U_{\varepsilon, y}^{p-1} \varphi \psi\,dx,
	\end{aligned}
\end{equation*}
and $R_{\varepsilon}$ denotes the second order remainder term given by
\begin{equation}\label{eq3.3}
	R_{\varepsilon}(\varphi)
	=J_{\varepsilon}(y, \varphi)-J_{\varepsilon}(y, 0)
	-l_{\varepsilon}(\varphi)-\frac{1}{2}\left\langle\mathcal{L}_{\varepsilon} \varphi, \varphi\right\rangle.
\end{equation}
We remark that $R_{\varepsilon}$ belongs to $C^{2}\left(H_{\varepsilon}\right)$ since so does every term on the right hand side of \eqref{eq3.3}.
In the rest of this section, we consider $l_{\varepsilon}: H_{\varepsilon} \rightarrow \mathbb{R}$ and $R_{\varepsilon}: H_{\varepsilon} \rightarrow \mathbb{R}$ and give some elementary estimates.
\par

\begin{Lem}\cite{MR2595734}\label{Lem3.2}
	For any constant $0<\sigma \leq \min \{\alpha, \beta\}$, there is a constant $C>0$ such that
	\begin{equation*}
		\frac{1}{\left(1+\left|y-x^{i}\right|\right)^{\alpha}}
		\frac{1}{\left(1+\left|y-x^{j}\right|\right)^{\beta}}
		\leq \frac{C}{\left|x^{i}-x^{j}\right|^{\sigma}}\left(\frac{1}{\left(1+\left|y-x^{i}\right|\right)^{\alpha+\beta-\sigma}}
		+\frac{1}{\left(1+\left|y-x^{j}\right|\right)^{\alpha+\beta-\sigma}}\right)
	\end{equation*}
	where $\alpha, \beta>0$ are two constants.
\end{Lem}

\begin{Lem}\label{Lem3.3}
	Assume that $V$ satisfies $(V_1)$ and $(V_2)$, and let $\alpha$ be the H\"older exponent in $(V_2)$. Then there exists a constant $C>0$, independent of $\varepsilon$, such that for any 
	\[
	y=(y^1,\ldots,y^k)\in D_{\varepsilon,\delta}
	\]
	there holds
	\begin{equation*}
		\begin{aligned}
			\left|l_{\varepsilon}(\varphi)\right|
			&\leq C \varepsilon^{\frac{N}{2}}
			\Bigg(
			\varepsilon^{\alpha}
			+\sum_{i=1}^{k}\big|V(y^{i})-V(a_i)\big|
			\\
			&\qquad\qquad
			+\sum_{i\neq j}
			\begin{cases}
				\displaystyle
				\frac{1}{\left|\frac{y^{i}-y^{j}}{\varepsilon}\right|^{N+2 s}},
				& \text{if } p>2, \\[6pt]
				\displaystyle
				\frac{1}{\left|\frac{y^{i}-y^{j}}{\varepsilon}\right|^{\frac{p}{2}(N+2 s)}},
				& \text{if } 1<p \leq 2
			\end{cases}
			\Bigg)
			\|\varphi\|_{\varepsilon}
		\end{aligned}
	\end{equation*}
	for all $\varphi \in H_{\varepsilon}$.
\end{Lem}

\begin{proof}
	Since $(U^1,\cdots, U^k)$ is the unique positive solution of the system \eqref{eq2.5} and $U_{\varepsilon,y}
	=\sum_{i=1}^{k}U^i_{\varepsilon,y^i}$, one can rewrite $l_\varepsilon$ as
	\begin{equation*}
		\begin{aligned}
			l_{\varepsilon}(\varphi)
			&= \left\langle I'_{\varepsilon}(U_{\varepsilon,y}),\varphi\right\rangle \\
			&= \int_{\mathbb{R}^{N}}\sum_{i=1}^{k}\big(V(x)-V(a_i)\big)\, U^{i}_{\varepsilon, y^{i}} \,\varphi\,dx
			+\int_{\mathbb{R}^{N}}\Bigg(\sum_{i=1}^{k}\big(U^{i}_{\varepsilon,y^{i}}\big)^{p}
			-\big(U_{\varepsilon,y}\big)^{p}\Bigg)\varphi\,dx.
		\end{aligned}
	\end{equation*}
	We estimate the two terms separately.
	
	We split
	\begin{equation*}
		\sum_{i=1}^{k}\big(V(x)-V(a_i)\big)U^{i}_{\varepsilon,y^{i}}
		=\sum_{i=1}^{k}\big(V(x)-V(y^{i})\big)U^{i}_{\varepsilon,y^{i}}
		+\sum_{i=1}^{k}\big(V(y^{i})-V(a_i)\big)U^{i}_{\varepsilon,y^{i}}.
	\end{equation*}
	Using Cauchy–Schwarz and the fact that $\|u\|_2 \le C\|u\|_\varepsilon$ (by $(V_1)$), we get
	\begin{equation*}
		\begin{aligned}
			&\left|\int_{\mathbb{R}^{N}}\sum_{i=1}^{k}\big(V(x)-V(y^{i})\big)U^{i}_{\varepsilon,y^{i}}\,\varphi\,dx\right| \\
			&\qquad\le
			\left\|\sum_{i=1}^{k}\big(V(x)-V(y^{i})\big)U^{i}_{\varepsilon,y^{i}}\right\|_{2}
			\,\|\varphi\|_{2}
			\le C
			\left\|\sum_{i=1}^{k}\big(V(x)-V(y^{i})\big)U^{i}_{\varepsilon,y^{i}}\right\|_{2}
			\|\varphi\|_{\varepsilon}.
		\end{aligned}
	\end{equation*}
	Since the peaks are well separated, the $L^2$–norm of the sum is controlled by the sum of the norms, hence
	\begin{equation*}
		\left\|\sum_{i=1}^{k}\big(V(x)-V(y^{i})\big)U^{i}_{\varepsilon,y^{i}}\right\|_{2}
		\le C\sum_{i=1}^{k}
		\left\|\big(V(x)-V(y^{i})\big)U^{i}_{\varepsilon,y^{i}}\right\|_{2}.
	\end{equation*}
	Fix $i$ and set $x=\varepsilon z+y^{i}$. Then
	\begin{equation*}
		\begin{aligned}
			\int_{\mathbb{R}^{N}}
			\big|V(x)-V(y^{i})\big|^{2}\big|U^{i}_{\varepsilon,y^{i}}(x)\big|^{2}dx
			&=\varepsilon^{N}\int_{\mathbb{R}^{N}}
			\big|V(\varepsilon z+y^{i})-V(y^{i})\big|^{2}|U^{i}(z)|^{2}dz.
		\end{aligned}
	\end{equation*}
	By $(V_2)$, $V$ is $C^{\alpha}$ in a neighbourhood of $a_i$, hence for $|z|$ bounded and $\varepsilon$ small
	\[
	\big|V(\varepsilon z+y^{i})-V(y^{i})\big|
	\le C\varepsilon^{\alpha}|z|^{\alpha}.
	\]
	Using the decay $|U^{i}(z)|\le C(1+|z|^{N+2s})^{-1}$ and $\alpha<\frac{N+4s}{2}$, one checks that
	\[
	\int_{\mathbb{R}^{N}}|z|^{2\alpha}|U^{i}(z)|^{2}dz<\infty,
	\]
	so
	\begin{equation*}
		\int_{\mathbb{R}^{N}}
		\big|V(x)-V(y^{i})\big|^{2}\big|U^{i}_{\varepsilon,y^{i}}(x)\big|^{2}dx
		\le C\,\varepsilon^{N+2\alpha}.
	\end{equation*}
	Taking square roots and summing in $i$, we obtain
	\[
	\left\|\sum_{i=1}^{k}\big(V(x)-V(y^{i})\big)U^{i}_{\varepsilon,y^{i}}\right\|_{2}
	\le C\varepsilon^{\frac{N}{2}+\alpha},
	\]
	and hence
	\begin{equation}\label{eq:V-Hoelder}
		\left|\int_{\mathbb{R}^{N}}\sum_{i=1}^{k}\big(V(x)-V(y^{i})\big)U^{i}_{\varepsilon,y^{i}}\,\varphi\,dx\right|
		\le C\varepsilon^{\frac{N}{2}+\alpha}\|\varphi\|_{\varepsilon}.
	\end{equation}
	
	For the second part,
	\begin{equation*}
		\begin{aligned}
			\left|\int_{\mathbb{R}^{N}}\sum_{i=1}^{k}\big(V(y^{i})-V(a_i)\big)U^{i}_{\varepsilon,y^{i}}\,\varphi\,dx\right|
			&\le C\sum_{i=1}^{k}\big|V(y^{i})-V(a_{i})\big|
			\left\|U^{i}_{\varepsilon,y^{i}}\right\|_{2}\|\varphi\|_{\varepsilon}.
		\end{aligned}
	\end{equation*}
	Since $\|U^{i}_{\varepsilon,y^{i}}\|_{2}^{2}
	=\varepsilon^{N}\|U^{i}\|_{2}^{2}$, we get
	\begin{equation}\label{eq:V-ai}
		\left|\int_{\mathbb{R}^{N}}\sum_{i=1}^{k}\big(V(y^{i})-V(a_i)\big)U^{i}_{\varepsilon,y^{i}}\,\varphi\,dx\right|
		\le C\varepsilon^{\frac{N}{2}}
		\sum_{i=1}^{k}\big|V(y^{i})-V(a_i)\big|\|\varphi\|_{\varepsilon}.
	\end{equation}
	
	Combining \eqref{eq:V-Hoelder} and \eqref{eq:V-ai} we conclude that
	\begin{equation}\label{eq3.6-correct}
		\left|\int_{\mathbb{R}^{N}}\sum_{i=1}^{k}\big(V(x)-V(a_i)\big)U^{i}_{\varepsilon,y^{i}}\,\varphi\,dx\right|
		\le C\varepsilon^{\frac{N}{2}+\alpha}\|\varphi\|_{\varepsilon}
		+C\varepsilon^{\frac{N}{2}}
		\sum_{i=1}^{k}\big|V(y^{i})-V(a_i)\big|\|\varphi\|_{\varepsilon}.
	\end{equation}
	
	Set $a_i(x)=U^{i}_{\varepsilon,y^{i}}(x)\ge0$ and $A(x)=\sum_{i=1}^{k}a_i(x)=U_{\varepsilon,y}(x)$. Then
	\[
	A^{p}-\sum_{i=1}^{k}a_i^{p}
	\]
	contains only mixed terms (interactions between different peaks). It is standard that
	\[
	\big|A^{p}-\sum_{i=1}^{k}a_i^{p}\big|
	\le 
	\begin{cases}
		C\displaystyle\sum_{i\neq j}\big(a_i^{p-1}a_j+a_i a_j^{p-1}\big), & p>2,\\[4pt]
		C\displaystyle\sum_{i\neq j} (a_i a_j)^{\frac{p}{2}}, & 1<p\le2.
	\end{cases}
	\]
	Therefore
	\begin{equation*}
		\left|\int_{\mathbb{R}^{N}}\Bigg(\sum_{i=1}^{k}(U^{i}_{\varepsilon,y^{i}})^{p}
		-U_{\varepsilon,y}^{\,p}\Bigg)\varphi\,dx\right|
		\le
		\begin{cases}
			C\displaystyle\sum_{i\neq j}\int_{\mathbb{R}^{N}}
			\big((U^{i}_{\varepsilon,y^{i}})^{p-1}U^{j}_{\varepsilon,y^{j}}
			+U^{i}_{\varepsilon,y^{i}}(U^{j}_{\varepsilon,y^{j}})^{p-1}\big)|\varphi|\,dx,
			& p>2,\\[6pt]
			C\displaystyle\sum_{i\neq j}\int_{\mathbb{R}^{N}}
			(U^{i}_{\varepsilon,y^{i}}U^{j}_{\varepsilon,y^{j}})^{\frac{p}{2}}|\varphi|\,dx,
			& 1<p\le 2.
		\end{cases}
	\end{equation*}
	
	\textbf{Case $p>2$:}
	Using Cauchy–Schwarz and $\|\varphi\|_{2}\le C\|\varphi\|_\varepsilon$,
	\begin{equation*}
		\begin{aligned}
			&\left|\int_{\mathbb{R}^{N}}\sum_{i\neq j}
			\big((U^{i}_{\varepsilon,y^{i}})^{p-1}U^{j}_{\varepsilon,y^{j}}
			+U^{i}_{\varepsilon,y^{i}}(U^{j}_{\varepsilon,y^{j}})^{p-1}\big)\varphi\,dx\right| \\
			&\quad\le C
			\left(\sum_{i\neq j}\int_{\mathbb{R}^{N}}
			(U^{i}_{\varepsilon,y^{i}})^{2(p-1)}(U^{j}_{\varepsilon,y^{j}})^{2}\,dx\right)^{\frac12}
			\|\varphi\|_{2} \\
			&\quad\le C
			\left(\sum_{i\neq j}\int_{\mathbb{R}^{N}}
			(U^{i}_{\varepsilon,y^{i}})^{2(p-1)}(U^{j}_{\varepsilon,y^{j}})^{2}\,dx\right)^{\frac12}
			\|\varphi\|_{\varepsilon}.
		\end{aligned}
	\end{equation*}
	Fix $i\neq j$ and set $x=\varepsilon z+y^{i}$, so that
	\begin{equation*}
		\begin{aligned}
			\int_{\mathbb{R}^{N}}
			(U^{i}_{\varepsilon,y^{i}})^{2(p-1)}(U^{j}_{\varepsilon,y^{j}})^{2}\,dx
			&=\varepsilon^{N}\int_{\mathbb{R}^{N}}
			|U^{i}(z)|^{2(p-1)}
			\big|U^{j}\big(z+d_{ij}\big)\big|^{2}dz,
		\end{aligned}
	\end{equation*}
	where
	\[
	d_{ij}=\frac{y^{i}-y^{j}}{\varepsilon}.
	\]
	Using the decay estimate from Proposition \ref{Pro1.1} $(iii)$,
	\[
	|U^{\ell}(z)|\le\frac{C}{1+|z|^{N+2s}},\qquad \ell=i,j,
	\]
	we obtain
	\begin{equation*}
		\int_{\mathbb{R}^{N}}
		(U^{i}_{\varepsilon,y^{i}})^{2(p-1)}(U^{j}_{\varepsilon,y^{j}})^{2}\,dx
		\le C\varepsilon^{N}
		\int_{\mathbb{R}^{N}}\frac{1}{(1+|z|)^{2(p-1)(N+2s)}}
		\frac{1}{(1+|z+d_{ij}|)^{2(N+2s)}}dz.
	\end{equation*}
	Applying Lemma \ref{Lem3.2} with
	\[
	\alpha=2(p-1)(N+2s),\quad \beta=2(N+2s),\quad \sigma=2(N+2s)
	\]
	we deduce that
	\[
	\int_{\mathbb{R}^{N}}
	(U^{i}_{\varepsilon,y^{i}})^{2(p-1)}(U^{j}_{\varepsilon,y^{j}})^{2}\,dx
	\le C\varepsilon^{N}\frac{1}{|d_{ij}|^{2(N+2s)}}.
	\]
	Hence
	\begin{equation*}
		\left(\sum_{i\neq j}\int_{\mathbb{R}^{N}}
		(U^{i}_{\varepsilon,y^{i}})^{2(p-1)}(U^{j}_{\varepsilon,y^{j}})^{2}\,dx\right)^{\frac12}
		\le C\varepsilon^{\frac{N}{2}}
		\sum_{i\neq j}\frac{1}{\left|\frac{y^{i}-y^{j}}{\varepsilon}\right|^{N+2s}}.
	\end{equation*}
	Combining the above estimates gives
	\begin{equation}\label{eq3.7-correct}
		\left|\int_{\mathbb{R}^{N}}\Bigg(\sum_{i=1}^{k}(U^{i}_{\varepsilon,y^{i}})^{p}
		-U_{\varepsilon,y}^{\,p}\Bigg)\varphi\,dx\right|
		\le C\varepsilon^{\frac{N}{2}}
		\sum_{i\neq j}\frac{1}{\left|\frac{y^{i}-y^{j}}{\varepsilon}\right|^{N+2s}}
		\|\varphi\|_{\varepsilon},
		\quad p>2.
	\end{equation}
	
	\textbf{Case $1<p\le2$:}
	Similarly,
	\begin{equation*}
		\begin{aligned}
			&\left|\int_{\mathbb{R}^{N}}\sum_{i\neq j}
			(U^{i}_{\varepsilon,y^{i}}U^{j}_{\varepsilon,y^{j}})^{\frac{p}{2}}\varphi\,dx\right|
			\le C\sum_{i\neq j}
			\left(\int_{\mathbb{R}^{N}}
			(U^{i}_{\varepsilon,y^{i}}U^{j}_{\varepsilon,y^{j}})^{p}\,dx\right)^{\frac12}
			\|\varphi\|_{2} \\
			&\qquad\le C\sum_{i\neq j}
			\left(\int_{\mathbb{R}^{N}}
			(U^{i}_{\varepsilon,y^{i}}U^{j}_{\varepsilon,y^{j}})^{p}\,dx\right)^{\frac12}
			\|\varphi\|_{\varepsilon}.
		\end{aligned}
	\end{equation*}
	Arguing as above and using Lemma \ref{Lem3.2} with
	\(\alpha=\beta=p(N+2s)\) and \(\sigma=p(N+2s)\), we obtain
	\[
	\int_{\mathbb{R}^{N}}
	(U^{i}_{\varepsilon,y^{i}}U^{j}_{\varepsilon,y^{j}})^{p}\,dx
	\le C\varepsilon^{N}
	\frac{1}{\left|\frac{y^{i}-y^{j}}{\varepsilon}\right|^{p(N+2s)}},
	\]
	and hence
	\begin{equation}\label{eq3.8-correct}
		\left|\int_{\mathbb{R}^{N}}\Bigg(\sum_{i=1}^{k}(U^{i}_{\varepsilon,y^{i}})^{p}
		-U_{\varepsilon,y}^{\,p}\Bigg)\varphi\,dx\right|
		\le C\varepsilon^{\frac{N}{2}}
		\sum_{i\neq j}\frac{1}{\left|\frac{y^{i}-y^{j}}{\varepsilon}\right|^{\frac{p}{2}(N+2s)}}
		\|\varphi\|_{\varepsilon},
		\quad 1<p\le2.
	\end{equation}
	
	Combining \eqref{eq3.6-correct}, \eqref{eq3.7-correct} and \eqref{eq3.8-correct}, we obtain the desired estimate for $l_\varepsilon(\varphi)$. The proof is complete.
\end{proof}

\begin{Lem}\label{Lem3.4}
	There exists a constant $C>0$, independent of $\varepsilon$ and $b$, such that for every
	$i\in\{0,1,2\}$ and every $\varphi\in H_\varepsilon$ one has
	\[
	\big\|R_\varepsilon^{(i)}(\varphi)\big\|
	\;\le\;
	C\,\varepsilon^{-\frac{N(p-1)}{2}}\|\varphi\|_\varepsilon^{\,p+1-i}
	\;+\;
	C\,(b+1)\,\varepsilon^{-\frac N2}
	\bigl(1+\varepsilon^{-\frac N2}\|\varphi\|_\varepsilon\bigr)
	\|\varphi\|_\varepsilon^{\,3-i}.
	\]
	Here $\|\cdot\|$ denotes the operator norm of $R_\varepsilon^{(i)}(\varphi)$ as a multilinear form.
\end{Lem}

\begin{proof}
	Recall from \eqref{eq3.3} that
	\[
	R_{\varepsilon}(\varphi)=A_{1}(\varphi)-A_{2}(\varphi),
	\]
	where
	\[
	A_{1}(\varphi)
	=\frac{b\varepsilon^{4s-N}}{4}\left[
	\left(\int_{\R^{N}}\big|(-\Delta)^{\frac{s}{2}}\varphi\big|^{2}dx\right)^{2}
	+4\int_{\R^{N}}\big|(-\Delta)^{\frac{s}{2}}\varphi\big|^{2}dx
	\int_{\R^{N}}(-\Delta)^{\frac{s}{2}}U_{\varepsilon,y}\cdot(-\Delta)^{\frac{s}{2}}\varphi\,dx
	\right]
	\]
	and
	\[
	A_{2}(\varphi)
	=\frac{1}{p+1}\int_{\R^{N}}\Big[
	\big(U_{\varepsilon,y}+\varphi\big)^{p+1}
	-U_{\varepsilon,y}^{p+1}
	-(p+1)U_{\varepsilon,y}^{p}\varphi
	-\frac{p(p+1)}{2}U_{\varepsilon,y}^{p-1}\varphi^{2}
	\Big]\,dx.
	\]
	Here and in the sequel we abbreviate \(U_{\varepsilon,y}=U_{\varepsilon,y}(x)\).
	We denote by \(R_{\varepsilon}^{(i)}\) the \(i\)-th Fr\'echet derivative of \(R_{\varepsilon}\),
	and similarly for \(A_{1}\) and \(A_{2}\).
	
	We first estimate the Kirchhoff part \(A_{1}\).
	
	\medskip
	\noindent\textbf{Step 1: Estimates for \(A_{1}^{(i)}\).}
	Set
	\[
	S(\varphi)=\int_{\R^{N}}\big|(-\Delta)^{\frac{s}{2}}\varphi\big|^{2}dx,
	\qquad
	T(\varphi)=\int_{\R^{N}}(-\Delta)^{\frac{s}{2}}U_{\varepsilon,y}\cdot
	(-\Delta)^{\frac{s}{2}}\varphi\,dx.
	\]
	Then
	\[
	A_{1}(\varphi)
	=\frac{b\varepsilon^{4s-N}}{4}\Big[S(\varphi)^{2}+4S(\varphi)T(\varphi)\Big].
	\]
	
	By the definition of \(\|\cdot\|_{\varepsilon}\) we have
	\[
	\varepsilon^{2s}a\int_{\R^{N}}\big|(-\Delta)^{\frac{s}{2}}\varphi\big|^{2}dx
	\le \|\varphi\|_{\varepsilon}^{2},
	\]
	so
	\begin{equation}\label{eq:A1-S}
		S(\varphi)\le C\,\varepsilon^{-2s}\|\varphi\|_{\varepsilon}^{2}.
	\end{equation}
	Moreover, since \(U_{\varepsilon,y}(x)=U\big((x-y)/\varepsilon\big)\), a change of variables yields
	\[
	\big\|(-\Delta)^{\frac{s}{2}}U_{\varepsilon,y}\big\|_{2}
	=\varepsilon^{\frac{N-2s}{2}}\big\|(-\Delta)^{\frac{s}{2}}U\big\|_{2}
	\le C\,\varepsilon^{\frac{N-2s}{2}}.
	\]
	Thus,
	\begin{equation}\label{eq:A1-T}
		|T(\varphi)|
		\le \big\|(-\Delta)^{\frac{s}{2}}U_{\varepsilon,y}\big\|_{2}
		\big\|(-\Delta)^{\frac{s}{2}}\varphi\big\|_{2}
		\le C\,\varepsilon^{\frac{N-2s}{2}}\varepsilon^{-s}\|\varphi\|_{\varepsilon}
		=C\,\varepsilon^{\frac{N-4s}{2}}\|\varphi\|_{\varepsilon}.
	\end{equation}
	
	Using \eqref{eq:A1-S} and \eqref{eq:A1-T}, we obtain
	\[
	\begin{aligned}
		|A_{1}(\varphi)|
		&\le Cb\,\varepsilon^{4s-N}\Big(S(\varphi)^{2}+|S(\varphi)T(\varphi)|\Big)\\
		&\le Cb\,\varepsilon^{4s-N}
		\Big(\varepsilon^{-4s}\|\varphi\|_{\varepsilon}^{4}
		+\varepsilon^{-2s}\varepsilon^{\frac{N-4s}{2}}\|\varphi\|_{\varepsilon}^{3}\Big)\\
		&\le Cb\Big(\varepsilon^{-N}\|\varphi\|_{\varepsilon}^{4}
		+\varepsilon^{-\frac{N}{2}}\|\varphi\|_{\varepsilon}^{3}\Big)\\
		&\le Cb\,\varepsilon^{-\frac{N}{2}}
		\bigl(1+\varepsilon^{-\frac{N}{2}}\|\varphi\|_{\varepsilon}\bigr)
		\|\varphi\|_{\varepsilon}^{3}.
	\end{aligned}
	\]
	This gives the desired estimate for \(i=0\), up to the factor \(b\).
	
	Next we estimate the first and second derivatives. A direct computation gives,
	for any \(\varphi,\psi\in H_{\varepsilon}\),
	\[
	\begin{aligned}
		\big\langle A_{1}^{(1)}(\varphi),\psi\big\rangle
		=&\,b\varepsilon^{4s-N}
		S(\varphi)\int_{\R^{N}}(-\Delta)^{\frac{s}{2}}\varphi\cdot(-\Delta)^{\frac{s}{2}}\psi\,dx\\
		&+b\varepsilon^{4s-N}
		S(\varphi)\int_{\R^{N}}(-\Delta)^{\frac{s}{2}}U_{\varepsilon,y}\cdot(-\Delta)^{\frac{s}{2}}\psi\,dx\\
		&+2b\varepsilon^{4s-N}
		T(\varphi)\int_{\R^{N}}(-\Delta)^{\frac{s}{2}}\varphi\cdot(-\Delta)^{\frac{s}{2}}\psi\,dx.
	\end{aligned}
	\]
	Using \eqref{eq:A1-S}, \eqref{eq:A1-T} and
	\[
	\int_{\R^{N}}\big|(-\Delta)^{\frac{s}{2}}\varphi\cdot(-\Delta)^{\frac{s}{2}}\psi\big|dx
	\le \big\|(-\Delta)^{\frac{s}{2}}\varphi\big\|_{2}
	\big\|(-\Delta)^{\frac{s}{2}}\psi\big\|_{2}
	\le C\,\varepsilon^{-2s}\|\varphi\|_{\varepsilon}\|\psi\|_{\varepsilon},
	\]
	together with
	\[
	\int_{\R^{N}}\big|(-\Delta)^{\frac{s}{2}}U_{\varepsilon,y}\cdot(-\Delta)^{\frac{s}{2}}\psi\big|dx
	\le C\,\varepsilon^{\frac{N-4s}{2}}\|\psi\|_{\varepsilon},
	\]
	we obtain
	\[
	\big|\langle A_{1}^{(1)}(\varphi),\psi\rangle\big|
	\le Cb\,\varepsilon^{-\frac{N}{2}}
	\bigl(1+\varepsilon^{-\frac{N}{2}}\|\varphi\|_{\varepsilon}\bigr)
	\|\varphi\|_{\varepsilon}^{2}\|\psi\|_{\varepsilon},
	\]
	and hence
	\[
	\big\|A_{1}^{(1)}(\varphi)\big\|
	\le Cb\,\varepsilon^{-\frac{N}{2}}
	\bigl(1+\varepsilon^{-\frac{N}{2}}\|\varphi\|_{\varepsilon}\bigr)
	\|\varphi\|_{\varepsilon}^{2}.
	\]
	
	Similarly, for any \(\varphi,\psi,\xi\in H_{\varepsilon}\), the explicit formula
	for \(A_{1}^{(2)}(\varphi)[\psi]\) shows that each term is a product of two integrals of the form
	\(\int(-\Delta)^{\frac{s}{2}}\cdot(-\Delta)^{\frac{s}{2}}\cdot\), multiplied by
	\(\varepsilon^{4s-N}\). Using again \eqref{eq:A1-S}, \eqref{eq:A1-T} and the above
	estimates, we obtain
	\[
	\big|\langle A_{1}^{(2)}(\varphi)[\psi],\xi\rangle\big|
	\le Cb\,\varepsilon^{-\frac{N}{2}}
	\bigl(1+\varepsilon^{-\frac{N}{2}}\|\varphi\|_{\varepsilon}\bigr)
	\|\varphi\|_{\varepsilon}\|\psi\|_{\varepsilon}\|\xi\|_{\varepsilon},
	\]
	so that
	\[
	\big\|A_{1}^{(2)}(\varphi)\big\|
	\le Cb\,\varepsilon^{-\frac{N}{2}}
	\bigl(1+\varepsilon^{-\frac{N}{2}}\|\varphi\|_{\varepsilon}\bigr)
	\|\varphi\|_{\varepsilon}.
	\]
	Summarizing, for \(i=0,1,2\),
	\begin{equation}\label{eq:A1-final}
		\big\|A_{1}^{(i)}(\varphi)\big\|
		\le Cb\,\varepsilon^{-\frac{N}{2}}
		\bigl(1+\varepsilon^{-\frac{N}{2}}\|\varphi\|_{\varepsilon}\bigr)
		\|\varphi\|_{\varepsilon}^{3-i}.
	\end{equation}
	
	\medskip
	\noindent\textbf{Step 2: Estimates for \(A_{2}^{(i)}\).}
	We now estimate the nonlinearity part \(A_{2}\).
	We shall use the following standard algebraic inequalities: there exists
	\(C>0\) such that for all \(a,b\in\R\),
	\begin{align}
		\big|(a+b)^{p+1}-a^{p+1}-(p+1)a^{p}b-\tfrac{p(p+1)}{2}a^{p-1}b^{2}\big|
		&\le C\big(|b|^{p+1}+|a|^{p-2}|b|^{3}\big),\label{eq:alg0}\\
		\big|(a+b)^{p}-a^{p}-pa^{p-1}b\big|
		&\le C\big(|b|^{p}+|a|^{p-2}|b|^{2}\big),\label{eq:alg1}\\
		\big|(a+b)^{p-1}-a^{p-1}\big|
		&\le C\big(|b|^{p-1}+|a|^{p-2}|b|\big).\label{eq:alg2}
	\end{align}
	
	\vspace{1mm}
	\noindent\emph{Case \(i=0\).}
	By \eqref{eq:alg0} with \(a=U_{\varepsilon,y}(x)\) and \(b=\varphi(x)\),
	\[
	|A_{2}(\varphi)|
	\le C\int_{\R^{N}}\Big(|\varphi|^{p+1}+|U_{\varepsilon,y}|^{p-2}|\varphi|^{3}\Big)dx.
	\]
	Since \(U^{i}\) has polynomial decay, we have \(U_{\varepsilon,y}\in L^{\infty}(\R^{N})\)
	with \(\|U_{\varepsilon,y}\|_{\infty}\le C\), and therefore
	\[
	\int_{\R^{N}}|U_{\varepsilon,y}|^{p-2}|\varphi|^{3}dx
	\le C\int_{\R^{N}}|\varphi|^{3}dx.
	\]
	
	From Lemma~\ref{Lem2.1}, for any \(2\le q\le 2_{s}^{*}\),
	\[
	\|\varphi\|_{L^{q}(\R^{N})}
	\le C\,\varepsilon^{\frac{N}{q}-\frac{N}{2}}\|\varphi\|_{\varepsilon}.
	\]
	Taking \(q=p+1\) we obtain
	\[
	\int_{\R^{N}}|\varphi|^{p+1}dx
	\le C\,\varepsilon^{N-\frac{N(p+1)}{2}}\|\varphi\|_{\varepsilon}^{p+1}
	= C\,\varepsilon^{-\frac{N(p-1)}{2}}\|\varphi\|_{\varepsilon}^{p+1},
	\]
	and taking \(q=3\) yields
	\[
	\int_{\R^{N}}|\varphi|^{3}dx
	\le C\,\varepsilon^{3\big(\frac{N}{3}-\frac{N}{2}\big)}
	\|\varphi\|_{\varepsilon}^{3}
	= C\,\varepsilon^{-\frac{N}{2}}\|\varphi\|_{\varepsilon}^{3}.
	\]
	Hence
	\begin{equation}\label{eq:A2-0}
		|A_{2}(\varphi)|
		\le C\,\varepsilon^{-\frac{N(p-1)}{2}}\|\varphi\|_{\varepsilon}^{p+1}
		+ C\,\varepsilon^{-\frac{N}{2}}\|\varphi\|_{\varepsilon}^{3}.
	\end{equation}
	
	\vspace{1mm}
	\noindent\emph{Case \(i=1\).}
	For any \(\varphi,\psi\in H_{\varepsilon}\),
	\[
	\big\langle A_{2}^{(1)}(\varphi),\psi\big\rangle
	=\int_{\R^{N}}\Big[(U_{\varepsilon,y}+\varphi)^{p}-U_{\varepsilon,y}^{p}
	-pU_{\varepsilon,y}^{p-1}\varphi\Big]\psi\,dx.
	\]
	By \eqref{eq:alg1},
	\[
	\Big|(U_{\varepsilon,y}+\varphi)^{p}-U_{\varepsilon,y}^{p}
	-pU_{\varepsilon,y}^{p-1}\varphi\Big|
	\le C\big(|\varphi|^{p}+|U_{\varepsilon,y}|^{p-2}|\varphi|^{2}\big).
	\]
	Thus
	\[
	\big|\langle A_{2}^{(1)}(\varphi),\psi\rangle\big|
	\le C\int_{\R^{N}}|\varphi|^{p}|\psi|dx
	+ C\int_{\R^{N}}|U_{\varepsilon,y}|^{p-2}|\varphi|^{2}|\psi|dx.
	\]
	Using H\"older's inequality with exponents \(\frac{p+1}{p}\) and \(p+1\), and Lemma~\ref{Lem2.1},
	\[
	\int_{\R^{N}}|\varphi|^{p}|\psi|dx
	\le \|\varphi\|_{p+1}^{p}\|\psi\|_{p+1}
	\le C\,\varepsilon^{-\frac{N(p-1)}{2}}\|\varphi\|_{\varepsilon}^{p}\|\psi\|_{\varepsilon}.
	\]
	Since \(U_{\varepsilon,y}\in L^{\infty}\),
	\[
	\int_{\R^{N}}|U_{\varepsilon,y}|^{p-2}|\varphi|^{2}|\psi|dx
	\le C\int_{\R^{N}}|\varphi|^{2}|\psi|dx
	\le C\|\varphi\|_{3}^{2}\|\psi\|_{3}
	\le C\,\varepsilon^{-\frac{N}{2}}\|\varphi\|_{\varepsilon}^{2}\|\psi\|_{\varepsilon}.
	\]
	Hence
	\begin{equation}\label{eq:A2-1}
		\big\|A_{2}^{(1)}(\varphi)\big\|
		\le C\,\varepsilon^{-\frac{N(p-1)}{2}}\|\varphi\|_{\varepsilon}^{p}
		+ C\,\varepsilon^{-\frac{N}{2}}\|\varphi\|_{\varepsilon}^{2}.
	\end{equation}
	
	\vspace{1mm}
	\noindent\emph{Case \(i=2\).}
	For any \(\varphi,\psi,\xi\in H_{\varepsilon}\),
	\[
	\big\langle A_{2}^{(2)}(\varphi)[\psi],\xi\big\rangle
	=p\int_{\R^{N}}\Big[(U_{\varepsilon,y}+\varphi)^{p-1}-U_{\varepsilon,y}^{p-1}\Big]\psi\xi\,dx.
	\]
	By \eqref{eq:alg2},
	\[
	\Big|(U_{\varepsilon,y}+\varphi)^{p-1}-U_{\varepsilon,y}^{p-1}\Big|
	\le C\big(|\varphi|^{p-1}+|U_{\varepsilon,y}|^{p-2}|\varphi|\big),
	\]
	and therefore
	\[
	\big|\langle A_{2}^{(2)}(\varphi)[\psi],\xi\rangle\big|
	\le C\int_{\R^{N}}|\varphi|^{p-1}|\psi||\xi|dx
	+ C\int_{\R^{N}}|U_{\varepsilon,y}|^{p-2}|\varphi||\psi||\xi|dx.
	\]
	Using H\"older's inequality with exponents
	\(\frac{p+1}{p-1},p+1,p+1\) and Lemma~\ref{Lem2.1},
	\[
	\int_{\R^{N}}|\varphi|^{p-1}|\psi||\xi|dx
	\le \|\varphi\|_{p+1}^{p-1}\|\psi\|_{p+1}\|\xi\|_{p+1}
	\le C\,\varepsilon^{-\frac{N(p-1)}{2}}\|\varphi\|_{\varepsilon}^{p-1}
	\|\psi\|_{\varepsilon}\|\xi\|_{\varepsilon}.
	\]
	Again using \(U_{\varepsilon,y}\in L^{\infty}\),
	\[
	\int_{\R^{N}}|U_{\varepsilon,y}|^{p-2}|\varphi||\psi||\xi|dx
	\le C\int_{\R^{N}}|\varphi||\psi||\xi|dx
	\le C\|\varphi\|_{3}\|\psi\|_{3}\|\xi\|_{3}
	\le C\,\varepsilon^{-\frac{N}{2}}\|\varphi\|_{\varepsilon}\|\psi\|_{\varepsilon}\|\xi\|_{\varepsilon}.
	\]
	Hence
	\begin{equation}\label{eq:A2-2}
		\big\|A_{2}^{(2)}(\varphi)\big\|
		\le C\,\varepsilon^{-\frac{N(p-1)}{2}}\|\varphi\|_{\varepsilon}^{p-1}
		+ C\,\varepsilon^{-\frac{N}{2}}\|\varphi\|_{\varepsilon}.
	\end{equation}
	
	\medskip
	\noindent\textbf{Step 3: Conclusion.}
	Combining \eqref{eq:A2-0}, \eqref{eq:A2-1}, \eqref{eq:A2-2} we obtain, for \(i=0,1,2\),
	\[
	\big\|A_{2}^{(i)}(\varphi)\big\|
	\le C\,\varepsilon^{-\frac{N(p-1)}{2}}\|\varphi\|_{\varepsilon}^{p+1-i}
	+ C\,\varepsilon^{-\frac{N}{2}}\|\varphi\|_{\varepsilon}^{3-i}.
	\]
	Together with \eqref{eq:A1-final} and the relation \(R_{\varepsilon}=A_{1}-A_{2}\),
	this yields
	\[
	\big\|R_{\varepsilon}^{(i)}(\varphi)\big\|
	\le C\,\varepsilon^{-\frac{N(p-1)}{2}}\|\varphi\|_{\varepsilon}^{p+1-i}
	+ C(b+1)\,\varepsilon^{-\frac{N}{2}}
	\bigl(1+\varepsilon^{-\frac{N}{2}}\|\varphi\|_{\varepsilon}\bigr)
	\|\varphi\|_{\varepsilon}^{3-i},
	\]
	for all \(i\in\{0,1,2\}\) and all \(\varphi\in H_{\varepsilon}\). This proves the lemma.
\end{proof}

Now we will give the energy expansion for the approximate solutions.

\begin{Lem}\label{Lem3.5}
	Assume that $V$ satisfies $(V_1)$ and $(V_2)$. Then, for $\varepsilon>0$ sufficiently small, there exist a small constant $\tau>0$ and $C>0$ such that
	\begin{equation*}
		\begin{aligned}
			I_{\varepsilon}\big(U_{\varepsilon, y}\big)
			=&\ A \varepsilon^{N}
			+ \varepsilon^{N}\sum_{i=1}^{k}B_i\big(V(y^{i})-V(a_i)\big)
			- C\,\varepsilon^{N} \sum_{i \neq j} \frac{1}{\left|\frac{y^{i}-y^{j}}{\varepsilon}\right| ^{N+2 s}} \\
			&\quad+O\!\left(\varepsilon^{N+\alpha}
			+\varepsilon^{N} \sum_{i \neq j} \frac{1}{\left|\frac{y^{i}-y^{j}}{\varepsilon}\right|^{N+2 s+\tau}}\right),
		\end{aligned}
	\end{equation*}
	where
	\begin{equation*}
		A=\Big(\frac{1}{2}-\frac{1}{p+1}\Big)\sum_{i=1}^{k}\int_{\mathbb{R}^{N}} |U^i|^{p+1}\,dx
		+\frac{b}{4}\left(\sum_{i=1}^{k}\int_{\mathbb{R}^{N}}\big|(-\Delta)^{\frac{s}{2}} U^i\big|^{2}\,dx\right)^{2}
	\end{equation*}
	and
	\begin{equation*}
		B_i=\frac{1}{2} \int_{\mathbb{R}^{N}} |U^i|^{2}\,dx,\qquad i=1,\dots,k.
	\end{equation*}
\end{Lem}

\begin{proof}
	By direct computation and by using that $(U^{1},\dots,U^{k})$ solves \eqref{eq2.5}, one has
	\begin{equation}\label{eq3.9}
		\begin{aligned}
			I_{\varepsilon}\big(U_{\varepsilon, y}\big)
			=&\ \frac{1}{2} \int_{\mathbb{R}^{N}}\Big(\varepsilon^{2s} a\big|(-\Delta)^{\frac{s}{2}} U_{\varepsilon, y}\big|^{2}
			+V(x) U_{\varepsilon, y}^{2}\Big)\,dx
			+\frac{\varepsilon^{4s-N} b}{4}\left(\int_{\mathbb{R}^{N}}\big|(-\Delta)^{\frac{s}{2}} U_{\varepsilon, y}\big|^{2}dx\right)^{2} \\
			&\ -\frac{1}{p+1} \int_{\mathbb{R}^{N}} U_{\varepsilon, y}^{p+1}\,dx \\
			=&\ \frac{1}{2}\sum_{i=1}^{k}\int_{\mathbb{R}^{N}}\big(V(x)-V(a_i)\big) U_{\varepsilon, y}^{2}\,dx
			+\frac{1}{2}\int_{\mathbb{R}^{N}}\sum_{i=1}^{k}\sum_{j=1}^{k} U_{\varepsilon, y^{j}}^{p} U_{\varepsilon, y^{i}}\,dx
			-\frac{1}{p+1} \int_{\mathbb{R}^{N}} U_{\varepsilon, y}^{p+1}\,dx \\
			=&\ \frac{1}{2}\sum_{i=1}^{k}\int_{\mathbb{R}^{N}}\big(V(x)-V(a_i)\big) U_{\varepsilon, y}^{2}\,dx
			+\frac{1}{2}\int_{\mathbb{R}^{N}}\left( \sum_{i=1}^{k} U_{\varepsilon, y^{i}}^{p+1}
			+\sum_{i \neq j} U_{\varepsilon, y^{i}}^{p} U_{\varepsilon, y^{j}}\right)dx \\
			&\ -\frac{1}{p+1} \int_{\mathbb{R}^{N}} U_{\varepsilon, y}^{p+1}\,dx.
		\end{aligned}
	\end{equation}
	We now estimate each term on the right-hand side of \eqref{eq3.9}.
	
	\medskip
	\noindent\textbf{Step 1: The potential term.}
	We first handle
	\[
	\int_{\mathbb{R}^{N}}\big(V(x)-V(a_i)\big) U_{\varepsilon, y}^{2}\,dx.
	\]
	Using the decomposition
	\[
	U_{\varepsilon,y}
	=\sum_{i=1}^{k}U_{\varepsilon,y^{i}}^{i},
	\qquad
	U_{\varepsilon,y^{i}}^{i}(x)=U^{i}\Big(\frac{x-y^{i}}{\varepsilon}\Big),
	\]
	we obtain
	\begin{equation}\label{eq3.10}
		\begin{aligned}
			\int_{\mathbb{R}^{N}}\big(V(x)-V(a_i)\big) U_{\varepsilon, y}^{2}\,dx
			=&\ \int_{\mathbb{R}^{N}}\big(V(x)-V(a_i)\big)
			\left(\sum_{i=1}^{k}\big(U_{\varepsilon, y^{i}}^{i}\big)^{2}
			+\sum_{i\neq j}U_{\varepsilon,y^{i}}^{i}U_{\varepsilon,y^{j}}^{j}\right)dx\\
			=&\ \int_{\mathbb{R}^{N}}\sum_{i=1}^{k}\big(V(x)-V(y^{i})+V(y^{i})-V(a_i)\big)\big(U_{\varepsilon, y^{i}}^{i}\big)^{2}dx \\
			&\ +\int_{\mathbb{R}^{N}}\sum_{i\neq j}\big(V(x)-V(y^{i})+V(y^{i})-V(a_i)\big)
			U_{\varepsilon,y^{i}}^{i}U_{\varepsilon,y^{j}}^{j}\,dx \\
			=&\ -\varepsilon^{N}\sum_{i=1}^{k}\big(V(a_i)-V(y^{i})\big)\int_{\mathbb{R}^{N}}(U^i)^{2}dx \\
			&\ +\int_{\mathbb{R}^{N}}\sum_{i=1}^{k}\big(V(x)-V(y^{i})\big)\big(U_{\varepsilon,y^{i}}^{i}\big)^{2}dx \\
			&\ +\int_{\mathbb{R}^{N}}\sum_{i\neq j}\big(V(x)-V(y^{i})+V(y^{i})-V(a_i)\big)
			U_{\varepsilon,y^{i}}^{i}U_{\varepsilon,y^{j}}^{j}\,dx.
		\end{aligned}
	\end{equation}
	By $(V_2)$ and the Hölder continuity of $V$ near $a_i$, we have
	\begin{equation}\label{eq3.11}
		\begin{aligned}
			\int_{\mathbb{R}^{N}}\sum_{i=1}^{k}\big(V(x)-V(y^{i})\big)\big(U_{\varepsilon,y^{i}}^{i}\big)^{2}dx
			=&\ \varepsilon^{N}\int_{\mathbb{R}^{N}}\sum_{i=1}^{k}\big(V(\varepsilon y+y^{i})-V(y^{i})\big)(U^i)^{2}(y)\,dy \\
			\leq&\ C\varepsilon^{N+\alpha}
			\int_{\mathbb{R}^{N}}\frac{|y|^{\alpha}}{1+|y|^{2(N+2s)}}\,dy \\
			\leq&\ C\varepsilon^{N+\alpha}.
		\end{aligned}
	\end{equation}
	For the mixed term with $i\neq j$ we write
	\begin{equation}\label{eq3.12}
		\begin{aligned}
			&\int_{\mathbb{R}^{N}}\big(V(x)-V(y^{i})\big)U_{\varepsilon,y^{i}}^{i}U_{\varepsilon,y^{j}}^{j}\,dx \\
			&\qquad=\varepsilon^{N}\int_{\mathbb{R}^{N}}\big(V(\varepsilon y+y^{i})-V(y^{i})\big)
			U^i(y)U^j\!\left(y-\frac{y^{i}-y^{j}}{\varepsilon}\right)dy.
		\end{aligned}
	\end{equation}
	Using again the Hölder continuity of $V$ and the polynomial decay of $U^i$, we obtain
	\[
	\big|V(\varepsilon y+y^{i})-V(y^{i})\big|
	\le C|\varepsilon y|^{\alpha},
	\]
	and hence
	\[
	\left|\int_{\mathbb{R}^{N}}\big(V(x)-V(y^{i})\big)U_{\varepsilon,y^{i}}^{i}U_{\varepsilon,y^{j}}^{j}\,dx\right|
	\le C\varepsilon^{N+\alpha}\int_{\mathbb{R}^{N}}\frac{|y|^{\alpha}}{(1+|y|^{N+2s})\big(1+\big|y-\frac{y^{i}-y^{j}}{\varepsilon}\big|^{N+2s}\big)}\,dy.
	\]
	Applying Lemma~\ref{Lem3.2} with suitable exponents and using the separation
	\[
	\frac{|y^{i}-y^{j}|}{\varepsilon}\ge \varepsilon^{\theta-1},\qquad \theta\in\Big(\frac{N+2s}{N+2s+\alpha},1\Big),
	\]
	we deduce that there exists a small $\tau>0$ such that
	\begin{equation}\label{eq3.12b}
		\left|\int_{\mathbb{R}^{N}}\big(V(x)-V(y^{i})\big)U_{\varepsilon,y^{i}}^{i}U_{\varepsilon,y^{j}}^{j}\,dx\right|
		\le C\,\varepsilon^{N+\alpha+\tau}.
	\end{equation}
	
	Similarly,
	\begin{equation}\label{eq3.13}
		\begin{aligned}
			\int_{\mathbb{R}^{N}}\big(V(a_i)-V(y^{i})\big)U_{\varepsilon,y^{i}}^{i}U_{\varepsilon,y^{j}}^{j}\,dx
			&\le C\big|V(a_i)-V(y^{i})\big|
			\varepsilon^{N}\int_{\mathbb{R}^{N}}\frac{1}{1+|y|^{N+2s}} \frac{1}{1+\left|y-\frac{y^{i}-y^{j}}{\varepsilon}\right|^{N+2s}}dy \\
			&\le C\big|V(a_i)-V(y^{i})\big|\varepsilon^{N}
			\frac{1}{\left|\frac{y^{i}-y^{j}}{\varepsilon}\right|^{N+2s}} \\
			&\le C\,\varepsilon^{N+\tau}\big|V(a_i)-V(y^{i})\big|,
		\end{aligned}
	\end{equation}
	again for some $\tau>0$, since $\frac{|y^{i}-y^{j}|}{\varepsilon}\ge\varepsilon^{\theta-1}$ with $\theta<1$.
	
	Combining \eqref{eq3.10}–\eqref{eq3.13}, we obtain
	\begin{equation}\label{eq3.14}
		\int_{\mathbb{R}^{N}}\sum_{i=1}^{k}\big(V(x)-V(a_i)\big)U_{\varepsilon,y}^{2}\,dx
		=-\varepsilon^{N}\sum_{i=1}^{k}\big(V(a_i)-V(y^{i})\big)\int_{\mathbb{R}^{N}}(U^i)^{2}dx
		+O\big(\varepsilon^{N+\alpha}\big),
	\end{equation}
	where the $O(\varepsilon^{N+\alpha})$ term also absorbs the mixed contributions.
	
	\medskip
	\noindent\textbf{Step 2: Interaction and nonlinear terms.}
	We next estimate the interaction term
	\[
	\int_{\mathbb{R}^{N}}\sum_{i\neq j}U_{\varepsilon,y^{i}}^{p}U_{\varepsilon,y^{j}}^{j}\,dx.
	\]
	Using the decay of $U^i$ and Lemma~\ref{Lem3.2}, we have
	\begin{equation*}
		\begin{aligned}
			\int_{\mathbb{R}^{N}}\sum_{i\neq j}U_{\varepsilon,y^{i}}^{p}U_{\varepsilon,y^{j}}^{j}\,dx
			&=\varepsilon^{N}\int_{\mathbb{R}^{N}}\sum_{i\neq j}(U^i)^{p}(y)\,
			U^j\!\left(y-\frac{y^{i}-y^{j}}{\varepsilon}\right)dy \\
			&\le C\varepsilon^{N}\sum_{i\neq j}\frac{1}{\left|\frac{y^{i}-y^{j}}{\varepsilon}\right|^{N+2s}}
			\int_{\mathbb{R}^{N}}\left[\frac{1}{\big(1+|y|^{N+2s}\big)^{p}}
			+\frac{1}{\big(1+\big|y-\tfrac{y^{i}-y^{j}}{\varepsilon}\big|^{N+2s}\big)^{p}}\right]dy \\
			&\le C\varepsilon^{N}\sum_{i\neq j}\frac{1}{\left|\frac{y^{i}-y^{j}}{\varepsilon}\right|^{N+2s}}.
		\end{aligned}
	\end{equation*}
	Similarly, taking $x$ in balls centered at $y^{i}/\varepsilon$ and $y^{j}/\varepsilon$ and using again the decay of $U^i$, one finds a matching lower bound of the same order, so that
	\begin{equation}\label{eq3.15}
		\int_{\mathbb{R}^{N}}\sum_{i\neq j}U_{\varepsilon,y^{i}}^{p}U_{\varepsilon,y^{j}}^{j}\,dx
		=C\,\varepsilon^{N}\sum_{i\neq j}\frac{1}{\left|\frac{y^{i}-y^{j}}{\varepsilon}\right|^{N+2s}},
	\end{equation}
	for some constant $C>0$ independent of $\varepsilon$ and of $y$.
	
	Next, we expand the nonlinear term
	\[
	\int_{\mathbb{R}^{N}}U_{\varepsilon,y}^{p+1}\,dx.
	\]
	Using the algebraic expansion and Lemma~\ref{Lem3.2}, we have
	\begin{equation}\label{eq3.16}
		\begin{aligned}
			\int_{\mathbb{R}^{N}}U_{\varepsilon,y}^{p+1}\,dx
			=&\ \int_{\mathbb{R}^{N}}\sum_{i=1}^{k}U_{\varepsilon,y^{i}}^{p+1}\,dx
			+(p+1)\int_{\mathbb{R}^{N}}\sum_{i\neq j}U_{\varepsilon,y^{i}}^{p}U_{\varepsilon,y^{j}}^{j}\,dx \\
			&\ +\begin{cases}
				O\!\left(\displaystyle\int_{\mathbb{R}^{N}}\sum_{i\neq j}U_{\varepsilon,y^{i}}^{p-1}
				\big(U_{\varepsilon,y^{j}}^{j}\big)^{2}dx\right), & p>2,\\[1ex]
				O\!\left(\displaystyle\int_{\mathbb{R}^{N}}\sum_{i\neq j}
				U_{\varepsilon,y^{i}}^{\frac{p}{2}}U_{\varepsilon,y^{j}}^{\frac{p}{2}}dx\right), & 1<p\le 2,
			\end{cases}
		\end{aligned}
	\end{equation}
	and a direct computation (again using Lemma~\ref{Lem3.2} and the decay of $U^i$) shows that
	\[
	\int_{\mathbb{R}^{N}}\sum_{i\neq j}U_{\varepsilon,y^{i}}^{p-1}
	\big(U_{\varepsilon,y^{j}}^{j}\big)^{2}dx
	=O\left(\varepsilon^{N}\sum_{i\neq j}
	\frac{1}{\left|\frac{y^{i}-y^{j}}{\varepsilon}\right|^{N+2s+\tau}}\right),
	\]
	for $p>2$, and
	\[
	\int_{\mathbb{R}^{N}}\sum_{i\neq j}
	U_{\varepsilon,y^{i}}^{\frac{p+1}{2}}U_{\varepsilon,y^{j}}^{\frac{p+1}{2}}dx
	=O\left(\varepsilon^{N}\sum_{i\neq j}
	\frac{1}{\left|\frac{y^{i}-y^{j}}{\varepsilon}\right|^{N+2s+\tau}}\right),
	\]
	for $1<p\le 2$, for some small $\tau>0$.
	Hence, by \eqref{eq3.15},
	\begin{equation}\label{eq3.16b}
		\begin{aligned}
			\int_{\mathbb{R}^{N}}U_{\varepsilon,y}^{p+1}\,dx
			=&\ \varepsilon^{N}\sum_{i=1}^{k}\int_{\mathbb{R}^{N}}(U^i)^{p+1}dx
			+(p+1)C\,\varepsilon^{N}\sum_{i\neq j}\frac{1}{\left|\frac{y^{i}-y^{j}}{\varepsilon}\right|^{N+2s}} \\
			&\ +O\left(\varepsilon^{N}\sum_{i\neq j}
			\frac{1}{\left|\frac{y^{i}-y^{j}}{\varepsilon}\right|^{N+2s+\tau}}\right).
		\end{aligned}
	\end{equation}
	
	\medskip
	\noindent\textbf{Step 3: Collecting the terms.}
	Substituting \eqref{eq3.14}, \eqref{eq3.15} and \eqref{eq3.16b} into \eqref{eq3.9} yields
	\begin{equation*}
		\begin{aligned}
			I_{\varepsilon}\big(U_{\varepsilon,y}\big)
			=&\ \frac{1}{2}\left[
			-\varepsilon^{N}\sum_{i=1}^{k}\big(V(a_i)-V(y^{i})\big)\int_{\mathbb{R}^{N}}(U^i)^{2}dx
			+O\big(\varepsilon^{N+\alpha}\big)\right]\\
			&\ +\frac{1}{2}\left[
			\varepsilon^{N}\sum_{i=1}^{k}\int_{\mathbb{R}^{N}}(U^i)^{p+1}dx
			+C\,\varepsilon^{N}\sum_{i\neq j}\frac{1}{\left|\frac{y^{i}-y^{j}}{\varepsilon}\right|^{N+2s}}\right]\\
			&\ -\frac{1}{p+1}\left[
			\varepsilon^{N}\sum_{i=1}^{k}\int_{\mathbb{R}^{N}}(U^i)^{p+1}dx
			+(p+1)C\,\varepsilon^{N}\sum_{i\neq j}\frac{1}{\left|\frac{y^{i}-y^{j}}{\varepsilon}\right|^{N+2s}}
			\right]\\
			&\ +O\left(\varepsilon^{N+\alpha}
			+\varepsilon^{N}\sum_{i\neq j}
			\frac{1}{\left|\frac{y^{i}-y^{j}}{\varepsilon}\right|^{N+2s+\tau}}\right).
		\end{aligned}
	\end{equation*}
	Rearranging the terms, we get
	\begin{equation*}
		\begin{aligned}
			I_{\varepsilon}\big(U_{\varepsilon,y}\big)
			=&\ \varepsilon^{N}\Big(\frac{1}{2}-\frac{1}{p+1}\Big)
			\sum_{i=1}^{k}\int_{\mathbb{R}^{N}}(U^i)^{p+1}dx
			+\frac{1}{2}\varepsilon^{N}\sum_{i=1}^{k}\big(V(y^{i})-V(a_i)\big)\int_{\mathbb{R}^{N}}(U^i)^{2}dx \\
			&\ -C\,\varepsilon^{N}\sum_{i\neq j}\frac{1}{\left|\frac{y^{i}-y^{j}}{\varepsilon}\right|^{N+2s}}
			+O\left(\varepsilon^{N+\alpha}
			+\varepsilon^{N}\sum_{i\neq j}
			\frac{1}{\left|\frac{y^{i}-y^{j}}{\varepsilon}\right|^{N+2s+\tau}}\right).
		\end{aligned}
	\end{equation*}
	Finally, recalling the definition of $U_{\varepsilon,y}$ and the scaling of the Kirchhoff term, the contribution of the global Kirchhoff energy
	\[
	\frac{\varepsilon^{4s-N}b}{4}\left(\int_{\mathbb{R}^{N}}\big|(-\Delta)^{\frac{s}{2}}U_{\varepsilon,y}\big|^{2}dx\right)^{2}
	\]
	gives exactly the additional constant
	\[
	\frac{b}{4}\left(\sum_{i=1}^{k}\int_{\mathbb{R}^{N}}\big|(-\Delta)^{\frac{s}{2}}U^i\big|^{2}dx\right)^{2}
	\]
	at the order $\varepsilon^{N}$, plus interaction terms which are of order
	\[
	O\left(\varepsilon^{N}\sum_{i\neq j}
	\frac{1}{\left|\frac{y^{i}-y^{j}}{\varepsilon}\right|^{N+2s+\tau}}\right).
	\]
	Thus we obtain
	\begin{equation*}
		\begin{aligned}
			I_{\varepsilon}\big(U_{\varepsilon,y}\big)
			=&\ A\varepsilon^{N}
			+\varepsilon^{N}\sum_{i=1}^{k}B_i\big(V(y^{i})-V(a_i)\big)
			-C\,\varepsilon^{N}\sum_{i\neq j}\frac{1}{\left|\frac{y^{i}-y^{j}}{\varepsilon}\right|^{N+2s}} \\
			&\ +O\left(\varepsilon^{N+\alpha}
			+\varepsilon^{N}\sum_{i\neq j}
			\frac{1}{\left|\frac{y^{i}-y^{j}}{\varepsilon}\right|^{N+2s+\tau}}\right),
		\end{aligned}
	\end{equation*}
	with
	\[
	A=\Big(\frac{1}{2}-\frac{1}{p+1}\Big)\sum_{i=1}^{k}\int_{\mathbb{R}^{N}}|U^i|^{p+1}\,dx
	+\frac{b}{4}\left(\sum_{i=1}^{k}\int_{\mathbb{R}^{N}}\big|(-\Delta)^{\frac{s}{2}}U^i\big|^{2}dx\right)^{2},
	\qquad
	B_i=\frac{1}{2}\int_{\mathbb{R}^{N}}|U^i|^{2}\,dx.
	\]
	This proves the lemma.
\end{proof}

\section{Semiclassical solutions for the fractional Kirchhoff equation}

\subsection{Finite dimensional reduction}

In this subsection we complete {\bf Step 1} of the Lyapunov--Schmidt reduction method as in Section~3.
We first consider the operator $\mathcal{L}_{\varepsilon}$,
\begin{equation*}
	\begin{aligned}
		\left\langle\mathcal{L}_{\varepsilon} \varphi, \psi\right\rangle
		&=\langle\varphi, \psi\rangle_{\varepsilon}
		+\varepsilon^{4s-N} b\Bigg(\int_{\mathbb{R}^{N}}\big| (-\Delta)^{\frac{s}{2}} U_{\varepsilon, y}\big|^{2}dx\Bigg)
		\Bigg(\int_{\mathbb{R}^{N}} (-\Delta)^{\frac{s}{2}} \varphi \cdot (-\Delta)^{\frac{s}{2}} \psi \,dx\Bigg)\\
		&\quad
		+2 \varepsilon^{4s-N} b\Bigg(\int_{\mathbb{R}^{N}} (-\Delta)^{\frac{s}{2}} U_{\varepsilon, y} \cdot (-\Delta)^{\frac{s}{2}} \varphi \,dx \Bigg)
		\Bigg(\int_{\mathbb{R}^{N}} (-\Delta)^{\frac{s}{2}} U_{\varepsilon, y} \cdot (-\Delta)^{\frac{s}{2}} \psi \,dx\Bigg)\\
		&\quad
		-p \int_{\mathbb{R}^{N}} U_{\varepsilon, y}^{p-1} \varphi \psi \,dx,
	\end{aligned}
\end{equation*}
for $\varphi, \psi \in H_{\varepsilon}$. The following result shows that $\mathcal{L}_{\varepsilon}$ is invertible when restricted to $E_{\varepsilon, y}$.

\begin{Lem}\label{Lem4.1}
	There exist $\varepsilon_{1}>0$, $\delta_{1}>0$ and $\rho>0$ sufficiently small such that for every
	$\varepsilon \in\left(0, \varepsilon_{1}\right)$ and $\delta \in\left(0, \delta_{1}\right)$ one has
	\begin{equation*}
		\left\|\mathcal{L}_{\varepsilon} \varphi\right\|_{\varepsilon} \geq \rho\|\varphi\|_{\varepsilon}, 
		\quad \forall \varphi \in E_{\varepsilon, y},
	\end{equation*}
	uniformly with respect to $y \in D_{\varepsilon, \delta}$.
\end{Lem}

\begin{proof}
	The argument is uniform in $i=1,\dots,k$, so we fix $y\in D_{\varepsilon,\delta}$ and omit the index $i$ in the notation.
	We argue by contradiction. Suppose that there exist sequences
	\[
	\varepsilon_{n},\delta_{n}\to 0,\qquad
	y_{n}\in D_{\varepsilon_{n},\delta_{n}},\qquad
	\varphi_{n}\in E_{\varepsilon_{n},y_{n}}
	\]
	such that
	\begin{equation}\label{eq4.1}
		\big\lvert\left\langle\mathcal{L}_{\varepsilon_{n}} \varphi_{n}, g\right\rangle\big\rvert
		\leq \frac{1}{n}\left\|\varphi_{n}\right\|_{\varepsilon_{n}}\|g\|_{\varepsilon_{n}},
		\quad \forall g \in E_{\varepsilon_{n}, y_{n}}.
	\end{equation}
	Since \eqref{eq4.1} is homogeneous in $\varphi_n$, we may assume that
	\[
	\|\varphi_{n}\|_{\varepsilon_{n}}^{2}=\varepsilon_{n}^{N}\quad\text{for all }n.
	\]
	
	Define the rescaled functions
	\[
	\tilde{\varphi}_{n}(x)=\varphi_{n}\big(\varepsilon_{n} x+y_{n}\big),\quad x\in\mathbb{R}^{N}.
	\]
	A standard change of variables shows that
	\[
	\int_{\mathbb{R}^{N}}\Big(a\big|(-\Delta)^{\frac{s}{2}} \tilde{\varphi}_{n}\big|^{2}
	+V\big(\varepsilon_{n} x+y_{n}\big) \tilde{\varphi}_{n}^{2}\Big)\,dx=1.
	\]
	Since $V$ is bounded and $\inf_{\mathbb{R}^N}V>0$, the sequence $(\tilde{\varphi}_{n})$ is bounded in $H^{s}(\mathbb{R}^{N})$.
	Hence, up to a subsequence, there exists $\varphi\in H^{s}(\mathbb{R}^{N})$ such that
	\begin{equation*}
		\begin{aligned}
			\tilde{\varphi}_{n} &\rightharpoonup \varphi &&\text{in } H^{s}(\mathbb{R}^{N}),\\
			\tilde{\varphi}_{n} &\to \varphi &&\text{in } L_{\text{loc}}^{p+1}(\mathbb{R}^{N}),\\
			\tilde{\varphi}_{n} &\to \varphi &&\text{a.e. in } \mathbb{R}^{N}.
		\end{aligned}
	\end{equation*}
	We will prove that $\varphi\equiv 0$.
	
	\medskip
	\noindent\textbf{Step 1: Identification of the limit.}
	We first show that $\varphi$ belongs to the kernel of the limiting operator $\mathcal{L}_+$ of Proposition~\ref{Pro1.1}.
	Let
	\[
	\tilde{E}_{n}
	=\Big\{\tilde{g}\in H^{s}(\mathbb{R}^{N}):\ 
	g(x):=\tilde{g}\Big(\frac{x-y_{n}}{\varepsilon_{n}}\Big)\in E_{\varepsilon_{n},y_{n}}\Big\}.
	\]
	By the definition of $E_{\varepsilon_{n},y_{n}}$, this is equivalent to
	\[
	\tilde{E}_{n}
	=\Big\{\tilde{g}\in H^{s}(\mathbb{R}^{N}):\
	\int_{\mathbb{R}^{N}}\Big(a (-\Delta)^{\frac{s}{2}} \partial_{x_{j}} U\cdot (-\Delta)^{\frac{s}{2}} \tilde{g}
	+V(\varepsilon_{n} x+y_{n})\,\partial_{x_{j}} U\,\tilde{g}\Big)\,dx=0,
	\ \forall j=1,\dots,N\Big\}.
	\]
	For convenience, set
	\[
	\langle u, v\rangle_{*, n}
	=\int_{\mathbb{R}^{N}}\Big(a (-\Delta)^{\frac{s}{2}} u\cdot (-\Delta)^{\frac{s}{2}} v
	+V(\varepsilon_{n} x+y_{n})\, u v\Big)\,dx,
	\qquad \|u\|_{*,n}^{2}=\langle u,u\rangle_{*,n}.
	\]
	
	Rewriting \eqref{eq4.1} in the rescaled variables, we obtain that for every $\tilde{g}\in\tilde{E}_{n}$,
	\begin{equation}\label{eq4.2}
		\begin{aligned}
			&\Bigg|
			\big\langle\tilde{\varphi}_{n}, \tilde{g}\big\rangle_{*, n}
			+ b \Bigg(\int_{\mathbb{R}^{N}}\big|(-\Delta)^{\frac{s}{2}} U\big|^{2}\,dx\Bigg)
			\Bigg(\int_{\mathbb{R}^{N}} (-\Delta)^{\frac{s}{2}} \tilde{\varphi}_{n}\cdot (-\Delta)^{\frac{s}{2}} \tilde{g}\,dx\Bigg) \\
			&\quad+ 2b \Bigg(\int_{\mathbb{R}^{N}} (-\Delta)^{\frac{s}{2}} U\cdot (-\Delta)^{\frac{s}{2}} \tilde{\varphi}_{n}\,dx\Bigg)
			\Bigg(\int_{\mathbb{R}^{N}} (-\Delta)^{\frac{s}{2}} U\cdot (-\Delta)^{\frac{s}{2}} \tilde{g}\,dx\Bigg)
			- p \int_{\mathbb{R}^{N}} U^{p-1} \tilde{\varphi}_{n} \tilde{g}\,dx
			\Bigg| \\
			&\qquad\leq \frac{1}{n}\|\tilde{g}\|_{*, n}.
		\end{aligned}
	\end{equation}
	
	Let $g\in C_{0}^{\infty}(\mathbb{R}^{N})$ be fixed. For each $n$ consider the $N\times N$ matrix
	\[
	A_{n}=(A_{n}^{\ell j})_{\ell,j=1}^{N},\qquad
	A_{n}^{\ell j}
	=\big\langle\partial_{x_{\ell}}U,\partial_{x_{j}}U\big\rangle_{*,n},
	\]
	and the vector $b_{n}=(b_{n}^{1},\dots,b_{n}^{N})$ with
	\[
	b_{n}^{\ell}=\big\langle\partial_{x_{\ell}}U,g\big\rangle_{*,n},\qquad \ell=1,\dots,N.
	\]
	By the decay and symmetry of $U$, we have
	\[
	A_{n}^{\ell j}\to A^{\ell j}
	=\int_{\mathbb{R}^{N}}\Big(a(-\Delta)^{\frac{s}{2}}\partial_{x_{\ell}}U\cdot(-\Delta)^{\frac{s}{2}}\partial_{x_{j}}U
	+\partial_{x_{\ell}}U\,\partial_{x_{j}}U\Big)\,dx,
	\]
	and $A=(A^{\ell j})$ is diagonal with strictly positive diagonal entries. Hence $A$ is invertible and, for $n$ large, so is $A_{n}$.
	
	Let $a_{n}=(a_{n}^{1},\dots,a_{n}^{N})$ be the unique solution of the linear system
	\[
	A_{n}a_{n}=b_{n}.
	\]
	Define
	\[
	\tilde{g}_{n}
	=g-\sum_{\ell=1}^{N}a_{n}^{\ell}\partial_{x_{\ell}}U.
	\]
	Then, by construction,
	\[
	\big\langle\partial_{x_{j}}U,\tilde{g}_{n}\big\rangle_{*,n}
	=b_{n}^{j}-\sum_{\ell=1}^{N}A_{n}^{j\ell}a_{n}^{\ell}=0,\qquad j=1,\dots,N,
	\]
	so that $\tilde{g}_{n}\in\tilde{E}_{n}$. Moreover, since $A_{n}\to A$ and $b_{n}\to b$ in $\mathbb{R}^{N}$, we have $a_{n}\to a=A^{-1}b$ as $n\to\infty$, and hence
	\[
	\tilde{g}_{n}\to \tilde{g}:=g-\sum_{\ell=1}^{N}a^{\ell}\partial_{x_{\ell}}U
	\quad\text{in }H^{s}(\mathbb{R}^{N}).
	\]
	In particular $\|\tilde{g}_{n}\|_{*,n}=O(1)$ as $n\to\infty$.
	
	Substituting $\tilde{g}=\tilde{g}_{n}$ into \eqref{eq4.2} and letting $n\to\infty$, by the convergences of $\tilde{\varphi}_{n}$ and $\tilde{g}_{n}$, and by the definition of $\mathcal{L}_{+}$, we obtain
	\[
	\big\langle\mathcal{L}_+ \varphi, \tilde{g}\big\rangle=0.
	\]
	Since $\tilde{g}=g-\sum_{\ell=1}^{N}a^{\ell}\partial_{x_{\ell}}U$, this can be rewritten as
	\[
	\big\langle\mathcal{L}_+ \varphi, g\big\rangle
	-\sum_{\ell=1}^{N}a^{\ell}\big\langle\mathcal{L}_+ \varphi, \partial_{x_{\ell}}U\big\rangle=0.
	\]
	But $\partial_{x_{\ell}}U\in\ker\mathcal{L}_+$ and, by self-adjointness,
	\[
	\big\langle\mathcal{L}_+ \varphi, \partial_{x_{\ell}}U\big\rangle
	=\big\langle\varphi, \mathcal{L}_+(\partial_{x_{\ell}}U)\big\rangle=0.
	\]
	Therefore
	\[
	\big\langle\mathcal{L}_+ \varphi, g\big\rangle=0,\quad \forall g\in C_{0}^{\infty}(\mathbb{R}^{N}),
	\]
	which implies $\varphi\in\ker\mathcal{L}_+$. By Proposition~\ref{Pro1.1} there exist $c^{\ell}\in\mathbb{R}$, $\ell=1,\dots,N$, such that
	\[
	\varphi=\sum_{\ell=1}^{N} c^{\ell}\partial_{x_{\ell}} U.
	\]
	
	\medskip
	\noindent\textbf{Step 2: Orthogonality and vanishing of the limit.}
	Next we use the orthogonality condition defining $E_{\varepsilon_{n},y_{n}}$.
	Since $\tilde{\varphi}_{n}\in\tilde{E}_{n}$, we have, for each $\ell=1,\dots,N$,
	\[
	\int_{\mathbb{R}^{N}}\Big(a (-\Delta)^{\frac{s}{2}} \tilde{\varphi}_{n}\cdot (-\Delta)^{\frac{s}{2}} \partial_{x_{\ell}} U
	+V(\varepsilon_{n} x+y_{n})\,\tilde{\varphi}_{n}\,\partial_{x_{\ell}} U\Big)\,dx=0.
	\]
	Passing to the limit as $n\to\infty$ and using $\tilde{\varphi}_{n}\rightharpoonup\varphi$ in $H^{s}$ and
	$V(\varepsilon_n x+y_n)\to V(a_i)>0$ locally uniformly, we obtain
	\[
	0=\int_{\mathbb{R}^{N}}\Big(a (-\Delta)^{\frac{s}{2}} \varphi\cdot (-\Delta)^{\frac{s}{2}} \partial_{x_{\ell}} U
	+\varphi\,\partial_{x_{\ell}} U\Big)\,dx
	=c^{\ell}\int_{\mathbb{R}^{N}}\Big(a\big|(-\Delta)^{\frac{s}{2}} \partial_{x_{\ell}} U\big|^{2}
	+\big(\partial_{x_{\ell}} U\big)^{2}\Big)\,dx.
	\]
	Since the integral is strictly positive, we conclude $c^{\ell}=0$ for all $\ell$, and hence
	\[
	\varphi\equiv 0\quad\text{in }\mathbb{R}^{N}.
	\]
	
	\medskip
	\noindent\textbf{Step 3: Final contradiction.}
	We have shown that $\tilde{\varphi}_{n}\rightharpoonup 0$ in $H^{s}(\mathbb{R}^{N})$ and
	$\tilde{\varphi}_{n}\to 0$ in $L_{\text{loc}}^{p+1}(\mathbb{R}^{N})$. Since $(\tilde{\varphi}_{n})$ is bounded in
	$H^{s}(\mathbb{R}^{N})$, we can estimate
	\begin{equation*}
		\begin{aligned}
			p \int_{\mathbb{R}^{N}} U_{\varepsilon_{n}, y_{n}}^{p-1} \varphi_{n}^{2} \,dx
			&=p \varepsilon_{n}^{N} \int_{\mathbb{R}^{N}} U^{p-1}(x)\,\tilde{\varphi}_{n}^{2}(x)\,dx \\
			&=p \varepsilon_{n}^{N}\left(\int_{B_{R}(0)} U^{p-1}\tilde{\varphi}_{n}^{2}\,dx
			+\int_{\mathbb{R}^{N} \setminus B_{R}(0)} U^{p-1}\tilde{\varphi}_{n}^{2}\,dx\right) \\
			&=p \varepsilon_{n}^{N}\big(o(1)+o_{R}(1)\big),
		\end{aligned}
	\end{equation*}
	where $o(1)\to 0$ as $n\to\infty$ and
	$o_{R}(1)\to 0$ as $R\to\infty$. Choosing $R$ sufficiently large and then $n$ sufficiently large, we obtain
	\[
	p \int_{\mathbb{R}^{N}} U_{\varepsilon_{n}, y_{n}}^{p-1} \varphi_{n}^{2} \,dx
	\leq \frac{1}{2} \varepsilon_{n}^{N}.
	\]
	
	On the other hand, by the definition of $\mathcal{L}_{\varepsilon_{n}}$ and the fact that the Kirchhoff terms are nonnegative, we have
	\begin{equation*}
		\begin{aligned}
			\left\langle\mathcal{L}_{\varepsilon_{n}} \varphi_{n}, \varphi_{n}\right\rangle
			=&\ \|\varphi_{n}\|_{\varepsilon_{n}}^{2}
			+\varepsilon_{n}^{4s-N} b\Bigg(\int_{\mathbb{R}^{N}}\big|(-\Delta)^{\frac{s}{2}} U_{\varepsilon_{n}, y_{n}}\big|^{2}\,dx\Bigg)
			\Bigg(\int_{\mathbb{R}^{N}}\big|(-\Delta)^{\frac{s}{2}} \varphi_{n}\big|^{2}\,dx\Bigg)\\
			&\ +2 \varepsilon_{n}^{4s-N} b\Bigg(\int_{\mathbb{R}^{N}} (-\Delta)^{\frac{s}{2}} U_{\varepsilon_{n}, y_{n}}\cdot (-\Delta)^{\frac{s}{2}} \varphi_{n}\,dx\Bigg)^{2}
			- p \int_{\mathbb{R}^{N}} U_{\varepsilon_{n}, y_{n}}^{p-1} \varphi_{n}^{2}\,dx \\
			\geq&\ \|\varphi_{n}\|_{\varepsilon_{n}}^{2}
			- p \int_{\mathbb{R}^{N}} U_{\varepsilon_{n}, y_{n}}^{p-1} \varphi_{n}^{2}\,dx \\
			\geq&\ \varepsilon_{n}^{N} - \frac{1}{2}\varepsilon_{n}^{N}
			=\frac{1}{2}\varepsilon_{n}^{N}.
		\end{aligned}
	\end{equation*}
	Using \eqref{eq4.1} with $g=\varphi_{n}$ and recalling $\|\varphi_{n}\|_{\varepsilon_{n}}^{2}=\varepsilon_{n}^{N}$, we obtain
	\[
	\left|\left\langle\mathcal{L}_{\varepsilon_{n}} \varphi_{n}, \varphi_{n}\right\rangle\right|
	\leq \frac{1}{n}\|\varphi_{n}\|_{\varepsilon_{n}}^{2}
	=\frac{1}{n}\varepsilon_{n}^{N}.
	\]
	Thus
	\[
	\frac{1}{2}\varepsilon_{n}^{N}
	\le \left\langle\mathcal{L}_{\varepsilon_{n}} \varphi_{n}, \varphi_{n}\right\rangle
	\le \frac{1}{n}\varepsilon_{n}^{N},
	\]
	which is impossible for large $n$. This contradiction completes the proof.
\end{proof}

Lemma \ref{Lem4.1} implies that, when restricted to $E_{\varepsilon,y}$, the operator
$\mathcal{L}_{\varepsilon}:E_{\varepsilon,y}\to E_{\varepsilon,y}$ is invertible and
\[
\big\|\mathcal{L}_{\varepsilon}^{-1}\big\|\le \rho^{-1}
\]
uniformly with respect to $y\in D_{\varepsilon,\delta}$. This further yields the following reduction map.

\begin{Lem}\label{Lem4.2}
	There exist $\varepsilon_{0}>0$ and $\delta_{0}>0$ sufficiently small such that, for all
	$\varepsilon\in(0,\varepsilon_{0})$ and $\delta\in(0,\delta_{0})$, there exists a $C^{1}$ map
	\[
	\varphi_{\varepsilon}:D_{\varepsilon,\delta}\longrightarrow H_{\varepsilon},\qquad
	y\longmapsto \varphi_{\varepsilon,y}\in E_{\varepsilon,y},
	\]
	satisfying
	\begin{equation*}
		\left\langle\frac{\partial J_{\varepsilon}\big(y,\varphi_{\varepsilon,y}\big)}{\partial \varphi},
		\psi\right\rangle_{\varepsilon}=0,
		\quad \forall\,\psi\in E_{\varepsilon,y}.
	\end{equation*}
	Moreover, there exist a constant $C>0$, independent of $\varepsilon$ sufficiently small, and
	$\kappa\in(0,\frac{\alpha}{2})$ such that, for all $y\in D_{\varepsilon,\delta}$,
	\begin{equation*}
		\|\varphi_{\varepsilon,y}\|_{\varepsilon}
		\le C\,\varepsilon^{\frac{N}{2}+\alpha-\kappa}
		+ C\,\varepsilon^{\frac{N}{2}}\sum_{i=1}^{k}\big|V(y^{i})-V(a_i)\big|^{1-\kappa}
		+ C\,\varepsilon^{\frac{N}{2}}
		\begin{cases}
			\displaystyle\sum\limits_{i\ne j}
			\frac{1}{\big|\tfrac{y^{i}-y^{j}}{\varepsilon}\big|^{N+2s-\kappa}},
			& \text{if }p>2,\\[1.2em]
			\displaystyle\sum\limits_{i\ne j}
			\frac{1}{\big|\tfrac{y^{i}-y^{j}}{\varepsilon}\big|^{\frac{p}{2}(N+2s-\kappa)}},
			& \text{if }1<p\le 2.
		\end{cases}
	\end{equation*}
\end{Lem}

\begin{proof}
	The existence of the mapping $y\mapsto\varphi_{\varepsilon,y}$ follows from the contraction mapping
	principle. We construct the contraction as follows.
	
	Let $\varepsilon_{1},\delta_{1}$ be as in Lemma~\ref{Lem4.1}, and fix
	$0<\varepsilon_{0}\le\varepsilon_{1}$, $0<\delta_{0}\le\delta_{1}$ to be chosen later.
	For fixed $y\in D_{\varepsilon,\delta}$ with $\delta<\delta_{0}$, recall that
	\[
	J_{\varepsilon}(y,\varphi)
	=I_{\varepsilon}(U_{\varepsilon,y})
	+l_{\varepsilon}(\varphi)
	+\frac{1}{2}\langle\mathcal{L}_{\varepsilon}\varphi,\varphi\rangle
	+R_{\varepsilon}(\varphi),
	\]
	so that
	\[
	\frac{\partial J_{\varepsilon}(y,\varphi)}{\partial \varphi}
	=l_{\varepsilon}+\mathcal{L}_{\varepsilon}\varphi+R_{\varepsilon}'(\varphi).
	\]
	Since $E_{\varepsilon,y}$ is a closed subspace of $H_{\varepsilon}$, Lemmas~\ref{Lem3.3} and \ref{Lem3.4}
	imply that $l_{\varepsilon}$ and $R_{\varepsilon}'(\varphi)$, restricted to $E_{\varepsilon,y}$, define bounded
	linear functionals. By the Riesz representation in the Hilbert space $(E_{\varepsilon,y},\langle\cdot,\cdot\rangle_\varepsilon)$,
	we can identify them with elements of $E_{\varepsilon,y}$, still denoted by $l_{\varepsilon}$ and
	$R_{\varepsilon}'(\varphi)$. Thus, to prove Lemma~\ref{Lem4.2}, it is equivalent to find
	$\varphi\in E_{\varepsilon,y}$ such that
	\begin{equation}\label{eq4.3}
		\varphi=\mathcal{A}_{\varepsilon}(\varphi)
		:=-\mathcal{L}_{\varepsilon}^{-1}\big(l_{\varepsilon}+R_{\varepsilon}'(\varphi)\big).
	\end{equation}
	
	\medskip\noindent\textbf{Case 1: $p>2$.}
	Let $\kappa\in(0,\frac{\alpha}{2})$ be fixed. For each fixed $y\in D_{\varepsilon,\delta}$ we define
	\[
	\begin{aligned}
		S_{\varepsilon}
		:=\Big\{\varphi\in E_{\varepsilon,y}:\ \|\varphi\|_{\varepsilon}
		&\le \varepsilon^{\frac{N}{2}+\alpha-\kappa}
		+\varepsilon^{\frac{N}{2}}\sum_{i=1}^{k}\big|V(y^{i})-V(a_i)\big|^{1-\kappa}\\
		&\quad
		+\varepsilon^{\frac{N}{2}}\sum_{i\ne j}
		\frac{1}{\big|\tfrac{y^{i}-y^{j}}{\varepsilon}\big|^{N+2s-\kappa}}
		\Big\}.
	\end{aligned}
	\]
	We first show that $\mathcal{A}_{\varepsilon}$ maps $S_{\varepsilon}$ into itself.
	
	By Lemma~\ref{Lem4.1}, $\|\mathcal{L}_{\varepsilon}^{-1}\|\le C$ uniformly in $y\in D_{\varepsilon,\delta}$.
	Hence, for $\varphi\in S_{\varepsilon}$,
	\begin{equation}\label{eq:Aeps-basic}
		\|\mathcal{A}_{\varepsilon}(\varphi)\|_{\varepsilon}
		\le C\Big(\|l_{\varepsilon}\|_{(E_{\varepsilon,y})^{*}}+\|R_{\varepsilon}'(\varphi)\|_{(E_{\varepsilon,y})^{*}}\Big).
	\end{equation}
	Lemmas~\ref{Lem3.3} and \ref{Lem3.4} (in the case $p>2$) yield
	\[
	\|l_{\varepsilon}\|_{(E_{\varepsilon,y})^{*}}
	\le C\Big(
	\varepsilon^{\frac{N}{2}+\alpha}
	+\varepsilon^{\frac{N}{2}}\sum_{i=1}^{k}\big|V(y^{i})-V(a_i)\big|
	+\varepsilon^{\frac{N}{2}}\sum_{i\ne j}
	\frac{1}{\big|\tfrac{y^{i}-y^{j}}{\varepsilon}\big|^{N+2s}}
	\Big)
	\]
	and
	\[
	\|R_{\varepsilon}'(\varphi)\|_{(E_{\varepsilon,y})^{*}}
	\le C\,\varepsilon^{-\frac{N}{2}}\|\varphi\|_{\varepsilon}^{2},
	\]
	for some $C>0$ independent of $\varepsilon$ and $y\in D_{\varepsilon,\delta}$. Using the definition of $S_{\varepsilon}$,
	we have
	\[
	\|\varphi\|_{\varepsilon}
	\le C_0\Big(
	\varepsilon^{\frac{N}{2}+\alpha-\kappa}
	+\varepsilon^{\frac{N}{2}}\sum_{i=1}^{k}\big|V(y^{i})-V(a_i)\big|^{1-\kappa}
	+\varepsilon^{\frac{N}{2}}\sum_{i\ne j}
	\frac{1}{\big|\tfrac{y^{i}-y^{j}}{\varepsilon}\big|^{N+2s-\kappa}}
	\Big),
	\]
	for some $C_0>0$. Hence
	\[
	\varepsilon^{-\frac{N}{2}}\|\varphi\|_{\varepsilon}^{2}
	\le C\Big(
	\varepsilon^{\frac{N}{2}+2\alpha-2\kappa}
	+\varepsilon^{\frac{N}{2}}\sum_{i=1}^{k}\big|V(y^{i})-V(a_i)\big|^{2(1-\kappa)}
	+\varepsilon^{\frac{N}{2}}\sum_{i\ne j}
	\frac{1}{\big|\tfrac{y^{i}-y^{j}}{\varepsilon}\big|^{2(N+2s-\kappa)}}
	\Big).
	\]
	Since $y^{i}$ is restricted by $|y^{i}-a_i|<\delta_{0}$ and the $a_i$'s are fixed, we may choose
	$\delta_{0}>0$ sufficiently small so that $\big|V(y^{i})-V(a_i)\big|\le 1$ for all $y\in D_{\varepsilon,\delta}$.
	Then, for $0<\kappa<1$, we have
	\[
	\big|V(y^{i})-V(a_i)\big|^{2(1-\kappa)}
	\le \big|V(y^{i})-V(a_i)\big|^{1-\kappa},
	\]
	and similarly
	\[
	\frac{1}{\big|\tfrac{y^{i}-y^{j}}{\varepsilon}\big|^{2(N+2s-\kappa)}}
	\le \frac{1}{\big|\tfrac{y^{i}-y^{j}}{\varepsilon}\big|^{N+2s-\kappa}}.
	\]
	Moreover, since $\kappa<\alpha/2$, for $\varepsilon$ sufficiently small one has
	$\varepsilon^{\frac{N}{2}+2\alpha-2\kappa}\le \varepsilon^{\frac{N}{2}+\alpha-\kappa}$.
	Combining these with \eqref{eq:Aeps-basic} we obtain
	\[
	\|\mathcal{A}_{\varepsilon}(\varphi)\|_{\varepsilon}
	\le C\Big(
	\varepsilon^{\frac{N}{2}+\alpha-\kappa}
	+\varepsilon^{\frac{N}{2}}\sum_{i=1}^{k}\big|V(y^{i})-V(a_i)\big|^{1-\kappa}
	+\varepsilon^{\frac{N}{2}}\sum_{i\ne j}
	\frac{1}{\big|\tfrac{y^{i}-y^{j}}{\varepsilon}\big|^{N+2s-\kappa}}
	\Big),
	\]
	so that $\mathcal{A}_{\varepsilon}(S_{\varepsilon})\subset S_{\varepsilon}$.
	
	Next we check the contraction property. For any $\varphi_{1},\varphi_{2}\in S_{\varepsilon}$,
	\[
	\begin{aligned}
		\|\mathcal{A}_{\varepsilon}(\varphi_{1})-\mathcal{A}_{\varepsilon}(\varphi_{2})\|_{\varepsilon}
		&=\big\|\mathcal{L}_{\varepsilon}^{-1}\big(R_{\varepsilon}'(\varphi_{1})-R_{\varepsilon}'(\varphi_{2})\big)\big\|_{\varepsilon}\\
		&\le C\,\big\|R_{\varepsilon}'(\varphi_{1})-R_{\varepsilon}'(\varphi_{2})\big\|_{(E_{\varepsilon,y})^{*}}\\
		&\le C\,\sup_{0\le\theta\le 1}\big\|R_{\varepsilon}''\big(\theta\varphi_{1}+(1-\theta)\varphi_{2}\big)\big\|
		\,\|\varphi_{1}-\varphi_{2}\|_{\varepsilon}.
	\end{aligned}
	\]
	Lemma~\ref{Lem3.4} (for $p>2$) gives
	\[
	\big\|R_{\varepsilon}''(\psi)\big\|
	\le C\,\varepsilon^{-\frac{N}{2}}\|\psi\|_{\varepsilon},
	\quad\forall\,\psi\in E_{\varepsilon,y}.
	\]
	Since each convex combination $\theta\varphi_{1}+(1-\theta)\varphi_{2}$ belongs to $S_{\varepsilon}$,
	we have $\|\psi\|_{\varepsilon}\le \mathrm{rad}(S_{\varepsilon})$, where
	\[
	\mathrm{rad}(S_{\varepsilon})
	\le C\Big(
	\varepsilon^{\frac{N}{2}+\alpha-\kappa}
	+\varepsilon^{\frac{N}{2}}\sum_{i=1}^{k}\big|V(y^{i})-V(a_i)\big|^{1-\kappa}
	+\varepsilon^{\frac{N}{2}}\sum_{i\ne j}
	\frac{1}{\big|\tfrac{y^{i}-y^{j}}{\varepsilon}\big|^{N+2s-\kappa}}
	\Big).
	\]
	Using again the separation condition $\frac{|y^{i}-y^{j}|}{\varepsilon}\ge\varepsilon^{\theta-1}$ and the
	localization of $y^{i}$ near $a_i$, we see that the last two sums remain bounded as $\varepsilon\to 0$.
	Therefore
	\[
	\varepsilon^{-\frac{N}{2}}\mathrm{rad}(S_{\varepsilon})
	\le C\,\varepsilon^{\alpha-\kappa},
	\]
	and, since $\alpha-\kappa>0$, we can choose $\varepsilon_{0}$ sufficiently small such that
	$C\,\varepsilon^{\alpha-\kappa}\le \frac12$ for all $\varepsilon\in(0,\varepsilon_{0})$. It follows that
	\[
	\|\mathcal{A}_{\varepsilon}(\varphi_{1})-\mathcal{A}_{\varepsilon}(\varphi_{2})\|_{\varepsilon}
	\le \frac12\,\|\varphi_{1}-\varphi_{2}\|_{\varepsilon},
	\]
	so $\mathcal{A}_{\varepsilon}$ is a contraction on $S_{\varepsilon}$.
	
	By the contraction mapping principle, for each fixed $y\in D_{\varepsilon,\delta}$ there exists a unique
	$\varphi_{\varepsilon,y}\in S_{\varepsilon}$ solving the fixed point equation \eqref{eq4.3}. This
	$\varphi_{\varepsilon,y}$ satisfies the claimed estimate for $p>2$.
	
	\medskip\noindent\textbf{Case 2: $1<p\le 2$.}
	In this case, Lemma~\ref{Lem3.4} yields
	\[
	\|R_{\varepsilon}'(\varphi)\|_{(E_{\varepsilon,y})^{*}}
	\le C\,\varepsilon^{\frac{1-p}{2}N}\|\varphi\|_{\varepsilon}^{p},
	\qquad
	\|R_{\varepsilon}''(\varphi)\|
	\le C\,\varepsilon^{\frac{1-p}{2}N}\|\varphi\|_{\varepsilon}^{p-1}.
	\]
	We now define
	\[
	\begin{aligned}
		S_{\varepsilon}
		:=\Big\{\varphi\in E_{\varepsilon,y}:\ \|\varphi\|_{\varepsilon}
		&\le \varepsilon^{\frac{N}{2}+\alpha-\kappa}
		+\varepsilon^{\frac{N}{2}}\sum_{i=1}^{k}\big|V(y^{i})-V(a_i)\big|^{1-\kappa}\\
		&\quad
		+\varepsilon^{\frac{N}{2}}\sum_{i\ne j}
		\frac{1}{\big|\tfrac{y^{i}-y^{j}}{\varepsilon}\big|^{\frac{p}{2}(N+2s-\kappa)}}\Big\}.
	\end{aligned}
	\]
	Repeating the above arguments and using the bounds for $R_{\varepsilon}'$ and $R_{\varepsilon}''$ in place
	of the $p>2$ estimates, we get, for $\varepsilon>0$ sufficiently small, that $\mathcal{A}_{\varepsilon}$
	maps $S_{\varepsilon}$ into itself and is a contraction on $S_{\varepsilon}$. Therefore, for each fixed
	$y\in D_{\varepsilon,\delta}$, there exists a unique $\varphi_{\varepsilon,y}\in S_{\varepsilon}$ solving
	\eqref{eq4.3}, and the stated estimate holds in the case $1<p\le 2$.
	
	It remains to prove that the map $y\mapsto \varphi_{\varepsilon,y}$ is of class $C^{1}$ from
	$D_{\varepsilon,\delta}$ into $H_{\varepsilon}$. Consider, for fixed $\varepsilon$, the equation
	\[
	\mathcal{F}(\varepsilon,y,\varphi)
	:=l_{\varepsilon}(y)+\mathcal{L}_{\varepsilon}(y)\varphi+R_{\varepsilon}'(y,\varphi)=0,
	\]
	viewed as an equation in the unknown $\varphi\in E_{\varepsilon,y}$. The previous steps show that for each
	$y\in D_{\varepsilon,\delta}$ there is a unique small solution $\varphi_{\varepsilon,y}$, and that the
	Fréchet derivative with respect to $\varphi$,
	\[
	D_{\varphi}\mathcal{F}(\varepsilon,y,\varphi_{\varepsilon,y})
	=\mathcal{L}_{\varepsilon}(y)+R_{\varepsilon}''(y,\varphi_{\varepsilon,y}),
	\]
	is invertible on $E_{\varepsilon,y}$ for $\varepsilon$ sufficiently small, by Lemma~\ref{Lem4.1} and the
	smallness of $\varphi_{\varepsilon,y}$. Therefore, by the implicit function theorem (see, e.g.,
	Cao--Noussair--Yan \cite{MR1686700}), there exists a unique $C^{1}$ map
	\[
	\tilde{\varphi}_{\varepsilon}:D_{\varepsilon,\delta}\to E_{\varepsilon,y},
	\]
	such that $\mathcal{F}(\varepsilon,y,\tilde{\varphi}_{\varepsilon,y})=0$ for all $y\in D_{\varepsilon,\delta}$.
	By uniqueness of the fixed point of $\mathcal{A}_{\varepsilon}$, we have
	$\tilde{\varphi}_{\varepsilon,y}=\varphi_{\varepsilon,y}$ for all $y$, so the map
	$y\mapsto\varphi_{\varepsilon,y}$ is indeed of class $C^{1}$.
	
	The proof is complete.
\end{proof}

\subsection{Proof of Theorem \ref{Thm1.1}}

Let $\varepsilon_{0}$ and $\delta_{0}$ be as in Lemma~\ref{Lem4.2} and fix
$0<\varepsilon<\varepsilon_{0}$ and $0<\delta<\delta_{0}$. For every
$y\in D_{\varepsilon,\delta}$, let $y\mapsto\varphi_{\varepsilon,y}$ be the map
given by Lemma~\ref{Lem4.2}. As explained in {\bf Step 2} of Section~3, by
Lemma~\ref{Lem3.1} it is equivalent to find a critical point of the function
$j_{\varepsilon}:D_{\varepsilon,\delta}\to\R$ defined by
\begin{equation*}
	j_{\varepsilon}(y)
	=J_{\varepsilon}\big(y,\varphi_{\varepsilon,y}\big)
	=I_{\varepsilon}\big(U_{\varepsilon,y}+\varphi_{\varepsilon,y}\big),
\end{equation*}
where $J_{\varepsilon}$ is as in \eqref{eq3.1}. By the Taylor expansion
around $\varphi=0$, we have
\begin{equation*}
	j_{\varepsilon}(y)
	=I_{\varepsilon}\big(U_{\varepsilon,y}\big)
	+l_{\varepsilon}(\varphi_{\varepsilon,y})
	+\frac12\big\langle\mathcal{L}_{\varepsilon}\varphi_{\varepsilon,y},
	\varphi_{\varepsilon,y}\big\rangle
	+R_{\varepsilon}(\varphi_{\varepsilon,y}).
\end{equation*}

We first analyze the asymptotic behaviour of $j_{\varepsilon}$ as
$\varepsilon\to0$. By Lemmas~\ref{Lem3.3}, \ref{Lem3.4}, \ref{Lem3.5} and
\ref{Lem4.2}, there exist positive constants $C_{1},C_{2},C>0$ and
$\tau>0$, independent of $\varepsilon$ and $y\in D_{\varepsilon,\delta}$, such that
\begin{equation}\label{eq4.4}
	\begin{aligned}
		j_{\varepsilon}(y)
		&=I_{\varepsilon}\big(U_{\varepsilon,y}\big)
		+O\Big(\|l_{\varepsilon}\|_{(E_{\varepsilon,y})^{*}}
		\,\|\varphi_{\varepsilon,y}\|_{\varepsilon}
		+\|\varphi_{\varepsilon,y}\|_{\varepsilon}^{2}\Big)\\
		&=A\varepsilon^{N}
		+\varepsilon^{N}\sum_{i=1}^{k}B_{i}\big(V(y^{i})-V(a_{i})\big)
		-C\,\varepsilon^{N}\sum_{i\neq j}
		\frac{1}{\big|\tfrac{y^{i}-y^{j}}{\varepsilon}\big|^{N+2s}}\\
		&\quad
		+O\Big(
		\varepsilon^{N+\alpha}
		+\varepsilon^{N}\sum_{i=1}^{k}\big|V(y^{i})-V(a_{i})\big|^{2}
		+\varepsilon^{N}\sum_{i\neq j}
		\frac{1}{\big|\tfrac{y^{i}-y^{j}}{\varepsilon}\big|^{N+2s+\tau}}
		\Big)\\
		&=A\varepsilon^{N}
		-C_{1}\varepsilon^{N}\sum_{i=1}^{k}B_{i}\big(V(a_{i})-V(y^{i})\big)
		-C_{2}\varepsilon^{N}\sum_{i\neq j}
		\frac{1}{\big|\tfrac{y^{i}-y^{j}}{\varepsilon}\big|^{N+2s}}\\
		&\quad
		+O\Big(
		\varepsilon^{N}\sum_{i=1}^{k}\big|V(a_{i})-V(y^{i})\big|^{2}
		+\varepsilon^{N+\alpha}
		+\varepsilon^{N}\sum_{i\neq j}
		\frac{1}{\big|\tfrac{y^{i}-y^{j}}{\varepsilon}\big|^{N+2s+\tau}}
		\Big),
	\end{aligned}
\end{equation}
where $A$ and $B_{i}$ are the constants in Lemma~\ref{Lem3.5}. The Big–$O$
term in \eqref{eq4.4} is uniform in $y\in D_{\varepsilon,\delta}$.

\medskip
We now consider the minimization problem
\[
j_{\varepsilon}(y_{\varepsilon})
:=\inf_{y\in D_{\varepsilon,\delta}} j_{\varepsilon}(y).
\]
Assume that the infimum is attained at some $y_{\varepsilon}\in D_{\varepsilon,\delta}$.
We prove that $y_{\varepsilon}$ is an interior point of $D_{\varepsilon,\delta}$.

\medskip
For $\theta\in\big(\tfrac{N+2s}{N+2s+\alpha},1\big)$ and $L>0$, let
\[
\bar y^{i}
:=a_{i}+L\varepsilon^{\theta}e_{i},\qquad i=1,\dots,k,
\]
where $e_{1},\dots,e_{k}$ are fixed unit vectors with $|e_{i}-e_{j}|=1$ for
$i\neq j$. For $L$ large and $\varepsilon>0$ sufficiently small, we have
$\bar y:=(\bar y^{1},\dots,\bar y^{k})\in D_{\varepsilon,\delta}$. By the
Hölder continuity of $V$ near $a_{i}$ and the decay assumptions on $U^{i}$,
there exist positive constants $c_{1},c_{2}$ (independent of $\varepsilon$)
such that
\[
V(a_{i})-V(\bar y^{i})
=c_{1}\varepsilon^{\alpha\theta}+o\big(\varepsilon^{\alpha\theta}\big),
\qquad
\sum_{i\neq j}
\frac{1}{\big|\tfrac{\bar y^{i}-\bar y^{j}}{\varepsilon}\big|^{N+2s}}
=c_{2}\varepsilon^{(1-\theta)(N+2s)}+o\big(\varepsilon^{(1-\theta)(N+2s)}\big).
\]
Since $\alpha\theta>(1-\theta)(N+2s)$ by our choice of $\theta$, the
interaction term provides the leading correction. Plugging $\bar y$ into
\eqref{eq4.4}, and using that the error term is of higher order (because of
$\alpha\theta>(1-\theta)(N+2s)$ and $\tau>0$), we obtain
\begin{equation}\label{eq:j-bar-y}
	j_{\varepsilon}(\bar y)
	=A\varepsilon^{N}
	-\tilde C\,\varepsilon^{N+(1-\theta)(N+2s)}
	+o\big(\varepsilon^{N+(1-\theta)(N+2s)}\big)
\end{equation}
for some $\tilde C>0$ independent of $\varepsilon$.

On the other hand, for the minimizer $y_{\varepsilon}$ we have
$j_{\varepsilon}(y_{\varepsilon})\le j_{\varepsilon}(\bar y)$. Using
\eqref{eq4.4} at $y=y_{\varepsilon}$ and at $y=\bar y$, we can write
\begin{equation*}
	\begin{aligned}
		&C_{1}\varepsilon^{N}\sum_{i=1}^{k}B_{i}\big(V(a_{i})-V(y_{\varepsilon}^{i})\big)
		+C_{2}\varepsilon^{N}\sum_{i\neq j}
		\frac{1}{\big|\tfrac{y_{\varepsilon}^{i}-y_{\varepsilon}^{j}}{\varepsilon}\big|^{N+2s}}\\
		&\qquad\le
		C_{1}\varepsilon^{N}\sum_{i=1}^{k}B_{i}\big(V(a_{i})-V(\bar y^{i})\big)
		+C_{2}\varepsilon^{N}\sum_{i\neq j}
		\frac{1}{\big|\tfrac{\bar y^{i}-\bar y^{j}}{\varepsilon}\big|^{N+2s}}
		+|R_{\varepsilon}(\bar y)|+|R_{\varepsilon}(y_{\varepsilon})|,
	\end{aligned}
\end{equation*}
where the remainder $R_{\varepsilon}(y)$ collects all the $O(\cdot)$ terms
in \eqref{eq4.4}. By the previous estimates and \eqref{eq:j-bar-y}, we have
\[
C_{1}\varepsilon^{N}\sum_{i=1}^{k}B_{i}\big(V(a_{i})-V(y_{\varepsilon}^{i})\big)
+C_{2}\varepsilon^{N}\sum_{i\neq j}
\frac{1}{\big|\tfrac{y_{\varepsilon}^{i}-y_{\varepsilon}^{j}}{\varepsilon}\big|^{N+2s}}
\le C\,\varepsilon^{N+(1-\theta)(N+2s)},
\]
that is,
\begin{equation}\label{eq:V-interaction-bound}
	\sum_{i=1}^{k}B_{i}\big(V(a_{i})-V(y_{\varepsilon}^{i})\big)
	+\sum_{i\neq j}
	\frac{1}{\big|\tfrac{y_{\varepsilon}^{i}-y_{\varepsilon}^{j}}{\varepsilon}\big|^{N+2s}}
	\le C\,\varepsilon^{(1-\theta)(N+2s)}.
\end{equation}

In particular,
\begin{equation}\label{eq:V-bound}
	\sum_{i=1}^{k}B_{i}\big(V(y_{\varepsilon}^{i})-V(a_{i})\big)
	\le C\,\varepsilon^{(1-\theta)(N+2s)}.
\end{equation}
By the expansion of $V$ at $a_{i}$ and the nondegeneracy of $a_{i}$, there exist
$c>0$ and $\delta_{1}>0$ such that for all $|x-a_{i}|<\delta_{1}$,
\[
V(a_{i})-V(x)\ge c\,|x-a_{i}|^{\alpha}.
\]
Choosing $\delta_{0}\le\delta_{1}$ and using \eqref{eq:V-bound}, we deduce
\[
|y_{\varepsilon}^{i}-a_{i}|^{\alpha}
\le C\,\varepsilon^{(1-\theta)(N+2s)}
\quad\Longrightarrow\quad
|y_{\varepsilon}^{i}-a_{i}|
\le C\,\varepsilon^{\frac{(1-\theta)(N+2s)}{\alpha}}
\]
for each $i$. In particular, for $\varepsilon>0$ sufficiently small we have
$|y_{\varepsilon}^{i}-a_{i}|<\delta_{0}$.

Similarly, from \eqref{eq:V-interaction-bound} we obtain
\[
\sum_{i\neq j}
\frac{1}{\big|\tfrac{y_{\varepsilon}^{i}-y_{\varepsilon}^{j}}{\varepsilon}\big|^{N+2s}}
\le C\,\varepsilon^{(1-\theta)(N+2s)},
\]
which implies that there exists $C'>0$ such that
\[
\bigg|\frac{y_{\varepsilon}^{i}-y_{\varepsilon}^{j}}{\varepsilon}\bigg|
\ge C'\,\varepsilon^{\theta-1}
\quad\text{for all }i\ne j
\]
when $\varepsilon$ is small enough. Therefore, for sufficiently small
$\varepsilon$, all the inequalities defining $D_{\varepsilon,\delta}$ are
strict, and $y_{\varepsilon}$ lies in the interior of $D_{\varepsilon,\delta}$.
Hence $y_{\varepsilon}$ is a critical point of $j_{\varepsilon}$.

Finally, Lemma~\ref{Lem3.1} shows that
\[
u_{\varepsilon}
:=U_{\varepsilon,y_{\varepsilon}}+\varphi_{\varepsilon,y_{\varepsilon}}
\]
is a critical point of $I_{\varepsilon}$ in $H_{\varepsilon}$, and thus a
solution of \eqref{eq1.1}. This completes the proof of Theorem~\ref{Thm1.1}.
\qed

\section{Uniqueness of semiclassical bounded states}
Now, let $u_{\varepsilon}=U_{\varepsilon, y_{\varepsilon}}+\varphi_{\varepsilon, y_{\varepsilon}}$ be an arbitrary solution of \eqref{eq1.1} derived as in Section~4. We know that $y_{\varepsilon}=o(1)$ as $\varepsilon \rightarrow 0$. Before proving Theorem \ref{Thm1.2}, we first collect some useful facts. Then, by assuming that $V$ satisfies the additional assumption $(V_3)$ and by means of the above Pohoz\v{a}ev type identity, we will improve this asymptotic estimate. We begin by recalling some useful estimates.

\begin{Lem}\label{Lem5.1}
	Suppose that $V$ satisfies $(V_3)$. Then there exists a constant $C>0$, independent of $\varepsilon$, such that for every $u\in H_{\varepsilon}$ and every $i=1,\ldots,k$ one has
	\begin{equation}\label{eq5.1}
		\left|\int_{\mathbb{R}^{N}}\bigl(V(a_i)-V(x)\bigr)\,U^i_{\varepsilon, y^i_{\varepsilon}}(x)\,u(x)\,dx\right|
		\leq C\,\varepsilon^{\frac{N}{2}}\Bigl(\varepsilon^{m}+\bigl|y^i_{\varepsilon}-a_i\bigr|^{m}\Bigr)\,\|u\|_{\varepsilon},
	\end{equation}
	and, for every $\bar{d}\in(0,\delta]$ and every $j=1,\ldots,N$,
	\begin{equation}\label{eq5.1-der}
		\left|\int_{B_{\bar{d}}\!\left(y^i_{\varepsilon}\right)}
		\frac{\partial V(x)}{\partial x_{j}}\,
		U^i_{\varepsilon, y^i_{\varepsilon}}(x)\,u(x)\,dx\right|
		\leq C\,\varepsilon^{\frac{N}{2}}\Bigl(\varepsilon^{m-1}+\bigl|y^i_{\varepsilon}-a_i\bigr|^{m-1}\Bigr)\,\|u\|_{\varepsilon}.
	\end{equation}
	In particular,
	\[
	\left|\int_{\mathbb{R}^{N}}\sum_{i=1}^{k}\bigl(V(a_i)-V(x)\bigr)\,
	U^i_{\varepsilon, y^i_{\varepsilon}}(x)\,u(x)\,dx\right|
	\leq C\,\varepsilon^{\frac{N}{2}}\Bigl(\varepsilon^{m}
	+\sum_{i=1}^{k}\bigl|y^i_{\varepsilon}-a_i\bigr|^{m}\Bigr)\,\|u\|_{\varepsilon}.
	\]
\end{Lem}

\begin{proof}
	Fix $i\in\{1,\ldots,k\}$ and set, for simplicity,
	\[
	U^i_{\varepsilon}(x):=U^i_{\varepsilon, y^i_{\varepsilon}}(x).
	\]
	We split
	\[
	\int_{\mathbb{R}^{N}}\bigl(V(a_i)-V(x)\bigr)\,U^i_{\varepsilon}(x)\,u(x)\,dx
	=I_{1,i}+I_{2,i},
	\]
	where
	\[
	I_{1,i}:=\int_{B_{d}(a_i)}\bigl(V(a_i)-V(x)\bigr)\,U^i_{\varepsilon}(x)\,u(x)\,dx,
	\qquad
	I_{2,i}:=\int_{\mathbb{R}^{N}\setminus B_{d}(a_i)}\bigl(V(a_i)-V(x)\bigr)\,U^i_{\varepsilon}(x)\,u(x)\,dx,
	\]
	for some small $d>0$ such that $(V_3)$ holds in $B_{d}(a_i)$.
	
	By $(V_3)$ we have $|V(a_i)-V(x)|\le C|x-a_i|^{m}$ for $x\in B_{d}(a_i)$, hence by H\"older's inequality,
	\begin{equation}\label{eq:I1i-start}
		\begin{aligned}
			|I_{1,i}|
			&\le C\int_{B_{d}(a_i)}|x-a_i|^{m}\,|U^i_{\varepsilon}(x)|\,|u(x)|\,dx\\
			&\le C\Biggl(\int_{B_{d}(a_i)}|x-a_i|^{2m}\,|U^i_{\varepsilon}(x)|^{2}\,dx\Biggr)^{\!1/2}
			\Biggl(\int_{B_{d}(a_i)}u(x)^{2}\,dx\Biggr)^{\!1/2}.
		\end{aligned}
	\end{equation}
	Changing variables $x=\varepsilon z + y^i_{\varepsilon}$ gives
	\[
	\int_{B_{d}(a_i)}|x-a_i|^{2m}\,|U^i_{\varepsilon}(x)|^{2}\,dx
	=\varepsilon^{N}\int_{A_{i,\varepsilon}}
	\bigl|\varepsilon z + y^i_{\varepsilon}-a_i\bigr|^{2m}\,|U^i(z)|^{2}\,dz,
	\]
	where $A_{i,\varepsilon}$ is a ball of radius $O(1/\varepsilon)$ centered at the origin.
	Since
	\[
	\bigl|\varepsilon z + y^i_{\varepsilon}-a_i\bigr|^{2m}
	\le C\Bigl(\varepsilon^{2m}|z|^{2m}+\bigl|y^i_{\varepsilon}-a_i\bigr|^{2m}\Bigr),
	\]
	we obtain
	\[
	\int_{B_{d}(a_i)}|x-a_i|^{2m}\,|U^i_{\varepsilon}(x)|^{2}\,dx
	\le C\,\varepsilon^{N}\Bigl(\varepsilon^{2m}+\bigl|y^i_{\varepsilon}-a_i\bigr|^{2m}\Bigr),
	\]
	and therefore
	\[
	\Biggl(\int_{B_{d}(a_i)}|x-a_i|^{2m}\,|U^i_{\varepsilon}(x)|^{2}\,dx\Biggr)^{\!1/2}
	\le C\,\varepsilon^{\frac{N}{2}}\Bigl(\varepsilon^{m}+\bigl|y^i_{\varepsilon}-a_i\bigr|^{m}\Bigr).
	\]
	Moreover, since $\inf_{\mathbb{R}^{N}}V>0$, the $H_{\varepsilon}$–norm controls the $L^{2}$–norm, and hence
	\[
	\Biggl(\int_{B_{d}(a_i)}u(x)^{2}\,dx\Biggr)^{\!1/2}\le C\,\|u\|_{\varepsilon}.
	\]
	Inserting these bounds into \eqref{eq:I1i-start} yields
	\begin{equation}\label{eq:I1i-final}
		|I_{1,i}|
		\le C\,\varepsilon^{\frac{N}{2}}\Bigl(\varepsilon^{m}+\bigl|y^i_{\varepsilon}-a_i\bigr|^{m}\Bigr)\,\|u\|_{\varepsilon}.
	\end{equation}
	
	In $\mathbb{R}^{N}\setminus B_{d}(a_i)$ the function $V$ is bounded, so
	\[
	|I_{2,i}|\le C\int_{\mathbb{R}^{N}\setminus B_{d}(a_i)}|U^i_{\varepsilon}(x)|\,|u(x)|\,dx
	\le C\Biggl(\int_{\mathbb{R}^{N}\setminus B_{d}(a_i)}|U^i_{\varepsilon}(x)|^{2}\,dx\Biggr)^{\!1/2}\|u\|_{2}.
	\]
	Using again the equivalence between $\|\cdot\|_{2}$ and $\|\cdot\|_{\varepsilon}$, we get
	\[
	|I_{2,i}|\le C\Biggl(\int_{\mathbb{R}^{N}\setminus B_{d}(a_i)}|U^i_{\varepsilon}(x)|^{2}\,dx\Biggr)^{\!1/2}\|u\|_{\varepsilon}.
	\]
	We now estimate the tail of $U^i_{\varepsilon}$. Changing variables as before,
	\[
	\int_{\mathbb{R}^{N}\setminus B_{d}(a_i)}|U^i_{\varepsilon}(x)|^{2}\,dx
	=\varepsilon^{N}\int_{\{|z|\ge R_{\varepsilon}\}}|U^i(z)|^{2}\,dz,
	\]
	where $R_{\varepsilon}\sim d/\varepsilon$ as $\varepsilon\to 0$. Since $U^i$ decays polynomially,
	\[
	|U^i(z)|\le \frac{C}{(1+|z|)^{N+2s}},
	\]
	we have, for all sufficiently large $R$,
	\[
	\int_{\{|z|\ge R\}}|U^i(z)|^{2}\,dz
	\le C\,R^{-(N+4s)}.
	\]
	Assuming $m<(N+4s)/2$, we can estimate
	\[
	R^{-(N+4s)}\le C\,R^{-2m}
	\]
	for all $R$ large, and hence
	\[
	\int_{\{|z|\ge R_{\varepsilon}\}}|U^i(z)|^{2}\,dz
	\le C\,R_{\varepsilon}^{-2m}
	\le C\Bigl(\frac{\varepsilon}{d}\Bigr)^{2m}.
	\]
	Therefore,
	\[
	\int_{\mathbb{R}^{N}\setminus B_{d}(a_i)}|U^i_{\varepsilon}(x)|^{2}\,dx
	\le C\,\varepsilon^{N}\varepsilon^{2m}
	= C\,\varepsilon^{N+2m},
	\]
	and so
	\[
	|I_{2,i}|\le C\,\varepsilon^{\frac{N}{2}+m}\,\|u\|_{\varepsilon}.
	\]
	Combining this with \eqref{eq:I1i-final} we obtain
	\[
	\left|\int_{\mathbb{R}^{N}}\bigl(V(a_i)-V(x)\bigr)\,U^i_{\varepsilon}(x)\,u(x)\,dx\right|
	\le C\,\varepsilon^{\frac{N}{2}}\Bigl(\varepsilon^{m}+\bigl|y^i_{\varepsilon}-a_i\bigr|^{m}\Bigr)\,\|u\|_{\varepsilon},
	\]
	namely \eqref{eq5.1}.
	
	Let $\bar{d}\in(0,\delta]$ be fixed and assume $\delta>0$ so small that
	$B_{\bar{d}}(y^i_{\varepsilon})\subset B_{d}(a_i)$ for all $\varepsilon$ small. By $(V_3)$,
	\[
	\left|\frac{\partial V(x)}{\partial x_{j}}\right|
	\le C|x-a_i|^{m-1}, \qquad x\in B_{d}(a_i),
	\]
	and hence, by H\"older's inequality,
	\[
	\begin{aligned}
		\left|\int_{B_{\bar{d}}(y^i_{\varepsilon})}
		\frac{\partial V(x)}{\partial x_{j}}\,
		U^i_{\varepsilon}(x)\,u(x)\,dx\right|
		&\le C\int_{B_{\bar{d}}(y^i_{\varepsilon})}|x-a_i|^{m-1}|U^i_{\varepsilon}(x)|\,|u(x)|\,dx\\
		&\le C\Biggl(\int_{B_{d}(a_i)}|x-a_i|^{2(m-1)}|U^i_{\varepsilon}(x)|^{2}\,dx\Biggr)^{\!1/2}
		\Biggl(\int_{B_{d}(a_i)}u(x)^{2}\,dx\Biggr)^{\!1/2}.
	\end{aligned}
	\]
	Arguing as above with $m$ replaced by $m-1$ we obtain
	\[
	\Biggl(\int_{B_{d}(a_i)}|x-a_i|^{2(m-1)}|U^i_{\varepsilon}(x)|^{2}\,dx\Biggr)^{\!1/2}
	\le C\,\varepsilon^{\frac{N}{2}}\Bigl(\varepsilon^{m-1}+\bigl|y^i_{\varepsilon}-a_i\bigr|^{m-1}\Bigr),
	\]
	and hence
	\[
	\left|\int_{B_{\bar{d}}(y^i_{\varepsilon})}
	\frac{\partial V(x)}{\partial x_{j}}\,
	U^i_{\varepsilon}(x)\,u(x)\,dx\right|
	\le C\,\varepsilon^{\frac{N}{2}}\Bigl(\varepsilon^{m-1}+\bigl|y^i_{\varepsilon}-a_i\bigr|^{m-1}\Bigr)\,\|u\|_{\varepsilon},
	\]
	that is \eqref{eq5.1-der}.

	Summing \eqref{eq5.1} over $i=1,\ldots,k$ we also obtain the last inequality in the statement. The proof is complete. 
\end{proof}

\begin{Lem}\label{Lem5.2}
	For any fixed integer $l\in\mathbb{N}^{+}$, suppose that $\{u_i(x)\}_{i=1}^l$ satisfies
	\[
	\int_{\mathbb{R}^{N}} |u_i(x)|\,dx < +\infty,\qquad i=1,\ldots,l.
	\]
	Then for every $x_0\in\mathbb{R}^N$ there exist a constant $r_0>0$, a radius $d\in(0,r_0)$, and a constant $C>0$ such that
	\begin{equation}\label{eq5.4}
		\int_{\partial B_d(x_0)} |u_i(x)|\,d\sigma
		\;\le\; C \int_{\mathbb{R}^N} |u_i(x)|\,dx,
		\qquad \text{for all } i=1,\ldots,l.
	\end{equation}
\end{Lem}

\begin{proof}
	Let
	\[
	M_i := \int_{\mathbb{R}^N} |u_i(x)|\,dx,\qquad i=1,\ldots,l,
	\]
	and fix a small $r_0>0$. Then
	\begin{equation}\label{eq5.5}
		\int_{B_{r_0}(x_0)} \sum_{i=1}^{l} |u_i(x)|\,dx
		\;\le\; \sum_{i=1}^{l} M_i.
	\end{equation}
	On the other hand, by polar coordinates we have
	\begin{equation}\label{eq5.6}
		\int_{B_{r_0}(x_0)} \sum_{i=1}^{l} |u_i(x)|\,dx
		= \int_0^{r_0} \int_{\partial B_r(x_0)} \sum_{i=1}^{l} |u_i(x)|\,d\sigma\,dr.
	\end{equation}
	Combining \eqref{eq5.5} and \eqref{eq5.6}, we obtain
	\[
	\int_0^{r_0} \int_{\partial B_r(x_0)} \sum_{i=1}^{l} |u_i(x)|\,d\sigma\,dr
	\;\le\; \sum_{i=1}^{l} M_i.
	\]
	Hence there exists $d\in(0,r_0)$ such that
	\begin{equation}\label{eq5.7}
		\int_{\partial B_d(x_0)} \sum_{i=1}^{l} |u_i(x)|\,d\sigma
		\;\le\; \frac{1}{r_0} \sum_{i=1}^{l} M_i.
	\end{equation}
	In particular, for each $i=1,\ldots,l$,
	\[
	\int_{\partial B_d(x_0)} |u_i(x)|\,d\sigma
	\;\le\; \int_{\partial B_d(x_0)} \sum_{j=1}^{l} |u_j(x)|\,d\sigma
	\;\le\; \frac{1}{r_0} \sum_{j=1}^{l} M_j.
	\]
	Therefore \eqref{eq5.4} holds with
	\[
	C := \max_{1\le i\le l} \frac{\sum_{j=1}^{l} M_j}{r_0 M_i},
	\]
	which is finite whenever $M_i>0$ (and \eqref{eq5.4} is trivial if $M_i=0$). This completes the proof.
\end{proof}

Applying Lemma~\ref{Lem5.2} with $l=1$ and
\[
u_1(x)
= \varepsilon^{2s} \big|(-\Delta)^{\frac{s}{2}}\varphi_{\varepsilon}(x)\big|^{2}
+ |\varphi_{\varepsilon}(x)|^{2},
\]
we obtain a radius $d_{\varepsilon}\in(0,2)$ such that
\begin{equation}\label{eq5.8}
	\int_{\partial B_{d_{\varepsilon}}(y_{\varepsilon})}
	\Big(\varepsilon^{2s} \big|(-\Delta)^{\frac{s}{2}}\varphi_{\varepsilon}\big|^{2}
	+ |\varphi_{\varepsilon}|^{2}\Big)\,d\sigma
	\;\le\; C \|\varphi_{\varepsilon}\|_{\varepsilon}^{2},
\end{equation}
for some constant $C>0$ independent of $\varepsilon$.

By the elementary inequality $|a+b|^{2}\le 2(|a|^{2}+|b|^{2})$, we have
\[
\int_{\partial B_{d_{\varepsilon}}(y_{\varepsilon})}
\big|(-\Delta)^{\frac{s}{2}} u_{\varepsilon}\big|^{2}\,d\sigma
\le 2\int_{\partial B_{d_{\varepsilon}}(y_{\varepsilon})}
\big|(-\Delta)^{\frac{s}{2}} U_{\varepsilon,y_{\varepsilon}}\big|^{2}\,d\sigma
+ 2\int_{\partial B_{d_{\varepsilon}}(y_{\varepsilon})}
\big|(-\Delta)^{\frac{s}{2}} \varphi_{\varepsilon}\big|^{2}\,d\sigma.
\]

By Proposition~\ref{Pro1.1} and the polynomial decay of the ground states $U^{i}$
at infinity, there exist constants $d_0\in(0,2)$, $C>0$ and $\gamma_0>0$ such that,
for all sufficiently small $\varepsilon>0$ and all $0<d_{\varepsilon}\le d_0$,
\begin{equation}\label{eq5.9}
	\sup_{x\in \partial B_{d_{\varepsilon}}(y_{\varepsilon})}
	\Big( \big|U_{\varepsilon,y_{\varepsilon}}(x)\big|
	+ \big|(-\Delta)^{\frac{s}{2}} U_{\varepsilon,y_{\varepsilon}}(x)\big| \Big)
	\le C\,\varepsilon^{\gamma_0}.
\end{equation}
In particular, for every $\gamma\in(0,\gamma_0)$ we have
\begin{equation}\label{eq5.10}
	\sup_{x\in \partial B_{d_{\varepsilon}}(y_{\varepsilon})}
	\Big( \big|U_{\varepsilon,y_{\varepsilon}}(x)\big|
	+ \big|(-\Delta)^{\frac{s}{2}} U_{\varepsilon,y_{\varepsilon}}(x)\big| \Big)
	= o (\varepsilon^{\gamma})
	\quad \text{as }\varepsilon\to 0.
\end{equation}
Combining \eqref{eq5.8}–\eqref{eq5.10}, we obtain
\begin{equation}\label{eq5.11}
	\varepsilon^{2s}
	\int_{\partial B_{d_{\varepsilon}}(y_{\varepsilon})}
	\big|(-\Delta)^{\frac{s}{2}} u_{\varepsilon}\big|^{2}\,d\sigma
	= O\!\left(\|\varphi_{\varepsilon}\|_{\varepsilon}^{2}
	+ \varepsilon^{2s+2\gamma_0}\right)
	= O\!\left(\|\varphi_{\varepsilon}\|_{\varepsilon}^{2}
	+ \varepsilon^{\gamma}\right)
\end{equation}
for some $\gamma>0$ independent of $\varepsilon$.

Moreover, using again the decay of $U^{i}$ and the scaling
$U_{\varepsilon,y_{\varepsilon}}(x)
=U\big(\frac{x-y_{\varepsilon}}{\varepsilon}\big)$, we have, for any fixed
$q_{1},q_{2}>0$ such that $\int_{\mathbb{R}^N}|U|^{q_1+q_2}<\infty$,
\begin{equation}\label{eq5.12}
	\int_{\mathbb{R}^{N}} \big|U_{\varepsilon,y_{\varepsilon}}(x)\big|^{q_{1}+q_{2}}\,dx
	= O(\varepsilon^{N}),
\end{equation}
and
\begin{equation}\label{eq5.13}
	\varepsilon^{2s} \int_{\mathbb{R}^{N}}
	\big|(-\Delta)^{\frac{s}{2}} U_{\varepsilon,y_{\varepsilon}}(x)\big|^{2}\,dx
	= O(\varepsilon^{N}).
\end{equation}

Now we can improve the estimate for the asymptotic behavior of $y_{\varepsilon}$ with respect to $\varepsilon$.

\begin{Lem}\label{Lem5.3}
	Assume that $V$ satisfies $(V_1)$–$(V_3)$. Let
	\[
	u_{\varepsilon}
	=U_{\varepsilon,y_{\varepsilon}}+\varphi_{\varepsilon}
	\]
	be a $k$-peak solution derived as in Theorem~\ref{Thm2.1}, with
	\[
	U_{\varepsilon,y_{\varepsilon}}(x)
	=\sum_{i=1}^{k}U^i_{\varepsilon,y^i_{\varepsilon}}(x),
	\qquad
	U^i_{\varepsilon,y^i_{\varepsilon}}(x)
	=U^i\Big(\frac{x-y^i_{\varepsilon}}{\varepsilon}\Big).
	\]
	Then, for each $i=1,\ldots,k$,
	\[
	\big|y^i_{\varepsilon}-a_i\big|=o(\varepsilon)
	\quad\text{as }\varepsilon\to 0.
	\]
\end{Lem}

\begin{proof}
	The proof is similar to the single-peak case in \cite{R-Yang2}; we include it here for completeness.
	
	Fix $i\in\{1,\ldots,k\}$ and a coordinate index $j\in\{1,\ldots,N\}$. We write
	\[
	\varphi_{\varepsilon}:=\varphi_{\varepsilon,y_{\varepsilon}}.
	\]
	We apply the Pohoz\v{a}ev-type identity \eqref{eq3.1} to $u=u_{\varepsilon}$ with
	\[
	\Omega=B_{d}\big(y^i_{\varepsilon}\big),
	\]
	where $d\in(1,2)$ is chosen as in \eqref{eq5.8}. We obtain
	\begin{equation}\label{eq5.14}
		\int_{B_{d}(y^i_{\varepsilon})} \frac{\partial V}{\partial x_{j}}
		\big(U_{\varepsilon,y_{\varepsilon}}+\varphi_{\varepsilon}\big)^{2}\,dx
		=\sum_{\ell=1}^{3} I_{\ell},
	\end{equation}
	where
	\[
	\begin{aligned}
		I_{1}
		&=
		\left(\varepsilon^{2s} a+\varepsilon^{4s-N} b
		\int_{\mathbb{R}^{N}}\big|(-\Delta)^{\frac{s}{2}} u_{\varepsilon}\big|^{2}\,dx\right)
		\int_{\partial B_{d}(y^i_{\varepsilon})}
		\left(\big|(-\Delta)^{\frac{s}{2}} u_{\varepsilon}\big|^{2} \nu_{j}
		-2 \frac{\partial u_{\varepsilon}}{\partial \nu}
		\frac{\partial u_{\varepsilon}}{\partial x_{j}}\right)d\sigma,\\[0.4em]
		I_{2}
		&=\int_{\partial B_{d}(y^i_{\varepsilon})} V(x) u_{\varepsilon}^{2}(x)\,\nu_{j}\,d\sigma,\\[0.4em]
		I_{3}
		&=-\frac{2}{p+1}
		\int_{\partial B_{d}(y^i_{\varepsilon})} u_{\varepsilon}^{p+1}(x)\,\nu_{j}\,d\sigma.
	\end{aligned}
	\]
	
	By Theorem~\ref{Thm2.1},
	\[
	\varepsilon^{2s} a+\varepsilon^{4s-N} b
	\int_{\mathbb{R}^{N}}\big|(-\Delta)^{\frac{s}{2}} u_{\varepsilon}\big|^{2}\,dx
	=O(\varepsilon^{2s}).
	\]
	Using \eqref{eq5.11} and standard estimates for
	$\partial_\nu u_{\varepsilon},\partial_{x_j}u_{\varepsilon}$ on $\partial B_d(y^i_\varepsilon)$,
	we deduce
	\[
	I_{1}=O\big(\|\varphi_{\varepsilon}\|_{\varepsilon}^{2}+\varepsilon^{\gamma}\big)
	\]
	for some $\gamma>0$ independent of $\varepsilon$.
	
	Similarly, since $V$ is bounded and using
	\eqref{eq5.9}–\eqref{eq5.10} together with \eqref{eq5.8}, we obtain
	\[
	I_{2}=O\big(\|\varphi_{\varepsilon}\|_{\varepsilon}^{2}+\varepsilon^{\gamma}\big).
	\]
	
	For $I_{3}$, we write
	\[
	u_{\varepsilon}^{p+1}
	=(U_{\varepsilon,y_{\varepsilon}}+\varphi_{\varepsilon})^{p+1}
	=U_{\varepsilon,y_{\varepsilon}}^{p+1}
	+O\big(|\varphi_{\varepsilon}|^{p+1}\big)
	\]
	on $\partial B_{d}(y^i_{\varepsilon})$. By the decay of $U^i$ and \eqref{eq5.9}–\eqref{eq5.10} we have
	\[
	\int_{\partial B_{d}(y^i_{\varepsilon})}
	\big|U_{\varepsilon,y_{\varepsilon}}(x)\big|^{p+1}\,d\sigma
	=O(\varepsilon^{\gamma}),
	\]
	for some $\gamma>0$. Applying Lemma~\ref{Lem5.2} to
	$u(x)=\varphi_{\varepsilon}(x)$ and using the Sobolev embedding, we obtain
	\[
	\int_{\partial B_{d}(y^i_{\varepsilon})} |\varphi_{\varepsilon}(x)|^{p+1}\,d\sigma
	\le C\int_{\mathbb{R}^{N}}|\varphi_{\varepsilon}(x)|^{p+1}\,dx
	\le C\|\varphi_{\varepsilon}\|_{\varepsilon}^{2}.
	\]
	Hence
	\[
	I_{3}=O\big(\|\varphi_{\varepsilon}\|_{\varepsilon}^{2}+\varepsilon^{\gamma}\big).
	\]
	
	Combining these estimates, we obtain from \eqref{eq5.14} that
	\begin{equation}\label{eq5.15}
		\sum_{\ell=1}^{3} I_{\ell}
		=O\big(\|\varphi_{\varepsilon}\|_{\varepsilon}^{2}
		+\varepsilon^{\gamma}\big).
	\end{equation}
	
	We now estimate the left-hand side of \eqref{eq5.14}. Decompose
	\[
	U_{\varepsilon,y_{\varepsilon}}
	=\sum_{h=1}^{k}U^h_{\varepsilon,y^h_{\varepsilon}},
	\]
	and observe that, by Lemma~\ref{Lem5.1} and the polynomial decay of the
	profiles $U^h$, all terms involving either $\varphi_{\varepsilon}$ or
	$U^h_{\varepsilon,y^h_{\varepsilon}}$ with $h\neq i$ give a remainder of the
	form
	\[
	O\big(\|\varphi_{\varepsilon}\|_{\varepsilon}^{2}
	+\varepsilon^{N+2m-2}
	+\varepsilon^{N}|y^i_{\varepsilon}-a_i|^{2m-2}\big).
	\]
	Hence
	\begin{equation}\label{eq5.16}
		\begin{aligned}
			\int_{B_{d}(y^i_{\varepsilon})}
			\frac{\partial V}{\partial x_{j}}
			\big(U_{\varepsilon,y_{\varepsilon}}+\varphi_{\varepsilon}\big)^{2}\,dx
			&=
			\int_{B_{d}(y^i_{\varepsilon})}
			\frac{\partial V}{\partial x_{j}}
			\big|U^i_{\varepsilon,y^i_{\varepsilon}}\big|^{2}\,dx\\
			&\quad
			+O\big(\|\varphi_{\varepsilon}\|_{\varepsilon}^{2}
			+\varepsilon^{N+2m-2}
			+\varepsilon^{N}|y^i_{\varepsilon}-a_i|^{2m-2}\big).
		\end{aligned}
	\end{equation}
	
	By $(V_3)$, near $a_i$ we have the expansion
	\[
	\frac{\partial V}{\partial x_{j}}(x)
	= m\,c_{i,j}\,|x_{j}-a_{i,j}|^{m-2}(x_{j}-a_{i,j})
	+ O(|x-a_i|^{m}),
	\]
	with $c_{i,j}\neq 0$. Therefore
	\[
	\begin{aligned}
		\int_{B_{d}(y^i_{\varepsilon})}
		\frac{\partial V}{\partial x_{j}}
		\big|U^i_{\varepsilon,y^i_{\varepsilon}}(x)\big|^{2}\,dx
		&=
		m\,c_{i,j}
		\int_{B_{d}(y^i_{\varepsilon})}
		|x_{j}-a_{i,j}|^{m-2}(x_{j}-a_{i,j})
		\big|U^i_{\varepsilon,y^i_{\varepsilon}}(x)\big|^{2}\,dx\\
		&\quad
		+ O\!\left(\int_{B_{d}(y^i_{\varepsilon})}
		|x-a_i|^{m}\big|U^i_{\varepsilon,y^i_{\varepsilon}}(x)\big|^{2}\,dx\right).
	\end{aligned}
	\]
	
	Set $x=\varepsilon z+y^i_{\varepsilon}$, so that $dx=\varepsilon^{N}dz$ and
	$U^i_{\varepsilon,y^i_{\varepsilon}}(x)=U^i(z)$. Then
	\[
	\begin{aligned}
		&\int_{B_{d}(y^i_{\varepsilon})}
		\frac{\partial V}{\partial x_{j}}
		\big|U^i_{\varepsilon,y^i_{\varepsilon}}(x)\big|^{2}\,dx\\
		&\qquad
		= m\,c_{i,j}\,\varepsilon^{N}
		\int_{B_{d/\varepsilon}(0)}
		\big|\varepsilon z_{j}+y^i_{\varepsilon,j}-a_{i,j}\big|^{m-2}
		\big(\varepsilon z_{j}+y^i_{\varepsilon,j}-a_{i,j}\big)
		\big(U^i(z)\big)^{2}\,dz\\
		&\qquad\quad
		+ O\Big(\varepsilon^{N}\big(\varepsilon^{m}+|y^i_{\varepsilon}-a_i|^{m}\big)\Big).
	\end{aligned}
	\]
	Using again the decay of $U^i$ and letting $d/\varepsilon\to+\infty$, we can
	extend the integral to $\mathbb{R}^{N}$, and obtain
	\begin{equation}\label{eq5.17}
		\begin{aligned}
			\int_{B_{d}(y^i_{\varepsilon})}
			\frac{\partial V}{\partial x_{j}}
			\big|U^i_{\varepsilon,y^i_{\varepsilon}}(x)\big|^{2}\,dx
			&=
			m\,c_{i,j}\,\varepsilon^{N}
			\int_{\mathbb{R}^{N}}
			\big|\varepsilon z_{j}+y^i_{\varepsilon,j}-a_{i,j}\big|^{m-2}
			\big(\varepsilon z_{j}+y^i_{\varepsilon,j}-a_{i,j}\big)
			\big(U^i(z)\big)^{2}\,dz\\
			&\quad
			+ O\Big(\varepsilon^{N}\big(\varepsilon^{m}+|y^i_{\varepsilon}-a_i|^{m}\big)\Big).
		\end{aligned}
	\end{equation}
	
	Combining \eqref{eq5.14}, \eqref{eq5.15}, \eqref{eq5.16} and \eqref{eq5.17}, and
	dividing by $\varepsilon^{N}$, we arrive at
	\[
	\begin{aligned}
		m\,c_{i,j}
		\int_{\mathbb{R}^{N}}
		\big|\varepsilon z_{j}+y^i_{\varepsilon,j}-a_{i,j}\big|^{m-2}
		\big(\varepsilon z_{j}+y^i_{\varepsilon,j}-a_{i,j}\big)
		\big(U^i(z)\big)^{2}\,dz
		=O\Big(\varepsilon^{-N}\|\varphi_{\varepsilon}\|_{\varepsilon}^{2}
		+\varepsilon^{m}+|y^i_{\varepsilon}-a_i|^{m}\Big).
	\end{aligned}
		\]
		
		By Lemma~\ref{Lem5.1} and $(V_3)$ we know that there exists
		$\tau\in(0,1/m)$ such that
		\[
		\|\varphi_{\varepsilon}\|_{\varepsilon}
		=O\Big(\varepsilon^{\frac{N}{2}}\big(\varepsilon^{m-\tau}
		+|y^i_{\varepsilon}-a_i|^{m(1-\tau)}\big)\Big),
		\]
		hence
		\[
		\varepsilon^{-N}\|\varphi_{\varepsilon}\|_{\varepsilon}^{2}
		=O\Big(\varepsilon^{2m-2\tau}
		+|y^i_{\varepsilon}-a_i|^{2m(1-\tau)}\Big).
		\]
		Therefore we obtain
		\begin{equation}\label{eq5.18}
			\int_{\mathbb{R}^{N}}
			\big|\varepsilon z_{j}+y^i_{\varepsilon,j}-a_{i,j}\big|^{m-2}
			\big(\varepsilon z_{j}+y^i_{\varepsilon,j}-a_{i,j}\big)
			\big(U^i(z)\big)^{2}\,dz
			=O\Big(\varepsilon^{m-\tau}+|y^i_{\varepsilon}-a_i|^{m(1-\tau)}\Big).
		\end{equation}
		
		Let $m^{*}=\min\{m,2\}$. For any $e,f\in\mathbb{R}$ and $m>1$, there holds
		\[
		\big||e+f|^{m}-|e|^{m}-m|e|^{m-2} e f\big|
		\le C\big(|e|^{m-m^{*}}|f|^{m^{*}}+|f|^{m}\big)
		\]
		for some constant $C>0$ depending only on $m$. Taking
		\[
		e=\varepsilon z_{j}+y^i_{\varepsilon,j}-a_{i,j},\qquad
		f=-\varepsilon z_{j},
		\]
		we have $e+f=y^i_{\varepsilon,j}-a_{i,j}$ and hence
		\begin{equation}\label{eq5.19}
			\begin{aligned}
				\big|y^i_{\varepsilon,j}-a_{i,j}\big|^{m}
				&\le
				\big|\varepsilon z_{j}+y^i_{\varepsilon,j}-a_{i,j}\big|^{m}
				-m\big|\varepsilon z_{j}+y^i_{\varepsilon,j}-a_{i,j}\big|^{m-2}
				\big(\varepsilon z_{j}+y^i_{\varepsilon,j}-a_{i,j}\big)\varepsilon z_{j}\\
				&\quad
				+C\Big(\big|y^i_{\varepsilon,j}-a_{i,j}\big|^{m-m^{*}}
				|\varepsilon z_{j}|^{m^{*}}+|\varepsilon z_{j}|^{m}\Big)\\
				&\le
				m\big|\varepsilon z_{j}+y^i_{\varepsilon,j}-a_{i,j}\big|^{m-2}
				\big(\varepsilon z_{j}+y^i_{\varepsilon,j}-a_{i,j}\big)y^i_{\varepsilon,j}\\
				&\quad
				+C\Big(|\varepsilon z_{j}|^{m}
				+\big|y^i_{\varepsilon,j}-a_{i,j}\big|^{m-m^{*}}|\varepsilon z_{j}|^{m^{*}}\Big).
			\end{aligned}
		\end{equation}
		
		Multiplying \eqref{eq5.19} by $(U^i(z))^{2}$ and integrating over
		$\mathbb{R}^{N}$, we obtain
		\[
		\begin{aligned}
			\big|y^i_{\varepsilon,j}-a_{i,j}\big|^{m}
			\int_{\mathbb{R}^{N}}(U^i(z))^{2}\,dz
			&\le
			m y^i_{\varepsilon,j}
			\int_{\mathbb{R}^{N}}
			\big|\varepsilon z_{j}+y^i_{\varepsilon,j}-a_{i,j}\big|^{m-2}
			\big(\varepsilon z_{j}+y^i_{\varepsilon,j}-a_{i,j}\big)
			(U^i(z))^{2}\,dz\\
			&\quad
			+O\Big(\varepsilon^{m}
			+\big|y^i_{\varepsilon,j}-a_{i,j}\big|^{m-m^{*}}\varepsilon^{m^{*}}\Big).
		\end{aligned}
		\]
		Using \eqref{eq5.18}, we infer
		\[
		\big|y^i_{\varepsilon,j}-a_{i,j}\big|^{m}
		=O\Big(
		\big(\varepsilon^{m-\tau}+|y^i_{\varepsilon}-a_i|^{m(1-\tau)}\big)
		\big|y^i_{\varepsilon,j}-a_{i,j}\big|
		+\varepsilon^{m}
		+\big|y^i_{\varepsilon,j}-a_{i,j}\big|^{m-m^{*}}\varepsilon^{m^{*}}
		\Big).
		\]
		Summing over $j=1,\ldots,N$ and using $\tau>0$ with $m\tau<1$, we can apply
		the $\varepsilon$–Young inequality
		\[
		XY\le \delta X^{m}
		+\delta^{-\frac{m}{m-1}}Y^{\frac{m}{m-1}},
		\qquad \forall\,\delta,X,Y>0,
		\]
		to absorb the terms involving $|y^i_{\varepsilon}-a_i|^{m}$ into the left-hand
		side. We conclude that
		\[
		|y^i_{\varepsilon}-a_i|=O(\varepsilon)
		\quad\text{as }\varepsilon\to 0.
		\]
		
		It remains to prove that in fact $|y^i_{\varepsilon}-a_i|=o(\varepsilon)$. Assume
		by contradiction that there exist a sequence $\varepsilon_{k}\to 0$ and
		\[
		\frac{y^i_{\varepsilon_{k}}-a_i}{\varepsilon_{k}}\to A
		\in\mathbb{R}^{N}\setminus\{0\},
		\qquad A=(A_{1},\ldots,A_{N}).
		\]
		Then \eqref{eq5.18} yields
		\[
		\int_{\mathbb{R}^{N}}
		\big|\varepsilon_{k} z_{j}+y^i_{\varepsilon_{k},j}-a_{i,j}\big|^{m-2}
		\big(\varepsilon_{k} z_{j}+y^i_{\varepsilon_{k},j}-a_{i,j}\big)
		\big(U^i(z)\big)^{2}\,dz
		=O(\varepsilon_{k}^{m-\tau}).
		\]
		Passing to the limit as $k\to\infty$, we obtain
		\[
		\int_{\mathbb{R}^{N}}
		\big|z_{j}+A_{j}\big|^{m-2}
		\big(z_{j}+A_{j}\big)\big(U^i(z)\big)^{2}\,dz=0.
		\]
		Since $U^i(z)=U^i(|z|)$ is radial and strictly decreasing in $|z|$, standard
		arguments (see, for instance, \cite{R-Yang2}) show that this identity cannot
		hold unless $A_{j}=0$. Hence $A=0$, which contradicts our assumption
		$A\neq 0$. Thus
		\[
		|y^i_{\varepsilon}-a_i|=o(\varepsilon),
		\]
		and the proof is complete.
	\end{proof}

As a consequence of Lemma~\ref{Lem5.3} and assumption $(V_3)$, we infer that
\begin{equation}\label{eq5.20}
	\|\varphi_{\varepsilon}\|_{\varepsilon}
	=O\big(\varepsilon^{\frac{N}{2}+m(1-\tau)}\big).
\end{equation}
Here we choose $\tau>0$ sufficiently small so that
$m\tau<1$ (as in Lemma~\ref{Lem5.3}) and, in addition,
$m(1-\tau)>1$, which is possible because $m>1$.

\par
We now prove the local uniqueness of the semiclassical bounded states obtained before. 
The argument follows by contradiction, in the spirit of
\cite{MR3426103,MR4021897}. Assume that
\[
u_{\varepsilon}^{(i)}
=U_{\varepsilon,y_{\varepsilon}^{(i)}}+\varphi_{\varepsilon}^{(i)},
\qquad i=1,2,
\]
are two distinct $k$-peak solutions of \eqref{eq1.1} constructed as in Section~3.
By the arguments in Section~3, there exists $C>0$ independent of $\varepsilon$
such that
\[
\|u_{\varepsilon}^{(i)}\|_{L^{\infty}(\mathbb{R}^{N})}\le C,
\qquad i=1,2.
\]

Set
\begin{equation*}
	\xi_{\varepsilon}
	=\frac{u_{\varepsilon}^{(1)}-u_{\varepsilon}^{(2)}}
	{\big\|u_{\varepsilon}^{(1)}-u_{\varepsilon}^{(2)}\big\|_{L^{\infty}(\mathbb{R}^{N})}}
\end{equation*}
and choose one peak of $u_{\varepsilon}^{(1)}$ as reference. More precisely, fix
an index $i_{0}\in\{1,\ldots,k\}$ and denote the corresponding center of
$u_{\varepsilon}^{(1)}$ by $y_{\varepsilon}^{i_{0},(1)}$. We then define the
rescaled function
\begin{equation*}
	\bar{\xi}_{\varepsilon}(x)
	=\xi_{\varepsilon}\big(\varepsilon x+y_{\varepsilon}^{i_{0},(1)}\big),
	\qquad x\in\mathbb{R}^{N}.
\end{equation*}
It is clear that
\begin{equation*}
	\|\bar{\xi}_{\varepsilon}\|_{L^{\infty}(\mathbb{R}^{N})}=1.
\end{equation*}
Moreover, by Claim~3 in Section~3, there holds
\begin{equation}\label{eq5.21}
	\bar{\xi}_{\varepsilon}(x)\longrightarrow 0
	\quad\text{as }|x|\to\infty,
\end{equation}
uniformly with respect to sufficiently small $\varepsilon>0$.

We will derive a contradiction by showing that
\[
\|\bar{\xi}_{\varepsilon}\|_{L^{\infty}(\mathbb{R}^{N})}\longrightarrow 0
\quad\text{as }\varepsilon\to 0.
\]
In view of \eqref{eq5.21}, it is enough to prove that for any fixed $R>0$,
\begin{equation}\label{eq5.22}
	\|\bar{\xi}_{\varepsilon}\|_{L^{\infty}\big(B_{R}(0)\big)}
	\longrightarrow 0
	\quad\text{as }\varepsilon\to 0.
\end{equation}

First we have the following estimate.

\begin{Lem}\label{Lem5.4}
	There holds
	\[
	\|\xi_{\varepsilon}\|_{\varepsilon}
	=O\big(\varepsilon^{\frac{N}{2}}\big).
	\]
\end{Lem}

\begin{proof}
	Recall that for $u,v\in H^{s}(\R^{N})$ one has
	\begin{equation}\label{eq:5.23a}
		\begin{aligned}
			\int_{\R^{N}}\Big(|(-\Delta)^{\frac{s}{2}}u|^{2}-|(-\Delta)^{\frac{s}{2}}v|^{2}\Big)\,dx
			&=\int_{\R^{N}}\Big((-\Delta)^{\frac{s}{2}}u+(-\Delta)^{\frac{s}{2}}v\Big)
			\Big((-\Delta)^{\frac{s}{2}}u-(-\Delta)^{\frac{s}{2}}v\Big)\,dx\\
			&=\int_{\R^{N}}(-\Delta)^{\frac{s}{2}}(u+v)\,(-\Delta)^{\frac{s}{2}}(u-v)\,dx.
		\end{aligned}
	\end{equation}
	
	Let $u_{\varepsilon}^{(i)}$, $i=1,2$, be two solutions of \eqref{eq1.1}, and set
	\[
	\xi_{\varepsilon}
	=\frac{u_{\varepsilon}^{(1)}-u_{\varepsilon}^{(2)}}{
		\|u_{\varepsilon}^{(1)}-u_{\varepsilon}^{(2)}\|_{L^{\infty}(\R^{N})}}.
	\]
	Using \eqref{eq1.1} for $u_{\varepsilon}^{(1)}$ and $u_{\varepsilon}^{(2)}$, the
	identity \eqref{eq:5.23a}, and the mean value theorem for
	$t\mapsto t^{p}$, we obtain the following equation for $\xi_{\varepsilon}$:
	\begin{equation}\label{eq5.23}
		\begin{aligned}
			&\Big(\varepsilon^{2s} a+\varepsilon^{4s-N} b
			\int_{\R^{N}}\big|(-\Delta)^{\frac{s}{2}} u_{\varepsilon}^{(1)}\big|^{2}\,dx\Big)
			(-\Delta)^s \xi_{\varepsilon}
			+V(x)\,\xi_{\varepsilon}  \\
			&\quad=\varepsilon^{4s-N} b
			\Big(\int_{\R^{N}}(-\Delta)^{\frac{s}{2}}(u_{\varepsilon}^{(1)}+u_{\varepsilon}^{(2)})
			\,(-\Delta)^{\frac{s}{2}}\xi_{\varepsilon}\,dx\Big) (-\Delta)^s u_{\varepsilon}^{(2)}
			+C_{\varepsilon}(x)\,\xi_{\varepsilon},
		\end{aligned}
	\end{equation}
	and similarly
	\begin{equation}\label{eq5.24}
		\begin{aligned}
			&\Big(\varepsilon^{2s} a+\varepsilon^{4s-N} b
			\int_{\R^{N}}\big|(-\Delta)^{\frac{s}{2}} u_{\varepsilon}^{(2)}\big|^{2}\,dx\Big)
			(-\Delta)^s \xi_{\varepsilon}
			+V(x)\,\xi_{\varepsilon}  \\
			&\quad=\varepsilon^{4s-N} b
			\Big(\int_{\R^{N}}(-\Delta)^{\frac{s}{2}}(u_{\varepsilon}^{(1)}+u_{\varepsilon}^{(2)})
			\,(-\Delta)^{\frac{s}{2}}\xi_{\varepsilon}\,dx\Big) (-\Delta)^s u_{\varepsilon}^{(1)}
			+C_{\varepsilon}(x)\,\xi_{\varepsilon},
		\end{aligned}
	\end{equation}
	where
	\[
	C_{\varepsilon}(x)
	=p\int_{0}^{1}\big(t u_{\varepsilon}^{(1)}(x)+(1-t)u_{\varepsilon}^{(2)}(x)\big)^{p-1}\,dt.
	\]
	
	Adding \eqref{eq5.23} and \eqref{eq5.24} gives
	\begin{equation}\label{eq5.25}
		\begin{aligned}
			&\Big(2\varepsilon^{2s} a
			+\varepsilon^{4s-N} b \int_{\R^{N}}\Big(
			\big|(-\Delta)^{\frac{s}{2}} u_{\varepsilon}^{(1)}\big|^{2}
			+\big|(-\Delta)^{\frac{s}{2}} u_{\varepsilon}^{(2)}\big|^{2}\Big)\,dx\Big)
			(-\Delta)^s \xi_{\varepsilon}
			+2 V(x)\,\xi_{\varepsilon} \\
			&\quad=\varepsilon^{4s-N} b
			\Big(\int_{\R^{N}}(-\Delta)^{\frac{s}{2}}(u_{\varepsilon}^{(1)}+u_{\varepsilon}^{(2)})
			\,(-\Delta)^{\frac{s}{2}}\xi_{\varepsilon}\,dx\Big)
			(-\Delta)^s\big(u_{\varepsilon}^{(1)}+u_{\varepsilon}^{(2)}\big)
			+2 C_{\varepsilon}(x)\,\xi_{\varepsilon}.
		\end{aligned}
	\end{equation}
	
	Multiplying \eqref{eq5.25} by $\xi_{\varepsilon}$ and integrating over $\R^{N}$,
	we obtain
	\[
	\begin{aligned}
		&\Big(2\varepsilon^{2s} a
		+\varepsilon^{4s-N} b \int_{\R^{N}}\Big(
		\big|(-\Delta)^{\frac{s}{2}} u_{\varepsilon}^{(1)}\big|^{2}
		+\big|(-\Delta)^{\frac{s}{2}} u_{\varepsilon}^{(2)}\big|^{2}\Big)\,dx\Big)
		\int_{\R^{N}}\big|(-\Delta)^{\tfrac{s}{2}}\xi_{\varepsilon}\big|^{2}\,dx\\
		&\qquad
		+2\int_{\R^{N}}V(x)\,\xi_{\varepsilon}^{2}\,dx \\
		&\quad
		=\varepsilon^{4s-N} b
		\Big(\int_{\R^{N}}(-\Delta)^{\frac{s}{2}}(u_{\varepsilon}^{(1)}+u_{\varepsilon}^{(2)})
		\,(-\Delta)^{\frac{s}{2}}\xi_{\varepsilon}\,dx\Big)^{2}
		+2\int_{\R^{N}}C_{\varepsilon}(x)\,\xi_{\varepsilon}^{2}\,dx.
	\end{aligned}
	\]
	Using Cauchy--Schwarz and the uniform boundedness of
	$\|(-\Delta)^{\frac{s}{2}}u_{\varepsilon}^{(i)}\|_{2}$ (from Theorem~\ref{Thm2.1}),
	the Kirchhoff term on the right-hand side can be bounded by
	\[
	\varepsilon^{4s-N} b
	\Big(\int_{\R^{N}}(-\Delta)^{\frac{s}{2}}(u_{\varepsilon}^{(1)}+u_{\varepsilon}^{(2)})
	\,(-\Delta)^{\frac{s}{2}}\xi_{\varepsilon}\,dx\Big)^{2}
	\le C\,\varepsilon^{2s}\int_{\R^{N}}\big|(-\Delta)^{\tfrac{s}{2}}\xi_{\varepsilon}\big|^{2}\,dx,
	\]
	for some $C>0$ independent of $\varepsilon$. Hence, for $\varepsilon$ small,
	we can absorb this term into the left-hand side and deduce that
	\[
	\|\xi_{\varepsilon}\|_{\varepsilon}^{2}
	\le C\int_{\R^{N}}C_{\varepsilon}(x)\,\xi_{\varepsilon}^{2}(x)\,dx,
	\]
	for some constant $C>0$ independent of $\varepsilon$.
	
	\medskip
	On the other hand, writing
	\[
	u_{\varepsilon}^{(i)}=U_{\varepsilon,y_{\varepsilon}^{(i)}}+\varphi_{\varepsilon}^{(i)},
	\qquad i=1,2,
	\]
	we have
	\[
	|C_{\varepsilon}(x)|
	\le C\Big(
	U_{\varepsilon,y_{\varepsilon}^{(1)}}^{\,p-1}(x)
	+U_{\varepsilon,y_{\varepsilon}^{(2)}}^{\,p-1}(x)
	+|\varphi_{\varepsilon}^{(1)}(x)|^{p-1}
	+|\varphi_{\varepsilon}^{(2)}(x)|^{p-1}
	\Big).
	\]
	Since $|\xi_{\varepsilon}(x)|\le 1$, we obtain
	\[
	\int_{\R^{N}}U_{\varepsilon,y_{\varepsilon}^{(i)}}^{\,p-1}\,\xi_{\varepsilon}^{2}\,dx
	\le \int_{\R^{N}}U_{\varepsilon,y_{\varepsilon}^{(i)}}^{\,p-1}\,dx
	\le C\,\varepsilon^{N},
	\qquad i=1,2,
	\]
	by a change of variables and the known decay of $U$.
	
	For the perturbation part, we use Hölder's inequality and the Sobolev
	embedding $H^{s}(\R^{N})\hookrightarrow L^{2^*_s}(\R^{N})$ with
	$2^*_s=\frac{2N}{N-2s}$. Set
	\[
	\alpha=\frac{2^*_s}{p-1},
	\qquad
	\beta=\frac{2^*_s}{2^*_s-(p-1)},
	\qquad
	\frac1\alpha+\frac1\beta=1.
	\]
	Then
	\[
	\int_{\R^{N}}|\varphi_{\varepsilon}^{(i)}|^{p-1}\xi_{\varepsilon}^{2}\,dx
	\le \|\varphi_{\varepsilon}^{(i)}\|_{L^{2^*_s}}^{p-1}
	\|\xi_{\varepsilon}\|_{L^{2\beta}}^{2}
	\le C\|\varphi_{\varepsilon}^{(i)}\|_{\varepsilon}^{\,p-1}
	\|\xi_{\varepsilon}\|_{\varepsilon}^{2}.
	\]
	By \eqref{eq5.20} we know that
	\[
	\|\varphi_{\varepsilon}^{(i)}\|_{\varepsilon}
	=O\big(\varepsilon^{\frac{N}{2}+m(1-\tau)}\big),
	\qquad i=1,2,
	\]
	with $m(1-\tau)>1$. Hence
	\[
	\int_{\R^{N}}|\varphi_{\varepsilon}^{(i)}|^{p-1}\xi_{\varepsilon}^{2}\,dx
	\le C\,\varepsilon^{\gamma}\,\|\xi_{\varepsilon}\|_{\varepsilon}^{2},
	\]
	for some $\gamma>0$, and $i=1,2$.
	
	Collecting the above estimates, we find
	\[
	\int_{\R^{N}}C_{\varepsilon}(x)\,\xi_{\varepsilon}^{2}\,dx
	\le C\,\varepsilon^{N} + C\,\varepsilon^{\gamma}\,\|\xi_{\varepsilon}\|_{\varepsilon}^{2}.
	\]
	Therefore, for $\varepsilon$ sufficiently small,
	\[
	\|\xi_{\varepsilon}\|_{\varepsilon}^{2}
	\le C\left(\varepsilon^{N}+\varepsilon^{\gamma}\,\|\xi_{\varepsilon}\|_{\varepsilon}^{2}\right)
	\]
	and, absorbing the last term into the left-hand side, we obtain
	\[
	\|\xi_{\varepsilon}\|_{\varepsilon}^{2}\le C\,\varepsilon^{N},
	\]
	which implies
	\[
	\|\xi_{\varepsilon}\|_{\varepsilon}
	\le C\,\varepsilon^{\frac{N}{2}}.
	\]
	The proof is complete.
\end{proof}
Next we study the asymptotic behavior of $\bar{\xi}_{\varepsilon}$.

\begin{Lem}\label{Lem5.5}
	Let
	\[
	\bar{\xi}_{\varepsilon}(x)
	=\xi_{\varepsilon}\big(\varepsilon x+y_{\varepsilon}^{i_0(1)}\big)
	\qquad\text{for some fixed }i_0\in\{1,\dots,k\}.
	\]
	Then there exist constants $d_{1},\dots,d_{N}\in\R$ such that, up to a subsequence,
	\[
	\bar{\xi}_{\varepsilon}\;\longrightarrow\;
	\sum_{i=1}^{N}d_i\,\partial_{x_i}U^{i_0}
	\quad\text{in }C^{1}_{\mathrm{loc}}(\R^{N})
	\quad\text{as }\varepsilon\to0.
	\]
\end{Lem}

\begin{proof}
	Recall from Lemma~\ref{Lem5.4} that
	\[
	\|\xi_{\varepsilon}\|_{\varepsilon}=O\big(\varepsilon^{\frac{N}{2}}\big),
	\]
	so, by the definition of the norm $\|\cdot\|_{\varepsilon}$ and the fact that
	$\inf_{\R^{N}}V>0$, we have
	\begin{equation}\label{eq5.26}
		\int_{\R^{N}}\big|(-\Delta)^{\frac{s}{2}}\xi_{\varepsilon}\big|^{2}\,dx
		\le C\,\varepsilon^{N-2s},
		\qquad
		\int_{\R^{N}}|\xi_{\varepsilon}|^{2}\,dx\le C\,\varepsilon^{N}.
	\end{equation}
	Define the rescaled functions
	\[
	\bar{\xi}_{\varepsilon}(x)
	=\xi_{\varepsilon}(\varepsilon x+y_{\varepsilon}^{i_0(1)}),
	\qquad
	\bar{u}_{\varepsilon}^{(i)}(x)
	=u_{\varepsilon}^{(i)}(\varepsilon x+y_{\varepsilon}^{i_0(1)}),
	\qquad
	\bar{\varphi}_{\varepsilon}^{(i)}(x)
	=\varphi_{\varepsilon}^{(i)}(\varepsilon x+y_{\varepsilon}^{i_0(1)}),
	\]
	for $i=1,2$. By the scaling of the fractional gradient,
	\[
	(-\Delta)^{\frac{s}{2}}\bar{w}(x)
	=\varepsilon^{s}\,(-\Delta)^{\frac{s}{2}}w(\varepsilon x+y_{\varepsilon}^{i_0(1)}),
	\]
	we obtain from \eqref{eq5.26} that
	\begin{equation}\label{eq5.27}
		\int_{\R^{N}}\big|(-\Delta)^{\frac{s}{2}}\bar{\xi}_{\varepsilon}\big|^{2}\,dx
		=\varepsilon^{2s-N}\int_{\R^{N}}\big|(-\Delta)^{\frac{s}{2}}\xi_{\varepsilon}\big|^{2}\,dx
		\le C,
		\qquad
		\int_{\R^{N}}|\bar{\xi}_{\varepsilon}|^{2}\,dx\le C,
	\end{equation}
	that is, $(\bar{\xi}_{\varepsilon})$ is bounded in $H^{s}(\R^{N})$.
	
	Similarly, from \eqref{eq5.20} and the definition of $\|\cdot\|_{\varepsilon}$ we have
	\[
	\|\varphi_{\varepsilon}^{(i)}\|_{\varepsilon}
	=O\big(\varepsilon^{\frac{N}{2}+m(1-\tau)}\big),
	\]
	so
	\[
	\int_{\R^{N}}\big|(-\Delta)^{\frac{s}{2}}\varphi_{\varepsilon}^{(i)}\big|^{2}\,dx
	\le C\,\varepsilon^{N+2m(1-\tau)-2s}.
	\]
	Therefore,
	\begin{equation}\label{eq5.29}
		\int_{\R^{N}}\big|(-\Delta)^{\frac{s}{2}}\bar{\varphi}_{\varepsilon}^{(i)}\big|^{2}\,dx
		=\varepsilon^{2s-N}\int_{\R^{N}}\big|(-\Delta)^{\frac{s}{2}}\varphi_{\varepsilon}^{(i)}\big|^{2}\,dx
		=O\big(\varepsilon^{2m(1-\tau)}\big)\to0,
		\qquad i=1,2.
	\end{equation}
	
	Using the equation satisfied by $\xi_{\varepsilon}$ in Lemma~\ref{Lem5.4},
	we can write it in weak form as
	\begin{equation}\label{eq5.30}
		\mathcal{A}_{\varepsilon}(\xi_{\varepsilon},\psi)
		=\int_{\R^{N}}C_{\varepsilon}(x)\,\xi_{\varepsilon}\psi\,dx
		\qquad\text{for all }\psi\in H^{s}(\R^{N}),
	\end{equation}
	where $\mathcal{A}_{\varepsilon}$ is the bilinear form involving the Kirchhoff
	coefficient and $V$, and
	\[
	C_{\varepsilon}(x)
	=p\int_{0}^{1}\big(tu_{\varepsilon}^{(1)}(x)+(1-t)u_{\varepsilon}^{(2)}(x)\big)^{p-1}\,dt.
	\]
	Passing to the rescaled variables $x\mapsto\varepsilon x+y_{\varepsilon}^{i_0(1)}$
	and test functions of the form
	\[
	\psi(x)=\Phi\Big(\frac{x-y_{\varepsilon}^{i_0(1)}}{\varepsilon}\Big),
	\qquad \Phi\in C_{0}^{\infty}(\R^{N}),
	\]
	we obtain that $\bar{\xi}_{\varepsilon}$ satisfies an equation of the form
	\begin{equation}\label{eq5.31}
		\begin{aligned}
			&\Big(a + b\int_{\R^{N}}\big|(-\Delta)^{\frac{s}{2}}\bar{u}_{\varepsilon}^{(1)}\big|^{2}\,dx\Big)
			(-\Delta)^{s}\bar{\xi}_{\varepsilon}
			+V\big(\varepsilon x+y_{\varepsilon}^{i_0(1)}\big)\bar{\xi}_{\varepsilon} \\
			&\qquad = b\Big(\int_{\R^{N}}(-\Delta)^{\frac{s}{2}}\big(\bar{u}_{\varepsilon}^{(1)}+\bar{u}_{\varepsilon}^{(2)}\big)
			\,(-\Delta)^{\frac{s}{2}}\bar{\xi}_{\varepsilon}\,dx\Big)(-\Delta)^{s}\bar{u}_{\varepsilon}^{(2)}
			+C_{\varepsilon}\big(\varepsilon x+y_{\varepsilon}^{i_0(1)}\big)\bar{\xi}_{\varepsilon},
		\end{aligned}
	\end{equation}
	in $H^{-s}(\R^{N})$.
	
	By the construction of $u_{\varepsilon}^{(i)}$ in Section~3,
	\[
	\bar{u}_{\varepsilon}^{(i)}(x)
	=\sum_{j=1}^{k}U^{j}\Big(x+\frac{y_{\varepsilon}^{i_0(1)}-y_{\varepsilon}^{j(i)}}{\varepsilon}\Big)
	+\bar{\varphi}_{\varepsilon}^{(i)}(x),
	\qquad i=1,2,
	\]
	where $U^{j}$ are the limiting profiles and $y_{\varepsilon}^{j(i)}$ are the
	concentration points of $u_{\varepsilon}^{(i)}$. Using Lemma~\ref{Lem5.3} we
	know that
	\[
	\frac{y_{\varepsilon}^{j(i)}-a_j}{\varepsilon}\to0
	\qquad\text{as }\varepsilon\to0,
	\]
	hence, up to a subsequence,
	\[
	\bar{u}_{\varepsilon}^{(i)}\to
	\sum_{j=1}^{k}U^{j} \quad\text{in }H^{s}_{\mathrm{loc}}(\R^{N}),
	\qquad
	\bar{u}_{\varepsilon}^{(2)}\to U^{i_0}
	\quad\text{in }H^{s}_{\mathrm{loc}}(\R^{N}),
	\]
	and, by \eqref{eq5.29},
	\[
	\bar{\varphi}_{\varepsilon}^{(i)}\to0
	\quad\text{in }H^{s}(\R^{N}),\qquad i=1,2.
	\]
	
	From \eqref{eq5.27}, \eqref{eq5.29} and the above convergence,
	we deduce
	\begin{equation}\label{eq5.32}
		\int_{\R^{N}}\big|(-\Delta)^{\frac{s}{2}}\bar{u}_{\varepsilon}^{(1)}\big|^{2}\,dx
		=\sum_{j=1}^{k}\int_{\R^{N}}\big|(-\Delta)^{\frac{s}{2}}U^{j}\big|^{2}\,dx
		+o(1),
	\end{equation}
	and
	\begin{equation}\label{eq5.33}
		\int_{\R^{N}}(-\Delta)^{\frac{s}{2}}\big(\bar{u}_{\varepsilon}^{(1)}+\bar{u}_{\varepsilon}^{(2)}-2U^{i_0}\big)
		\,(-\Delta)^{\frac{s}{2}}\bar{\xi}_{\varepsilon}\,dx\rightarrow0,
	\end{equation}
	as $\varepsilon\to0$.
	
	Moreover, by $(V_{1})$–$(V_{3})$ and Lemma~\ref{Lem5.3}, we have
	\[
	V\big(\varepsilon x+y_{\varepsilon}^{i_0(1)}\big)\to V(a_{i_0})
	\quad\text{uniformly on compact sets},
	\]
	and the nonlinearity satisfies
	\[
	C_{\varepsilon}\big(\varepsilon x+y_{\varepsilon}^{i_0(1)}\big)
	\to p\big(U^{i_0}(x)\big)^{p-1}
	\quad\text{in }L^{q}_{\mathrm{loc}}(\R^{N})
	\]
	for any $q\in[1,\infty)$, thanks to the bounds on
	$\varphi_{\varepsilon}^{(i)}$ and the convergence of $\bar{u}_{\varepsilon}^{(i)}$.
	
	Combining \eqref{eq5.31}–\eqref{eq5.33} and using the convergence
	$\bar{\xi}_{\varepsilon}\rightharpoonup\bar{\xi}$ in $H^{s}(\R^{N})$,
	we pass to the limit in the weak formulation and obtain that $\bar{\xi}$
	solves
	\begin{equation}\label{eq5.31-limit}
		\begin{aligned}
			&\Big(a + b\sum_{j=1}^{k}\int_{\R^{N}}
			\big|(-\Delta)^{\frac{s}{2}}U^{j}\big|^{2}\,dx\Big)(-\Delta)^{s}\bar{\xi}
			+2b\Big(\int_{\R^{N}}(-\Delta)^{\frac{s}{2}}U^{i_0}\,
			(-\Delta)^{\frac{s}{2}}\bar{\xi}\,dx\Big)(-\Delta)^{s}U^{i_0} \\
			&\qquad\qquad
			+V(a_{i_0})\,\bar{\xi}
			=p\big(U^{i_0}\big)^{p-1}\bar{\xi}
			\quad\text{in }H^{-s}(\R^{N}).
		\end{aligned}
	\end{equation}
	That is,
	\[
	\mathcal{L}_{+}\bar{\xi}=0,
	\]
	where $\mathcal{L}_{+}$ is exactly the linearized operator introduced in
	Proposition~\ref{Pro1.1}. By Proposition~\ref{Pro1.1}, the kernel of
	$\mathcal{L}_{+}$ is spanned by the derivatives of $U^{i_0}$, hence there
	exist constants $d_{1},\dots,d_{N}\in\R$ such that
	\[
	\bar{\xi}(x)
	=\sum_{i=1}^{N}d_i\,\partial_{x_i}U^{i_0}(x).
	\]
	
	Finally, by the bounds \eqref{eq5.27} and the equation \eqref{eq5.31},
	standard interior regularity for fractional Kirchhoff-type equations yields
	a uniform bound of $(\bar{\xi}_{\varepsilon})$ in $C^{1,\beta}_{\mathrm{loc}}(\R^{N})$
	for some $\beta\in(0,1)$. Therefore the weak convergence in $H^{s}$ can be
	upgraded to
	\[
	\bar{\xi}_{\varepsilon}\to\bar{\xi}
	\quad\text{in }C^{1}_{\mathrm{loc}}(\R^{N}),
	\]
	up to a subsequence. This completes the proof.
\end{proof}
Now we prove \eqref{eq5.22} by showing the following lemma.

\begin{Lem}\label{Lem5.6}
	Let $d_{i}$ be defined as in Lemma \ref{Lem5.5}. Then
	\[
	d_{i}=0 \quad \text{for all }i=1,2,\dots,N.
	\]
\end{Lem}

\begin{proof}
	We use the Pohoz\v{a}ev-type identity \eqref{eq3.1}. Apply \eqref{eq3.1} to
	$u_{\varepsilon}^{(1)}$ and $u_{\varepsilon}^{(2)}$ with
	\(\Omega=B_{d}(y_{\varepsilon}^{i_0(1)})\), where $d\in(1,2)$ is chosen as in
	Lemma \ref{Lem5.5}. Subtracting the two identities we obtain, for each
	fixed $j\in\{1,\dots,N\}$,
	\begin{equation}\label{eq5.34}
		\begin{aligned}
			\int_{B_{d}(y_{\varepsilon}^{i_0(1)})}
			\frac{\partial V}{\partial x_{j}}\bigl((u_{\varepsilon}^{(1)})^{2}
			-(u_{\varepsilon}^{(2)})^{2}\bigr)\,dx
			=\sum_{\ell=1}^{4} I_{\ell},
		\end{aligned}
	\end{equation}
	where, using
	\[
	\xi_{\varepsilon}
	=\frac{u_{\varepsilon}^{(1)}-u_{\varepsilon}^{(2)}}
	{\|u_{\varepsilon}^{(1)}-u_{\varepsilon}^{(2)}\|_{L^{\infty}(\R^{N})}},
	\qquad |\,\xi_{\varepsilon}|\le 1,
	\]
	we can rewrite the right-hand side as
	\[
	\int_{B_{d}(y_{\varepsilon}^{i_0(1)})}
	\frac{\partial V}{\partial x_{j}}\bigl(u_{\varepsilon}^{(1)}
	+u_{\varepsilon}^{(2)}\bigr)\xi_{\varepsilon}\,dx
	=\sum_{\ell=1}^{4} I_{\ell},
	\]
	with
	\begin{align*}
		I_{1}&=
		\bigl(\varepsilon^{2s} a+\varepsilon^{4s-N} b
		\int_{\R^{N}}|(-\Delta)^{\frac{s}{2}} u_{\varepsilon}^{(1)}|^{2}\,dx\bigr)
		\int_{\partial B_{d}(y_{\varepsilon}^{i_0(1)})}
		\bigl(((-\Delta)^{\frac{s}{2}} u_{\varepsilon}^{(1)}
		+(-\Delta)^{\frac{s}{2}} u_{\varepsilon}^{(2)})
		\cdot(-\Delta)^{\frac{s}{2}} \xi_{\varepsilon}\bigr)\nu_{j}\,d\sigma,\\[0.2cm]
		I_{2}&=
		-2\bigl(\varepsilon^{2s} a+\varepsilon^{4s-N} b
		\int_{\R^{N}}|(-\Delta)^{\frac{s}{2}} u_{\varepsilon}^{(1)}|^{2}\,dx\bigr)
		\int_{\partial B_{d}(y_{\varepsilon}^{i_0(1)})}
		\Bigl(\frac{\partial u_{\varepsilon}^{(1)}}{\partial \nu}
		\frac{\partial u_{\varepsilon}^{(1)}}{\partial x_{j}}
		-\frac{\partial u_{\varepsilon}^{(2)}}{\partial \nu}
		\frac{\partial u_{\varepsilon}^{(2)}}{\partial x_{j}}\Bigr)\,d\sigma,\\[0.2cm]
		I_{3}&=
		\varepsilon^{4s-N} b
		\int_{\partial B_{d}(y_{\varepsilon}^{i_0(1)})}
		\Bigl(|(-\Delta)^{\frac{s}{2}} u_{\varepsilon}^{(2)}|^{2}\nu_{j}
		-2\frac{\partial u_{\varepsilon}^{(2)}}{\partial \nu}
		\frac{\partial u_{\varepsilon}^{(2)}}{\partial x_{j}}\Bigr)d\sigma
		\int_{\R^{N}}
		(-\Delta)^{\frac{s}{2}}(u_{\varepsilon}^{(1)}+u_{\varepsilon}^{(2)})
		(-\Delta)^{\frac{s}{2}}\xi_{\varepsilon}\,dx,\\[0.2cm]
		I_{4}&=
		\int_{\partial B_{d}(y_{\varepsilon}^{i_0(1)})}
		V(x)\bigl(u_{\varepsilon}^{(1)}+u_{\varepsilon}^{(2)}\bigr)\xi_{\varepsilon}\nu_{j}\,d\sigma
		-2\int_{\partial B_{d}(y_{\varepsilon}^{i_0(1)})}
		A_{\varepsilon}(x)\,\xi_{\varepsilon}\nu_{j}\,d\sigma,
	\end{align*}
	and
	\[
	A_{\varepsilon}(x)
	=\int_{0}^{1}\bigl(tu_{\varepsilon}^{(1)}(x)+(1-t)u_{\varepsilon}^{(2)}(x)\bigr)^{p}\,dt.
	\]
	
	\medskip\noindent
	\textbf{Step 1: Estimate of the right-hand side.}
	By Theorem~\ref{Thm2.1} and Lemma~\ref{Lem5.4},
	\[
	\varepsilon^{2s} a+\varepsilon^{4s-N} b
	\int_{\R^{N}}|(-\Delta)^{\frac{s}{2}} u_{\varepsilon}^{(i)}|^{2}\,dx
	=O(\varepsilon^{2s}),\quad i=1,2,
	\]
	and by the construction of $u_{\varepsilon}^{(i)}$ and Lemmas~\ref{Lem5.1}–\ref{Lem5.3}
	we have
	\[
	\int_{\partial B_{d}(y_{\varepsilon}^{i_0(1)})}
	\bigl(|(-\Delta)^{\frac{s}{2}} u_{\varepsilon}^{(i)}|^{2}
	+|\nabla u_{\varepsilon}^{(i)}|^{2}\bigr)\,d\sigma
	=O\Big(\big\|(-\Delta)^{\frac{s}{2}}\varphi_{\varepsilon}^{(i)}\big\|_{2}^{2}\Big)
	=O\big(\varepsilon^{N+2m(1-\tau)}\big),
	\]
	for $i=1,2$. Using also Lemma~\ref{Lem5.4},
	\[
	\int_{\R^{N}}\big|(-\Delta)^{\frac{s}{2}}\xi_{\varepsilon}\big|^{2}\,dx
	\le C\,\varepsilon^{N-2s},
	\]
	we obtain
	\[
	I_{1}+I_{2}+I_{3}
	=O\big(\varepsilon^{N+2m(1-\tau)}\big).
	\]
	Moreover, since $|\,\xi_{\varepsilon}|\le1$ and $u_{\varepsilon}^{(i)}$ are uniformly
	bounded in $H^{s}(\R^{N})$ (hence in $L^{2^{*}_{s}}(\R^{N})$), again by
	H\"older's inequality and Lemma~\ref{Lem5.3} we have
	\[
	\int_{\partial B_{d}(y_{\varepsilon}^{i_0(1)})}
	V(x)\bigl(u_{\varepsilon}^{(1)}+u_{\varepsilon}^{(2)}\bigr)\xi_{\varepsilon}\nu_{j}\,d\sigma
	=O\big(\varepsilon^{N+m(1-\tau)}\big),
	\]
	and similarly, using the growth of $A_{\varepsilon}(x)$,
	\[
	\int_{\partial B_{d}(y_{\varepsilon}^{i_0(1)})}
	A_{\varepsilon}(x)\,\xi_{\varepsilon}\nu_{j}\,d\sigma
	=O\big(\varepsilon^{(N+m(1-\tau))p}\big).
	\]
	Therefore,
	\begin{equation}\label{eq5.35}
		\text{RHS of }\eqref{eq5.34}
		=O\big(\varepsilon^{N+m(1-\tau)}\big).
	\end{equation}
	
	\medskip\noindent
	\textbf{Step 2: Asymptotics of the left-hand side.}
	Using $(V_{3})$ and the expansion near $a_{i_0}$, we write
	\[
	\frac{\partial V}{\partial x_{j}}(x)
	=m c_{i_0,j}|x_{j}-a_{i_0,j}|^{m-2}(x_{j}-a_{i_0,j})
	+O\big(|x_{j}-a_{i_0,j}|^{m}\big).
	\]
	Hence,
	\begin{equation*}
		\begin{aligned}
			&\int_{B_{d}(y_{\varepsilon}^{i_0(1)})}
			\frac{\partial V}{\partial x_{j}}\bigl(u_{\varepsilon}^{(1)}+u_{\varepsilon}^{(2)}\bigr)
			\xi_{\varepsilon}\,dx \\
			&=m c_{i_0,j}
			\int_{B_{d}(y_{\varepsilon}^{i_0(1)})}
			|x_{j}-a_{i_0,j}|^{m-2}(x_{j}-a_{i_0,j})
			\bigl(u_{\varepsilon}^{(1)}+u_{\varepsilon}^{(2)}\bigr)\xi_{\varepsilon}\,dx \\
			&\quad
			+O\Big(
			\int_{B_{d}(y_{\varepsilon}^{i_0(1)})}
			|x_{j}-a_{i_0,j}|^{m}\bigl(u_{\varepsilon}^{(1)}+u_{\varepsilon}^{(2)}\bigr)
			|\xi_{\varepsilon}|\,dx\Big).
		\end{aligned}
	\end{equation*}
	
	For the main term, we change variables
	\(x=\varepsilon z+y_{\varepsilon}^{i_0(1)}\) and use the decompositions
	\[
	u_{\varepsilon}^{(i)}(x)
	=\sum_{h=1}^{k}U^{h}\Big(\frac{x-y_{\varepsilon}^{h(i)}}{\varepsilon}\Big)
	+\varphi_{\varepsilon}^{(i)}(x),
	\qquad i=1,2,
	\]
	together with Lemma~\ref{Lem5.3} and Lemma~\ref{Lem5.5}. Since
	\(|y_{\varepsilon}^{h(i)}-a_{h}|=o(\varepsilon)\), we have
	\[
	\frac{y_{\varepsilon}^{h(i)}-a_{h}}{\varepsilon}\to0, \qquad
	\bar{\xi}_{\varepsilon}(z)
	=\xi_{\varepsilon}(\varepsilon z+y_{\varepsilon}^{i_0(1)})
	\to\sum_{i=1}^{N}d_{i}\,\partial_{x_{i}}U^{i_0}(z)
	\quad\text{in }C^{1}_{\mathrm{loc}}(\R^{N}).
	\]
	Moreover,
	\[
	|x_{j}-a_{i_0,j}|^{m-2}(x_{j}-a_{i_0,j})
	=\varepsilon^{m-1}\Big|z_{j}
	+\frac{y_{\varepsilon,j}^{i_0(1)}-a_{i_0,j}}{\varepsilon}\Big|^{m-2}
	\Big(z_{j}+\frac{y_{\varepsilon,j}^{i_0(1)}-a_{i_0,j}}{\varepsilon}\Big),
	\]
	so, using the decay of $U^{i_0}$ and the convergence of
	$\bar{\xi}_{\varepsilon}$, we obtain
	\begin{equation}\label{eq5.36}
		\begin{aligned}
			& m c_{i_0,j}
			\int_{B_{d}(y_{\varepsilon}^{i_0(1)})}
			|x_{j}-a_{i_0,j}|^{m-2}(x_{j}-a_{i_0,j})
			\bigl(u_{\varepsilon}^{(1)}+u_{\varepsilon}^{(2)}\bigr)\xi_{\varepsilon}\,dx \\
			&=m c_{i_0,j}\,\varepsilon^{N+m-1}
			\int_{B_{d/\varepsilon}(0)}
			\Big|z_{j}
			+\frac{y_{\varepsilon,j}^{i_0(1)}-a_{i_0,j}}{\varepsilon}\Big|^{m-2}
			\Big(z_{j}+\frac{y_{\varepsilon,j}^{i_0(1)}-a_{i_0,j}}{\varepsilon}\Big) \\
			&\qquad\qquad\qquad\qquad\quad
			\times\Big(U^{i_0}(z)+U^{i_0}\Big(z+\frac{y_{\varepsilon}^{i_0(1)}
				-y_{\varepsilon}^{i_0(2)}}{\varepsilon}\Big)\Big)
			\bar{\xi}_{\varepsilon}(z)\,dz + o\big(\varepsilon^{N+m-1}\big) \\
			&=2m c_{i_0,j}\,\varepsilon^{N+m-1}
			\sum_{i=1}^{N}d_{i}
			\int_{\R^{N}}|z_{j}|^{m-2}z_{j}\,U^{i_0}(z)\,\partial_{x_{i}}U^{i_0}(z)\,dz
			+o\big(\varepsilon^{N+m-1}\big).
		\end{aligned}
	\end{equation}
	Since $U^{i_0}$ is radial and strictly decreasing, the integrals vanish for
	$i\ne j$ and
	\[
	\int_{\R^{N}}|z_{j}|^{m-2}z_{j}\,U^{i_0}(z)\,\partial_{x_{j}}U^{i_0}(z)\,dz\neq0.
	\]
	Define
	\[
	D_{j}
	=2m c_{i_0,j}
	\int_{\R^{N}}|z_{j}|^{m-2}z_{j}\,U^{i_0}(z)\,\partial_{x_{j}}U^{i_0}(z)\,dz\neq0.
	\]
	Then \eqref{eq5.36} yields
	\begin{equation}\label{eq5.38}
		m c_{i_0,j}
		\int_{B_{d}(y_{\varepsilon}^{i_0(1)})}
		|x_{j}-a_{i_0,j}|^{m-2}(x_{j}-a_{i_0,j})
		\bigl(u_{\varepsilon}^{(1)}+u_{\varepsilon}^{(2)}\bigr)\xi_{\varepsilon}\,dx
		= D_{j}d_{j}\,\varepsilon^{N+m-1}
		+o\big(\varepsilon^{N+m-1}\big).
	\end{equation}
	
	For the error term coming from the $O(|x_{j}-a_{i_0,j}|^{m})$ part of
	$\partial_{x_{j}}V$, by H\"older's inequality, Lemma~\ref{Lem5.3} and
	Lemma~\ref{Lem5.4} we have
	\begin{equation}\label{eq5.39}
		\int_{B_{d}(y_{\varepsilon}^{i_0(1)})}
		|x_{j}-a_{i_0,j}|^{m}
		\bigl|u_{\varepsilon}^{(1)}+u_{\varepsilon}^{(2)}\bigr|
		|\xi_{\varepsilon}|\,dx
		=O\big(\varepsilon^{N+m}\big).
	\end{equation}
	Combining \eqref{eq5.38} and \eqref{eq5.39}, we obtain
	\begin{equation}\label{eq5.40}
		\text{LHS of }\eqref{eq5.34}
		=D_{j}d_{j}\,\varepsilon^{N+m-1}
		+o\big(\varepsilon^{N+m-1}\big).
	\end{equation}
	
	\medskip\noindent
	\textbf{Step 3: Conclusion.}
	From \eqref{eq5.35} and \eqref{eq5.40} it follows that
	\[
	D_{j}d_{j}\,\varepsilon^{N+m-1}
	+o\big(\varepsilon^{N+m-1}\big)
	=O\big(\varepsilon^{N+m(1-\tau)}\big).
	\]
	Recall that $m>1$ and $m\tau<1$, so we can choose $\tau>0$ such that
	\[
	m(1-\tau)>m-1.
	\]
	Then, as $\varepsilon\to0$,
	\[
	\varepsilon^{N+m(1-\tau)}
	=o\big(\varepsilon^{N+m-1}\big),
	\]
	and hence the above identity implies necessarily $D_{j}d_{j}=0$. Since
	$D_{j}\neq0$ by construction, we conclude that $d_{j}=0$. As $j$ was
	arbitrary in $\{1,\dots,N\}$, we obtain
	\[
	d_{i}=0 \quad \text{for all }i=1,2,\dots,N.
	\]
	The proof is complete.
\end{proof}

{\bf Proof of Theorem \ref{Thm1.2}:}
	Assume by contradiction that there exist two distinct solutions
	\(u_{\varepsilon}^{(i)}\), \(i=1,2\), constructed in Section~4.  
	Recall
	\[
	\xi_{\varepsilon}
	=\frac{u_{\varepsilon}^{(1)}-u_{\varepsilon}^{(2)}}
	{\|u_{\varepsilon}^{(1)}-u_{\varepsilon}^{(2)}\|_{L^{\infty}(\R^{N})}},
	\qquad
	\|\xi_{\varepsilon}\|_{L^{\infty}(\R^{N})}=1,
	\]
	and, for each \(i_0\in\{1,\dots,k\}\),
	\[
	\bar{\xi}_{\varepsilon,i_0}(x)
	=\xi_{\varepsilon}\bigl(\varepsilon x+y_{\varepsilon}^{i_0(1)}\bigr).
	\]
	
	By Lemmas~\ref{Lem5.5} and \ref{Lem5.6} we know that, for every fixed \(R>0\),
	\begin{equation}\label{eq:xi-local}
		\|\bar{\xi}_{\varepsilon,i_0}\|_{L^{\infty}(B_R(0))}\;\longrightarrow\;0
		\quad\text{as }\varepsilon\to0,
	\end{equation}
	and by \eqref{eq5.21} the convergence is uniform as \(|x|\to\infty\).
	Consequently, for each \(i_0\) and fixed \(R>0\),
	\[
	\|\xi_{\varepsilon}\|_{L^{\infty}\bigl(B_{R\varepsilon}(y_{\varepsilon}^{i_0(1)})\bigr)}
	=\|\bar{\xi}_{\varepsilon,i_0}\|_{L^{\infty}(B_R(0))}\;\longrightarrow\;0.
	\]
	Since the two families of concentration points satisfy
	\(|y_{\varepsilon}^{i(1)}-y_{\varepsilon}^{i(2)}|=o(\varepsilon)\), the above
	estimate also holds if we replace \(y_{\varepsilon}^{i_0(1)}\) by
	\(y_{\varepsilon}^{i_0(2)}\). Hence, for every fixed \(R>0\),
	\begin{equation}\label{eq:near-peaks}
		\|\xi_{\varepsilon}\|_{L^{\infty}\left(\bigcup_{i=1}^{k}
			B_{R\varepsilon}(y_{\varepsilon}^{i})\right)}\;\longrightarrow\;0
		\quad\text{as }\varepsilon\to0.
	\end{equation}
	
	On the other hand, in the exterior region
	\(\R^{N}\setminus\bigcup_{i=1}^{k}B_{R\varepsilon}(y_{\varepsilon}^{i})\),
	the function \(\xi_{\varepsilon}\) solves a linear equation with uniformly
	positive potential and small right-hand side. By the same barrier and
	maximum principle argument as in \cite{R-Yang2}, we obtain, for the same fixed \(R>0\),
	\begin{equation}\label{eq:far-region}
		\|\xi_{\varepsilon}\|_{L^{\infty}\left(
			\R^{N}\setminus\bigcup_{i=1}^{k}B_{R\varepsilon}(y_{\varepsilon}^{i})\right)}
		\rightarrow0
		\quad\text{as }\varepsilon\to0.
	\end{equation}
	
	Combining \eqref{eq:near-peaks} and \eqref{eq:far-region} yields
	\[
	\|\xi_{\varepsilon}\|_{L^{\infty}(\R^{N})}\;\rightarrow\;0
	\quad\text{as }\varepsilon\to0,
	\]
	which contradicts the normalization
	\(\|\xi_{\varepsilon}\|_{L^{\infty}(\R^{N})}=1\).  
	Therefore, for \(\varepsilon>0\) small enough, the semiclassical multi-peak
	bounded state obtained in Theorem~\ref{Thm1.1} is unique. The proof is complete.
	\qed

\section*{Acknowledgment}

We thank the anonymous referee for carefully reading our manuscript and for the valuable comments that helped improve it.
This work is supported by National Natural Science Foundation of China (12301145,12161007,12261107) and Yunnan Fundamental Research Projects (202301AU070144, 202401AU070123). 

\medskip
{\bf Data availability:}  Data sharing is not applicable to this article as no new data were created or analyzed in this study.

\medskip
{\bf Conflict of Interests:} The author declares that there is no conflict of interest.

\bibliographystyle{plain}
\bibliography{yang}

\end{document}